\documentclass{article}
\usepackage{amssymb}
\usepackage{amsmath}
\usepackage{amsthm}

\usepackage[usenames,dvipsnames]{pstricks}
 \usepackage{epsfig}
 \usepackage{xcolor}

\newtheorem{thm}{Theorem}[section]
\newtheorem{lem}[thm]{Lemma}
\newtheorem{proposition}[thm]{\bf Proposition}%
\newtheorem{corollary}[thm]{\bf Corollary}%
\newtheorem{problem}[thm]{\bf Problem}%
\newtheorem{definition}[thm]{\bf Definition}{}%
\newcommand{\be}{\begin{equation}}
\newcommand{\ee}{\end{equation}}
\newcommand{\ol}{\overline}

\newcommand{\ul}{\underline}

\newtheorem{MIex}[thm]{Example}
\newenvironment{example}    {\begin{MIex}\em}    {\end{MIex}}

\newtheorem{MIrem}[thm]{Remark}
\newenvironment{remark}    {\begin{MIrem}\em}    {\end{MIrem}}

\title{Polynomial graph invariants from \\ homomorphism numbers}
\author{Delia Garijo ${}^1$ \and
              Andrew J. Goodall ${}^2$ \and
              Jaroslav Ne\v{s}et\v{r}il ${}^2$}
\begin{document}

\maketitle

\footnotetext[1]
{University of Seville, Seville, Spain. Partially supported by JA-FQM164. Email address: \tt{dgarijo@us.es}.}
\footnotetext[2]
{Charles University, Prague, Czech Republic. Supported by the Center of Excellence-Inst for Theor. Comp. Sci., Prague, P202/12/G061, and by Project ERCCZ LL1201 Cores. Email addresses: \tt{andrew@iuuk.mff.cuni.cz, nesetril@kam.mff.cuni.cz}}

\begin{abstract}
We give a method of generating strongly polynomial sequences of graphs, i.e., sequences $(H_{\mathbf{k}})$ indexed by a multivariate parameter $\mathbf{k}=(k_1,\ldots, k_h)$ such that, for each fixed graph $G$, there is a multivariate polynomial $p(G;x_1,\ldots, x_h)$ such that the number of homomorphisms from $G$ to $H_{\mathbf{k}}$ is given by the evaluation $p(G;k_1,\ldots, k_h)$. Our construction produces a large family of graph polynomials that includes the Tutte polynomial, the Averbouch-Godlin-Makowsky polynomial and the Tittmann-Averbouch-Makowsky polynomial.
We also introduce a new graph parameter, the {\em branching core size} of a simple graph, related to how many involutive automorphisms with fixed points it has. We prove that a countable family of graphs of bounded branching core size (which in particular implies bounded tree-depth) can always be partitioned into a finite number of strongly polynomial subsequences.
\end{abstract}

\section{Introduction}
\subsection{Motivation}
Let $\mathbb{N}^h$ denote the set of $h$-tuples of positive integers ($h\geq 1$), and let $\mathcal{H}$ be a countably infinite set of graphs possibly with loops and/or weights on edges. Suppose  $\mathcal{H}$ is presented as a sequence $(H_\mathbf{k})$ indexed by tuples $\mathbf{k}=(k_1,\ldots, k_h)\in \mathbb{N}^h$.
Countable families of graphs are often given in the form of such a sequence, for example, the complete graphs $(K_k)$, or bipartite graphs $(K_{k_1,k_2})$. In these and other concrete examples the indices $k_1,\ldots, k_h$ correspond to some natural graph parameter (such as number of vertices) or indicate how to construct the graph in this position of the sequence (such as substituting every vertex of a base graph $H$ by $k$ twin copies to obtain $H_k$). 

In this paper, we shall be interested in the number of homomorphisms of a graph $G$ to $H_\mathbf{k}$, denoted by hom$(G, H_\mathbf{k})$, as a function of $\mathbf{k}$ and  $G$. More specifically, when is this function a multivariate polynomial in $\mathbf{k}$ for every graph $G$? We say the sequence $(H_\mathbf{k})$ is {\em strongly polynomial} if this is the case for all $\mathbf{k}$, and {\em polynomial} if for each graph $G$ there is a finite number of multivariate polynomials such that ${\rm hom}(G,H_\mathbf{k})$ is the evaluation of one of them at $\mathbf{k}$.
A well known example is the sequence of cliques $(K_k)$, which is  strongly polynomial  since  ${\rm hom}(G,K_k)$ is the value of the chromatic polynomial of $G$ at $k$. The sequence of paths $(P_k)$ provide an example of a polynomial sequence, since ${\rm hom}(G,P_k)$ is polynomial in $k$ for $k>{\rm diam}(G)$, and it is not a strongly polynomial sequence~\cite{dlHJ95}.

De la Harpe and Jaeger \cite{dlHJ95} gave necessary and sufficient criteria that enabled them to prove that ${\rm hom}(G,H_k)$ is polynomial in $k$ for a number of graph sequences $(H_k)$ with $k\in \mathbb{N}$. They also gave a general method of generating such sequences of graphs by using the operation of graph composition (described in Section~\ref{sec:ornamentation} below). On the other hand, in our paper \cite{GGN11} we established precisely for which edge-weighted graphs $H$ homomorphism functions from multigraphs $G$ to $H$ are specializations of the Tutte polynomial $T(G;x,y)$, the Averbouch-Godlin-Makowsky polynomial $\xi_G(x,y,z)$ \cite{AGM}, and the Tittmann-Averbouch-Makowsky polynomial $Q_G(x,y)$ \cite{TAM09}. The edge-weighted graphs $H$ obtained for the three polynomials take the form of a sequence of graphs $(H_{\mathbf{k}})$ indexed by a multivariate parameter $\mathbf{k}$ (shown in Figure~\ref{fig:ornamented_trees} later in this paper).
This motivates the problem of determining in general which sequences of graphs $(H_{\mathbf{k}})$ have the property that hom$(G,H_{\mathbf{k}})$ is a multivariate polynomial in $\mathbf{k}$ for all graphs $G$.

\subsection{Outline and main results}
In Section~\ref{sec:def} we formally define homomorphism numbers, and polynomial and strongly polynomial sequences of graphs. We begin in Section~\ref{sec:complementation} with some useful operations on a given (strongly) polynomial sequence $(H_{\mathbf{k}})$, indexed by ${\mathbf{k}}\in\mathbb{N}^h$, that preserve the property of being (strongly) polynomial. In Section~\ref{sec:ornamentation} we move to operations on (strongly) polynomial sequences $(F_{\mathbf{k}})$ and $(H_{\mathbf{k}})$ that produce another (strongly) polynomial sequence, including lexicographic products (Proposition~\ref{lexicographic}). We also describe the method of composition of ornamented graphs~\cite[Ex. B.7]{dlHJ95},
which can be used to produce a rich class of polynomial sequences from basic building blocks. In Section~\ref{examples} we illustrate this with some examples.

In Section~\ref{sec:treesbranching}  we introduce our new method to generate a polynomial sequence $(H_{\mathbf{k}})$
from a starting graph $H$. The idea is to view $H$ as a subgraph of the closure of a rooted tree and then apply a tree operation we shall call ``branching'' with multiplicity $k_i$, for some $1\leq i\leq h$, at vertices of this rooted tree. This generates the graph $H_{\mathbf{k}}$ indexed by $\mathbf{k}=(k_1,\ldots, k_h)$ and we prove that ${\rm hom}(G,H_{\mathbf{k}})$ is a polynomial in ${\mathbf{k}}$ for every graph $G$ for $\mathbf{k}$ sufficiently large (a special case of Theorem~\ref{thm:poly_branching}). 
Branching is also defined for ornamented graphs, and the full strength of Theorem~\ref{thm:poly_branching} consists in combining branching with composition of ornamented graphs. We thereby obtain a large class of graph polynomials that include the Tutte polynomial, the Averbouch-Godlin-Makowsky polynomial, and the Tittmann-Averbouch-Makowsky polynomial (see Figure~\ref{fig:ornamented_trees} in Section~\ref{sec:branching}). 

The branching operation produces bilateral symmetries (swap isomorphic subtrees) and in order to state our main result in Section~\ref{sec:bushiness} we introduce a new graph parameter, namely {\em branching core size}.
Our main result (Theorem~\ref{thm:bounded_bushiness}) is that given any countable family $\mathcal{H}$ of graphs of bounded branching core size there is a partition of $\mathcal{H}$ into subsets such that each subset can be indexed by a multivariate parameter $\mathbf{k}=(k_1,\ldots, k_h)$ so that it forms a subsequence of a strongly polynomial sequence of graphs, and moreover these sequences are produced by the branching operation. This result extends to the composition of ornamented graphs of bounded branching core size when there is a finite number of distinct ornaments used among all the graphs in $\mathcal{H}$.
On the other hand, we give some examples of families of graphs of unbounded branching core size (but bounded tree-depth) and yet which can also be partitioned into a finite number of strongly polynomial subsequences. 


\section{Definitions}\label{sec:def}

\subsection{Homomorphism numbers}\label{sec:hom}

Let hom$(G,H)$ denote the number of homomorphisms from a graph $G$ to a simple graph $H$, i.e., adjacency preserving maps from $V(G)$ to $V(H)$. This parameter can be extended to weighted graphs: let $H$ be a weighted graph given by its adjacency matrix $(a_{i,j})$, where $a_{i,j}$ is the weight of the edge $ij$. Then, for a multigraph $G$, the homomorphism function hom$(G,H)$ is defined by
$${\rm hom}(G,H)=\sum_{f:V(G)\rightarrow V(H)}\;\prod_{uv\in E(G)}a_{f(u),f(v)},$$
where the sum is over all functions from $V(G)$ to $V(H)$ and edges of $G$ are taken with multiplicity in the product.
When $a_{i,j}\in\{0,1\}$ this coincides with the number of homomorphisms from $G$ to $H$ as previously defined. When $a_{i,j}\in\mathbb{N}$, the graph $H$ is a multigraph and ${\rm hom}(G,H)$ counts the number of homomorphisms from $G$ to $H$ again, where now a homomorphism needs to be defined rather in terms of a pair of maps $f_0:V(G)\rightarrow V(H)$, $f_1:E(G)\rightarrow E(H)$, the defining property being that $f_1(uv)$ has endpoints $f_0(u)$ and $f_0(v)$ for every edge $uv\in E(G)$.

\begin{example}\label{ex:hom}
Let $K_k^\ell$ denote the complete graph on $k$ vertices with a loop of weight $\ell$ attached to each vertex.
Then
\begin{align*}{\rm hom}(G,K_k^\ell) & =\sum_{f:V(G)\rightarrow V(K_k^\ell)} \ell^{\#\{uv\in E(G): f(u)=f(v)\}}\\
& = k^{c(G)}(\ell-1)^{r(G)}T\left(G;\frac{\ell-1+k}{\ell-1},\ell\right),\end{align*}
where $c(G)$ is the number of connected components of $G$, $r(G)$ the rank of $G$, and $T(G;x,y)$ the Tutte polynomial of $G$.
This includes the special cases
$${\rm hom}(G,K_k)=P(G;k),$$
where $P(G;k)$ is the number of proper vertex $k$-colourings of $G$, and
$${\rm hom}(G,K_k^{1-k})=(-1)^{|E(G)|}k^{|V(G)|}F(G;k),$$
where $F(G;k)$ is the number of nowhere-zero $A$-flows of $G$ for $A$ a finite Abelian group of order $k$ (such as the cyclic group $\mathbb{Z}_k$).
\end{example}

Homomorphisms to a fixed graph $H$ are often called {\em $H$-colourings}: the vertices of $H$ being the colours and its adjacencies determining which colours are allowed to be next to each other in a colouring of the graph~$G$.

\subsection{Polynomial sequences of graphs}\label{sec:polyseq}


We now introduce the principal notions of this paper. Definition \ref{def:poly} below concerns sequences that are indexed by {\em all} tuples in $\mathbb{N}^h$, for some fixed $h\geq 1$. These sequences are our principal object of study. Definition \ref{def:poly-ext} then extends the first definition to sequences indexed by some proper subset $I\subset\mathbb{N}^h$.

\begin{definition}\label{def:poly}
Let $(H_\mathbf{k})$, $\mathbf{k}=(k_1,\ldots, k_h)\in\mathbb{N}^h$, be a sequence of graphs.

\begin{itemize}
\item[(i)] We say $(H_\mathbf{k})$ is a {\em polynomial sequence} if for every graph $G$ there exists a finite set of polynomials $p_1(G;x_1,\ldots, x_h), \ldots,$ $p_m(G;x_1,\ldots, x_h)$ such that for every $\mathbf{k}=(k_1,\ldots, k_h)\in \mathbb{N}^h$ we have  $${\rm hom}(G,H_\mathbf{k})=p_\ell(G;k_1,\ldots, k_h)\hspace{0.2cm}\mbox{\rm  for some $1\leq \ell\leq m$.}$$
In other words, for each graph $G$ there exists a partition of the index set $\mathbb{N}^h$ into subsets $I_1(G), \ldots, I_m(G)$ such that ${\rm hom}(G,H_\mathbf{k})=p_\ell(G;k_1,\ldots, k_h)$ for $\mathbf{k}=(k_1,\ldots, k_h)\in I_\ell(G)$.

\item[(ii)] We say $(H_\mathbf{k})$ is a {\em strongly polynomial sequence} if for every graph $G$ there exists a single polynomial $p(G;x_1,\ldots, x_h)$ such that ${\rm hom}(G,H_\mathbf{k})=p(G;k_1,\ldots, k_h)$ for every $\mathbf{k}=(k_1,\ldots, k_h)\in \mathbb{N}^h$.
\end{itemize}
\end{definition}

To simplify notation, henceforth we shall write $p_\ell(G;k_1,\ldots, k_h)=p_\ell(G;\mathbf{k})$, $p(G;k_1,\ldots, k_h)=p(G;\mathbf{k})$ etc.

\begin{remark} Our definition of a polynomial sequence includes as a special case that of~\cite{dlHJ95}, where polynomial sequences are defined only for $h=1$, a sequence $(H_k)$ there being defined as polynomial if ${\rm hom}(G, H_k)$ is the evaluation of a univariate polynomimal $p(G)$ for sufficiently large $k$.  In this situation, for every graph $G$ all but one of the sets $I_1(G), \ldots, I_m(G)$ (in Definition~\ref{def:poly}) is finite, i.e.,  for each graph $G$ there is a polynomial $p(G;x)$ such that ${\rm hom}(G,H_k)=p(G;k)$ for all but finitely many $k\in\mathbb{N}$, these exceptions depending on $G$.
The class of univariate polynomial sequences as defined in Definition~\ref{def:poly} is therefore larger than these ``eventually polynomial'' sequences, but we shall see in due course that our more general class of polynomial sequences shares all the key properties of this subclass.

However, it ought to be noted that there is another natural extension of de la Harpe and Jaeger's definition of a polynomial sequence of graphs to the multivariate case. This is to consider a sequence $(H_{\mathbf{k}})$ indexed by $\mathbf{k}=(k_1,\ldots, k_h)$ to be polynomial if for each graph $G$ there is a polynomial $p(G)$ and $\mathbf{m}=\mathbf{m}(G)=(m_1,\ldots, m_h)\in\mathbb{N}^h$ such that ${\rm hom}(G,H_{\mathbf{k}})=p(G;\mathbf{k})$ when $k_i\geq m_i$ for each $1\leq i\leq h$ (i.e., ``for sufficiently large $\mathbf{k}$''). In fact, it would be easy to modify the statement of our results in accordance with such a definition, the difference in definition taken not affecting the validity of the results. 

\end{remark}

Of course, a given parametrization of a family $\mathcal{H}$ of graphs may fail to give a polynomial sequence $(H_{\mathbf{k}})$. However, we are looking from the positive side: we seek instances  $\mathcal{H}$ where there exists a parametrization of $\mathcal{H}$ by $\mathbf{k}\in\mathbb{N}^h$ for some $h\geq 1$ which does yield a polynomial sequence.

\begin{definition}\label{def:poly-ext}
Let $(F_\mathbf{k})$, $\mathbf{k}\in I\subseteq\mathbb{N}^h$, be a sequence of graphs indexed by a subset $I$ of $\mathbb{N}^h$.

We say $(F_\mathbf{k})$ is a {\em polynomial subsequence} if it is a subsequence of a polynomial sequence, i.e., there is a polynomial sequence $(H_\mathbf{k})$,  $\mathbf{k}\in\mathbb{N}^h$, such that $H_\mathbf{k}=F_\mathbf{k}$ when $\mathbf{k}\in I$.

Similarly, we say $(F_\mathbf{k})$ is a {\em strongly polynomial subsequence} if it is a subsequence of a strongly polynomial sequence.
\end{definition}

\begin{example}
The sequence $(K_{2k})$, indexed by $2k$, is a strongly polynomial subsequence, as it is a subsequence of $(K_k)$, indexed by $k$, which is strongly polynomial (see Example~\ref{ex:hom}). If we instead consider $(K_{2k})$ as being indexed by $k\in\mathbb{N}$ then, given in this form, this polynomial sequence is no longer a subsequence of $(K_k)$. (It requires reindexing by $2k$ in order for this to become the case.) 
\end{example}

Clearly, the property of being a (strongly) polynomial subsequence is unaffected by removing any number of its terms, and the property of being a polynomial subsequence is unaffected by changing a finite number of terms or introducing a finite number of extra terms.

 \begin{remark}\label{rmk:subseq} Suppose $(F_\mathbf{k})$ is a sequence indexed by $I\subset \mathbb{N}^h$ such that for each graph $G$ there exists a partition of $I$ into subsets $I_1(G), \ldots, I_m(G)$ such that ${\rm hom}(G,F_\mathbf{k})=p_\ell(G;\mathbf{k})$ for all $\mathbf{k}\in I_\ell(G)$, $1\leq\ell\leq m$. It may be that $(F_\mathbf{k})$ is not a polynomial subsequence, i.e., there is no polynomial sequence $(H_\mathbf{k})$, $\mathbf{k}\in\mathbb{N}^h$, with $H_\mathbf{k}=F_\mathbf{k}$ for $\mathbf{k}\in I$.
 (A candidate for such a sequence may be the sequence of hypercubes -- see Example~\ref{ex:cubes} below.)
 \end{remark}


\section{New sequences from old}\label{sec:constructs}

In this section we review and extend some constructions given in~\cite{dlHJ95} for generating new polynomial sequences from old, and give a number of concrete examples of polynomial sequences that can be created by these constructions.

\subsection{Complementation and line graph}\label{sec:complementation}

The {\em complement} $\overline{H}$ of a simple graph $H$ has $V(\overline{H})=V(H)$ and $E(\overline{H})=\binom{V(H)}{2}\setminus E(H)$. The following is a straightforward translation of \cite[Prop.1]{dlHJ95} to accommodate our more general notion of a polynomial sequence of graphs:
\begin{proposition}\label{complement}
If $(H_{\mathbf{k}})$ is a (strongly) polynomial sequence of simple graphs indexed by ${\mathbf{k}}\in\mathbb{N}^h$, then the sequence of complements $(\overline{H_{\mathbf{k}}})$ is (strongly) polynomial.
\end{proposition}
\begin{proof} We use the identity
\be\label{eq:compl}{\rm hom}(G,\overline{H_{\mathbf{k}}})=\sum_{\stackrel{C\cup D\subseteq E(G)}{C\cap D=\emptyset}}(-1)^{|E(G)|-|D|}{\rm hom}(G/C-D,H_{\mathbf{k}}),\ee
where $G/C-D$ is the graph obtained from $G$ by contracting the edges of $C$ and by deleting the edges of $D$.

Using the notation of Definition~\ref{def:poly}, for each graph $G/C-D$ there is a partition of $\mathbb{N}^h$ into subsets $I_{\ell}(G/C-D)$, $1\leq\ell\leq m(G/C-D)$, such that ${\rm hom}(G/C-D,H_{\mathbf{k}})=p_\ell(G/C-D;{\mathbf{k}})$ when ${\mathbf{k}}\in I_{\ell}(G/C-D)$. (In the case where $(H_{\mathbf{k}})$ is a strongly polynomial sequence we have $m(G/C-D)=1$ for all choices of $G$ and $C\cup D\subseteq E(G)$.) The greatest common refinement of the partitions $\{I_\ell(G/C-D):1\leq \ell\leq m(G/C-D)\}$ of $\mathbb{N}^h$ for $C\cup D\subseteq E(G)$ is another partition of $\mathbb{N}^h$ into a finite number of subsets (dependent on $G$), on each of which ${\rm hom}(G,\overline{H_{\mathbf{k}}})$ is a polynomial in ${\mathbf{k}}$.
\end{proof}

A minor modification in the proof of Proposition~\ref{complement} yields the following variant, where now we consider the {\em looped complement} of a graph $H$, i.e., the graph $\widetilde{H}$ with the same vertex set but whose edge set is $V(H)\times V(H)\setminus E(H)$. For example, the looped complement of the complete graph $K_k$, $k\in\mathbb{N}$, consists of $k$ vertices each with a loop: $\widetilde{K_k}=kK_1^1$.

\begin{proposition}\label{lcomplement}
If $(H_{\mathbf{k}})$ is a (strongly) polynomial sequence of simple graphs indexed by ${\mathbf{k}}\in\mathbb{N}^h$, then the sequence of looped complements $(\widetilde{H_{\mathbf{k}}})$ is  (strongly) polynomial.
\end{proposition}
\begin{proof}
By inclusion-exclusion (on the property that a given set of edges $D$ of $G$ are sent to edges of $H_{\mathbf{k}}$ and hence to non-edges of $\overline{H_{\mathbf{k}}}$) we have
\be\label{eq:loop_compl}{\rm hom}(G,\widetilde{H_{\mathbf{k}}})=\sum_{D\subseteq E(G)}(-1)^{|E(G)|-|D|}{\rm hom}(G-D,H_{\mathbf{k}}),\ee
where the sum is over spanning subgraphs of $G$, i.e., $G-D$ for $D\subseteq E(G)$ (rather than over minors of $G$ as in Proposition~\ref{complement}).
\end{proof}

We obtain as a consequence:

\begin{corollary}\label{remove_loops}
If $(H_{\mathbf{k}})$ is a (strongly) polynomial sequence of graphs indexed by ${\mathbf{k}}\in\mathbb{N}^h$, each graph with a single loop on each vertex, then the sequence of graphs obtained by removing all loops from the graphs $H_{\mathbf{k}}$ is (strongly) polynomial.
\end{corollary}
\begin{proof}
Let $(H_{\mathbf{k}})$ be a polynomial sequence of graphs with a loop on each vertex. By Proposition~\ref{lcomplement} the sequence $(\widetilde{H_{\mathbf{k}}})$ is also polynomial and by Proposition~\ref{complement} so is the sequence $(\overline{\widetilde{H_{\mathbf{k}}}})$. This sequence is $(H_{\mathbf{k}})$ with loops removed.
\end{proof}

Propositions~\ref{complement} and~\ref{lcomplement} are special cases of the following weighted and combined version:

\begin{proposition}\label{poly_seq_preserve}
Suppose that $(H_{\mathbf{k}})$, ${\mathbf{k}}\in\mathbb{N}^h$, is a (strongly) polynomial sequence of edge-weighted graphs, and that $H_{\mathbf{k}}$ has (weighted) adjacency matrix $(a_{i,j}^{\mathbf{k}})$. For given constants $\alpha,\beta,\alpha',\beta'$, let $(\widetilde{H}_{\mathbf{k}})$ be the sequence of edge-weighted graphs with $V(\widetilde{H}_{\mathbf{k}})=V(H_{\mathbf{k}})$,  weight $\alpha+\beta a_{i,j}^{\mathbf{k}}$ on non-loop edges ($i\neq j$) and weight $\alpha'+\beta' a_{i,i}^{\mathbf{k}}$ on loop edges.  Then  the sequence $(\widetilde{H}_{\mathbf{k}})$ is  (strongly) polynomial.
\end{proposition}
\begin{proof}
Let $\delta_{i,j}=1$ if $i=j$ and $\delta_{i,j}=0$ if $i\neq j$. Set $a=\alpha$, $b=\alpha-\alpha'$, $c=-\beta$ and $d=\beta-\beta'$.
Then the graph $\widetilde{H}_{\mathbf{k}}$ with adjacency matrix $$(a-b\delta_{i,j}-ca_{i,j}^{\mathbf{k}}-d\delta_{i,j}a_{i,j}^{\mathbf{k}})$$ gives weight $a-ca_{i,j}^{\mathbf{k}}$ to a non-loop edge $ij$
and weight $a-b-(c+d)a_{i,i}^{\mathbf{k}}$ to a loop on $i$. For a graph $G=(V,E)$ we have

{\small{\begin{align*} & {\rm hom}(G,\widetilde{H}_{\mathbf{k}})  =\sum_{f:V\rightarrow V(H_{\mathbf{k}})}\;\prod_{uv\in E}(a-b\delta_{f(u),f(v)}-ca_{f(u),f(v)}^{\mathbf{k}}-d\delta_{f(u),f(v)}a_{f(u),f(v)}^{\mathbf{k}})\\ \noalign{\bigskip}
& = \sum_{f:V\rightarrow V(H_{\mathbf{k}})}\; \sum_{A\subseteq E}a^{|A|}(-1)^{|E|-|A|}\prod_{uv\in E\setminus A}(b\delta_{f(u),f(v)}+ca_{f(u),f(v)}^{\mathbf{k}}+d\delta_{f(u),f(v)}a_{f(u),f(v)}^{\mathbf{k}})\\
\noalign{\bigskip}
& = \sum_{f:V\rightarrow V(H_{\mathbf{k}})}\; \sum_{A\subseteq E}a^{|A|}(-1)^{|E|\!-\!|A|}\sum_{B\subseteq E\setminus A}\;\prod_{uv\in B}(b\!+\!(c\!+\!d)a_{f(u),f(v)}^{\mathbf{k}})\prod_{uv\in E\setminus (A\cup B)}ca_{f(u),f(v)}^{\mathbf{k}}\\ \noalign{\bigskip}
& = \sum_{A\subseteq E}a^{|A|}(-1)^{|E|\!-\!|A|}\!\sum_{B\subseteq E\!\setminus\! A}\sum_{C\subseteq B}b^{|B|\!-\!|C|}(c\!+\!d)^{|C|}c^{|E|\!-\!|A|\!-\!|B|}{\rm hom}(G\!/\!C\!-\!(A\cup\!(B\!\setminus\! C)), H_{\mathbf{k}}),\end{align*}}}

\noindent where $B=\{uv\in E\setminus A \, : \, f(u)=f(v)\}$ and $G/C-(A\cup (B\setminus C))$ is the graph obtained from $G$ by contracting the edges of $C$ and by deleting the edges of $A\cup (B\setminus C)$.
This implies that the sequence $(\widetilde{H}_{\mathbf{k}})$  is (strongly) polynomial when $(H_{\mathbf{k}})$ is.
\end{proof}

A useful consequence of Proposition~\ref{poly_seq_preserve} is the following result, complementary to Corollary~\ref{remove_loops}:

\begin{corollary}\label{add_loops}
 Let $(H_{\mathbf{k}})$ be a sequence of simple graphs (with the possible exception of a loop on some vertices) which is strongly polynomial in ${\mathbf{k}}\in\mathbb{N}^h$. Then the sequence resulting from adding a loop to all vertices of $H_{\mathbf{k}}$ that do not yet have a loop attached 
  is (strongly) polynomial in $\mathbf{k}$.

\end{corollary}


To conclude this section, we show that another unary operation on graphs preserves the property of being (strongly) polynomial, namely that of taking the line graph.

\begin{proposition}\label{line}
If $(H_{\mathbf{k}})$ is a (strongly) polynomial sequence of simple graphs, then the sequence $(L(H_{\mathbf{k}}))$ of line graphs is also (strongly) polynomial.
\end{proposition}
\begin{proof}
We prove the result for strongly polynomial sequences; the result for polynomial sequences can be shown in a similar way.

Whitney's classical graph isomorphism theorem from 1932 states that $L(H)\cong L (H')$ for simple graphs $H, H'$ if and only if $H\cong H'$ with the exception only of $K_3$ and $K_{1,3}$, which both have line graph $K_3$.

By \cite[Prop. 5]{dlHJ95}, a sequence $(H_{\mathbf{k}})$ of graphs is strongly polynomial if and only if the number of subgraphs of $H_{\mathbf{k}}$ isomorphic to a given graph $S$ is polynomial in $\mathbf{k}$. Further, when each $H_{\mathbf{k}}$ is simple, this is equivalent to requiring that the number of induced subgraphs of $H_\mathbf{k}$ isomorphic to a given graph $S$ is polynomial in $\mathbf{k}$.
An induced subgraph of line graph $L(H)$ on vertices $A\subseteq E(H)$ is again a line graph, equal to $L(H_A)$ where $H_A$ is the subgraph $(V(H),A)$ of $H$.

Since each graph $H_{\mathbf{k}}$ is simple, the sequence $(L(H_{\mathbf{k}}))$ of line graphs also consists of simple graphs.
Suppose first that $S\not\cong K_3, K_{1,3}$. Then the number of occurrences of $L(S)$ in $L(H_{\mathbf{k}})$ as an induced subgraph is by Whitney's isomorphism theorem equal to the number of copies of $S$ in $H_{\mathbf{k}}$ as a subgraph, which by asssumption is polynomial in $\mathbf{k}$.
When $L(S)=K_3$, the number of induced copies of $L(S)$ in $L(H_{\mathbf{k}})$ is equal to the sum of the number of $K_3$'s and number of $K_{1,3}$'s occuring as a subgraph in $H_{\mathbf{k}}$, again polynomial in $\mathbf{k}$ by assumption.
 \end{proof}

\subsection{Products, ornamentation and composition}\label{sec:ornamentation}

In the results of this section for combining a pair of sequences $(H_{\mathbf{k}})$ indexed by $\mathbf{k}=(k_1,\ldots, k_h)\in \mathbb{N}^h$ and $(F_\mathbf{j})$ indexed by $\mathbf{j}=(j_1,\dots, j_\ell)\in \mathbb{N}^{\ell}$ we make the following assumptions. 
For each $1\leq s\leq h$ either $k_s=j_t$ for some $1\leq t\leq \ell$ or $k_s$ is independent of $j_t$ for each $1\leq t\leq \ell$. Likewise, each $j_t$ is either equal to one of the $k_s$ or is independent of all of them.
After combining the sequences $(H_{\mathbf{k}})$ and $(F_{\mathbf{j}})$ we then obtain a sequence indexed by variables $\{k_s:1\leq s\leq h\}\cup\{j_t:1\leq t\leq \ell\}$, but we shall simply say the sequence obtained is ``indexed by $(\mathbf{j},\mathbf{k})$'', the elimination of repeated variables being understood. In particular, when $\mathbf{j}=\mathbf{k}$, the result of combining sequences $(F_{\mathbf{k}})$ and $(H_{\mathbf{k}})$ will be another sequence indexed by $\mathbf{k}$. 

The \emph{categorical product} (or {\em direct product}) of two graphs $G$ and $H$, denoted by $G\times H$, is the graph with vertex set $V(G)\times V(H)$, two vertices $(u,v)$, $(u',v')$ being adjacent if $uu'\in E(G)$ and $vv'\in E(H)$.

\begin{proposition}\label{disjoint_union}
If $(F_{\mathbf{j}})$ and $(H_{\mathbf{k}})$ are (strongly) polynomial sequences of graphs in $\mathbf{j}$ and $\mathbf{k}$ respectively, 
 then the sequences $(F_{\mathbf{j}}\cup H_{\mathbf{k}})$ and $(F_{\mathbf{j}}\times H_{\mathbf{k}})$ obtained by disjoint union and categorical product are (strongly) polynomial in $(\mathbf{j},\mathbf{k})$.
\end{proposition}
\begin{proof}
When $G=G_1\cup G_2$ we have $${\rm hom}(G,F_{\mathbf{j}}\cup H_{\mathbf{k}})={\rm hom}(G_1,F_{\mathbf{j}}\cup H_{\mathbf{k}}){\rm hom}(G_2,F_{\mathbf{j}}\cup H_{\mathbf{k}}).$$ Thus, it suffices to consider a connected graph $G$, where ${\rm hom}(G,F_{\mathbf{j}}\cup H_{\mathbf{k}})={\rm hom}(G,F_{\mathbf{j}})+{\rm hom}(G,H_{\mathbf{k}})$, because the image of $G$ under a homomorphism is connected.

Let $\mathbf{j}$ range over $\mathbb{N}^\ell$ and $\mathbf{k}$ range over $\mathbb{N}^h$. Given a partition $\{I_1,\ldots, I_r\}$ of $\mathbb{N}^{\ell}$ and partition $\{J_1,\ldots, J_s\}$ of $\mathbb{N}^{h}$ into a finite number of subsets, the partition $\{I_t\times J_u:1\leq t\leq r, 1\leq u\leq s\}$ of $\mathbb{N}^{\ell+h}$ also has a finite number of subsets.
Suppose that ${\rm hom}(G,F_{\mathbf{j}})=p_t(G; {\mathbf{j}})$ for ${\mathbf{j}}\in I_t$, and ${\rm hom}(G,H_{\mathbf{k}})=q_u(G; {\mathbf{k}})$ for ${\mathbf{k}}\in J_u$. Then, ${\rm hom}(G,F_{\mathbf{j}})+{\rm hom}(G,H_{\mathbf{k}})=p_t(G; {\mathbf{j}})+q_u(G; {\mathbf{k}})$ when $(\mathbf{j},\mathbf{k})\in I_t\times J_u$. Hence $(F_{\mathbf{j}}\cup H_{\mathbf{k}})$ is a polynomial sequence in $(\mathbf{j},\mathbf{k})$ (strongly polynomial if both $(F_{\mathbf{j}})$ and $(H_{\mathbf{k}})$ are strongly polynomial).

The assertion for the categorical product follows in a similar way and from the fact that ${\rm hom}(G,F_{\mathbf{j}}\times H_{\mathbf{k}})=$ ${\rm hom}(G,F_{\mathbf{j}}){\rm hom}(G,H_{\mathbf{k}})$.
\end{proof}


The {\em join} of two graphs $G_1, G_2$ is defined by taking their disjoint union and then adding edges $v_1v_2$ for every pair $v_1\in V(G_1), v_2\in V(G_2)$.

\begin{corollary}\label{join}
If $(F_{\mathbf{j}})$ and $(H_{\mathbf{k}})$ are (strongly) polynomial sequences of graphs in $\mathbf{j}$ and $\mathbf{k}$ respectively,
 then the sequence $(F_{\mathbf{j}}+ H_{\mathbf{k}})$ of graph joins is (strongly) polynomial in $(\mathbf{j},\mathbf{k})$.
\end{corollary}
\begin{proof}
Since $F+H=\overline{\overline{F}\cup\overline{H}}$ this follows from Propositions~\ref{complement} and~\ref{disjoint_union}.
\end{proof}

For a fixed simple graph $H$, a family of (possibly weighted) graphs $\{F_v:v\in V(H)\}$ that are indexed by vertices of $H$ is  an {\em ornamentation} of $H$. By an \emph{ornamented graph} we mean a graph together with an ornamentation of its vertices. 
We extend the notion of ornamentation to graph sequences: an ornamentation of a graph sequence $(H_{\mathbf{k}})$ by a family  $\{(F_{v;{\mathbf{j}}}:{\mathbf{j}}\in\mathbb{N}^{h})\}$ of graph sequences assigns, for each given value ${\mathbf{j}}\in\mathbb{N}^{h}$,  the graph $F_{v;{\mathbf{j}}}$ to vertex $v\in V(H_{\mathbf{k}})$. This gives a graph sequence indexed by $(\mathbf{j},\mathbf{k})$.

From an ornamentation of a graph $H$ we form the {\em composition} 
by taking the disjoint union of the ornaments $F_{v}$, $v\in V(H)$, and adding edges to form the graph join of $F_{u}$ and $F_{v}$ whenever $uv\in E(H)$. This composition is denoted by $H[\{F_v\}]$. Note that we distinguish between an ornamented graph  and the graph obtained after composition.

It is remarked in~\cite[Ex. B7]{dlHJ95} that the composition of a fixed graph $H$ with polynomial sequences as ornaments gives a polynomial sequence. For completeness we formulate this in the following slightly more general form:
\begin{proposition}\label{composition}
Let $H$ be a simple graph and $h\in\mathbb{N}$ be fixed.
  For each $v\in V(H)$, let $(F_{v;{\mathbf{j}_v}})$ be a (strongly) polynomial graph sequence in $\mathbf{j}_v\in\mathbb{N}^{h_v}$.
Then $(H[\{F_{v;{\mathbf{j}_v}}\}])$ is a (strongly) polynomial sequence in $(\mathbf{j}_v:v\in V(H))$. 
   \end{proposition}
\begin{proof} To prove this result use the formula
\be\label{eq:composition}{\rm hom}(G,H[\{F_{v;{\mathbf{j}_v}}\}])=\sum_{\stackrel{f:G\rightarrow H[\{K_1^1\}]}{\mbox{\rm \tiny homomorphism}}}\;\;\prod_{v\in V(H)}{\rm hom}(G[f^{-1}(\{v\})],F_{v;{\mathbf{j}_v}}),\ee
where $K_1^1$ denotes the graph consisting of one vertex with a loop attached (and thus $H[\{K_1^1\}]$ is the graph $H$ with a loop attached to each vertex), and $G[f^{-1}(\{v\})]$ denotes the induced subgraph of $G$ on vertices in the pre-image of $v$ under homomorphism $f$.
\end{proof}

Notice that Proposition~\ref{disjoint_union} for disjoint unions is the special case $H=K_1\cup K_1$ in Proposition~\ref{composition} with vertices ornamented by sequences $(F_{\mathbf{j}})$ and $(H_{\mathbf{k}})$ and Corollary~\ref{join} is the special case $K_2$ with the same ornamentation.

Proposition~\ref{composition} also includes an important special case, namely that of {\em blow-ups}. Given a graph $H$ and $\mathbf{k}\in \mathbb{N}^{V(H)}=(k_s:s\in V(H))$, the {\em $\mathbf{k}$-blow-up} of $H$ is obtained from $H$ by replacing every vertex $s\in V(H)$ with $k_s$ distinct twin vertices, where a copy of $s$ is adjacent to a copy of $t$ in the blow-up graph if and only if $s$ is adjacent to $t$ in $H$. When there are no loops this is precisely the composition $H[\{\overline{K_{k_s}}:s\in V(H)\}]$, and by Proposition~\ref{composition} the sequence $(\, H[\{\overline{K_{k_s}}:s\in V(H)\}]\,)$ is strongly polynomial in $\mathbf{k}$. When $s$ has a loop we ornament $s$ with $K_{k_s}$ instead of $\overline{K_{k_s}}$.
In particular, starting from a graph $H$, the operation of blowing up a vertex $k$-fold produces a strongly polynomial sequence $(H_k)$. 

The {\em lexicographic product} $G_1[G_2]$ of two simple graphs $G_1$ and $G_2$ is the simple graph with vertex set $V(G_1)\times V(G_2)$, two vertices $(u_1,u_2)$ and $(v_1, v_2)$ being adjacent if either $u_1v_1\in E(G_1)$ or $u_1=v_1$ and $u_2v_2\in E(G_2)$. This corresponds to ornamenting the vertices of $G_1$ each with a copy of the graph $G_2$ and forming the composition.

An extension of Proposition~\ref{composition} replacing the fixed graph $H$ by a sequence of graphs $(H_{\mathbf{k}})$ is the following:

\begin{proposition}\label{lexicographic}
If $(F_{\mathbf{j}})$ and $(H_{\mathbf{k}})$ are (strongly) polynomial sequences in $\mathbf{j}$ and ${\mathbf{k}}$ respectively, then the sequence $(H_{\mathbf{k}}[F_{\mathbf{j}}])$ of lexicographic products is (strongly) polynomial in $(\mathbf{j},\mathbf{k})$.
\end{proposition}
\begin{proof}
We prove the result for strongly polynomial sequences, the case of polynomial sequences being similar, only notationally more complicated.

By equation~\eqref{eq:composition},
\begin{align*} {\rm hom}(G,H_{\mathbf{k}}[\{F_{\mathbf{j}}\}]) & =\sum_{\stackrel{f:G\rightarrow H_{\mathbf{k}}[\{K_1^1\}]}{\mbox{\rm \tiny homomorphism}}}\;\;\prod_{v\in V(H_{\mathbf{k}})}{\rm hom}(G[f^{-1}(\{v\})], F_{\mathbf{j}})\\
 & =\sum_{\stackrel{f:G\rightarrow H_{\mathbf{k}}[\{K_1^1\}]}{\mbox{\rm \tiny homomorphism}}}\;\;{\rm hom}\left(\bigsqcup_{v\in V(H_{\mathbf{k}})}\!\!G[f^{-1}(\{v\})],F_{\mathbf{j}}\right).\end{align*}

 Each term in this sum is of the form ${\rm hom}(G',F_{\mathbf{j}})$ for spanning subgraph $G'$ of $G$, and hence by assumption equal to a polynomial in $\mathbf{j}$. Furthermore, this polynomial is independent of any particular vertex $v\in V(H_{\mathbf{k}})$ (this is important so as not to have the situation where this polynomial, which is evaluated at $\mathbf{j}$, is in some way itself dependent on $\mathbf{k}$, which is problematic should $\mathbf{j}$ and $\mathbf{k}$ be dependent (e.g. if $\mathbf{j}=\mathbf{k}$) -- see discussion in Remark~\ref{rem:non_poly} below).
 Since by Corollary~\ref{add_loops}, ${\rm hom}(G, H_{\mathbf{k}}[\{K_1^1\}])$ is the evaluation of a polynomial in ${\mathbf{k}}$, we have ${\rm hom}(G,H_{\mathbf{k}}[\{F_{\mathbf{j}}\}])$ equal to a sum comprising polynomial in ${\mathbf{k}}$ number of summands, each of which is a polynomial in ${\mathbf{j}}$, and hence itself equal to a polynomial in $(\mathbf{j},{\mathbf{k}})$.
\end{proof}

\begin{remark}\label{rem:non_poly}
Consider the sequence of graphs $(H_k)$, $k\in \mathbb{N}$, with $V(H_k)=[k]$ and where vertex $v\in [k]$ is an isolated vertex if $v\neq 2^\ell$ for some $\ell$, and an isolated vertex with a loop if $v=2^\ell$. The graph $H_k$ is the composition of $\overline{K}_k$ ornamented with $K_1^1$ on vertices equal to $2^\ell$ for some $\ell$, and $K_1$ on remaining vertices. The sequence $(\overline{K}_k)$ is strongly polynomial, as are the sequences $(K_1^1)$ and $(K_1)$ (being constant), but ${\rm hom}(G,H_k)$ is not polynomial in $k$. 
This indicates that Proposition~\ref{lexicographic} might be the best we could hope for as an extension of Proposition~\ref{composition} to ornamentations of a sequence $(H_{\mathbf{k}})$ (rather than a fixed graph $H$), for as soon as we allow more than one possible ornament graph the property of polynomiality may be destroyed. However, in Section~\ref{sec:branching}, we shall see that for a large class of sequences $(H_{\mathbf{k}})$ there are other ways to ornament each vertex of $H_{\mathbf{k}}$ by the $\mathbf{k}$th term of a polynomial sequence of graphs so that the resulting sequence of compositions is again a polynomial sequence (Theorem~\ref{thm:poly_branching}). 
\end{remark}

\subsection{Examples}\label{examples}
As in~\cite{dlHJ95} an essential part of this paper consists of examples.
We give some concrete instances of polynomial sequences of graphs that can be obtained from basic polynomial sequences such as $(\overline{K}_k)$ (which are easy to verify are polynomial) and the use of such operations as complementation and composition described in the previous two sections.

\begin{example} Blowing up $k$-fold the graph $K_1$ gives the sequence $(\overline{K_k})$ of cocliques, which is strongly polynomial; indeed, ${\rm hom}(G,\overline{K}_k)=0$ unless $G$ is edgeless, in which case it is equal to $k^{|V(G)|}$. By Proposition \ref{lcomplement}, the sequence $\widetilde{\overline{K}}_k=K_k^1$ of looped complete graphs is also strongly polynomial, and indeed ${\rm hom}(G,K_k^1)=k^{|V(G)|}$ for all graphs $G$.
Alternatively, this can be seen because the $k$-fold blow-up of $K_1^1$ is $K_k^1$. 
 \end{example}

\begin{example} The sequence $(kK_1^1)$ is strongly polynomial, with ${\rm hom}(G,kK_1^1)=k^{c(G)}$, where $c(G)$ is the number of connected components of $G$. (This also follows by Corollary~\ref{add_loops} and the fact that $(\overline{K_k})$ is strongly polynomial, 
or by Proposition~\ref{lexicographic}, taking the lexicographic product of $\ol{K_k}$ with $K_1^1$.)

By Proposition \ref{lcomplement}, the looped complement sequence $(\widetilde{kK_1^1})=(K_k)$ is also strongly polynomial: ${\rm hom}(G,K_k)=P(G;k)$. Here $P(G; k)$ is the chromatic polynomial of $G$ and equation~{\rm \eqref{eq:loop_compl}} gives the well known subgraph expansion of the chromatic polynomial,
$$P(G;k)=\sum_{D\subseteq E(G)}(-1)^{|E(G)|-|D|}k^{c(G-D)}.$$
Equation~{\rm \eqref{eq:compl}} yields the same subgraph expansion,
because ${\rm hom}(G/C-D,\ol{K}_k)$ is zero if $C\cup D\neq E(G)$ and $k^{|V(G/C-D)|}$ otherwise, and $|V(G/C-D)|=c(G-D)$ since contracting edges in $C=E(G)\setminus D$ shrinks each component of $G-D$ into an isolated vertex.
\end{example}

\begin{example}\label{ex:rooks}
Ornamenting any given simple graph with either of the two sequences $(K_k)$ and $(\overline{K_k})$ gives, by Proposition~\ref{composition},  another strongly polynomial sequence. 
In particular, the complete bipartite graphs $(K_{j,k})$ form a strongly polynomial sequence in $(j,k)$.
 Indeed, ${\rm hom}(G,K_{j,k})=j^{|V_0(G)|}k^{|V_1(G)|}+j^{|V_1(G)|}k^{|V_0(G)|}$ when $G$ is connected and bipartite with bipartition $V(G)=V_0(G)\cup V_1(G)$, and zero when $G$ is not bipartite.

The line graph $L(K_{j,k})$ of a complete bipartite graph is known as a {\em rook's graph} (for the $j\times k$ chessboard) and is isomorphic to the Cartesian product $K_j\square K_k$ of complete graphs.
 By Proposition~\ref{line} the sequence $(K_{j}\square K_k)$ is strongly polynomial in $(j,k)$.
(However, we do not have a proof that in general $(F_{\mathbf{j}}\square H_{\mathbf{k}})$ is strongly polynomial in $(\mathbf{j},\mathbf{k})$ when $(F_{\mathbf{j}})$ and $(H_{\mathbf{k}})$ are strongly polynomial.)
\end{example}

\begin{example}
That the sequence $(K_k^1+K_{t-k}^{\ell})$ is polynomial in $k, t$ and $\ell$ follows from the fact that $(K_k^1)$ is a polynomial sequence: by Proposition~\ref{poly_seq_preserve} the sequence $(K_{t-k}^{\ell})$ is polynomial in $t-k$ and $\ell$, and by Proposition~\ref{composition} the ornamentation of $K_2$ with these two sequences gives by composition a polynomial sequence in $k, t-k$ and $\ell$.
In fact we have $${\rm hom}(G,K_k^1+K_{t-k}^{\ell})=\xi(G;t,\ell-1,(k-t)(\ell-1)),$$ where $\xi(G;x,y,z)$ is the Averbouch--Godlin--Makowsky polynomial (see~\cite{AGM, GGN11}).

\end{example}

\begin{example}

We consider here sequences with index variables $(j,k_1,\ldots, k_{h})\in \mathbb{N}^{h+1}$. (In fact just one variable will be used for each sequence.)
By Proposition~\ref{composition}, for fixed $j$ the sequence $(K_j[\{\overline{K}_{k_i}:1\leq i\leq h\}])$ is strongly polynomial in $(k_1,\ldots, k_{h})$.
Indeed, for $j\geq h$ we have
$${\rm hom}(G,K_j[\{\overline{K}_{k_i}:1\leq i\leq h\}])=X(G;(k_1,k_2,\ldots, k_{h},1,\ldots, 1, 0, 0, \ldots)),$$
where on the right-hand side there are $j$ non-zero arguments in the function
$$X(G;(x_i:i\in\mathbb{N}))=\sum_{f:V(G)\rightarrow\mathbb{N}}\;\prod_{i\in\mathbb{N}}x_i^{|f^{-1}(\{i\})|},$$
 where the summation is over all proper colourings of $G$ by positive integers. (This is Stanley's symmetric chromatic function~\cite{S95}.)

Thus, fixing $j= h$, the strongly polynomial sequence $(K_h[\{\overline{K}_{k_i}:1\leq i\leq h\}])$ indexed by $k_1,k_2,\ldots, k_{h}$ determines Stanley's symmetric function in its first $h$ variables.

By ornamenting vertex $i$ in $K_j$ with $\overline{K}_{k_i}^{\ell_i}$ ($\ell_i$ loops on each vertex of $\overline{K}_{k_i}$) instead of $\overline{K}_{k_i}$ we obtain an evaluation of Noble and Merino's {\em strong Tutte symmetric function}~\cite{NM08}, itself equivalent to the strong $U$-polynomial (in a doubly infinite sequence of indeterminates), a simultaneous generalization of Stanley's chromatic function and Tutte's V-functions. The case where $\ell_i=\ell$ is constant gives an evaluation of the monochrome polynomial extension of Stanley's symmetric chromatic function:
$$\sum_{f:V(G)\rightarrow\mathbb{N}}\;\prod_{i\in\mathbb{N}}x_i^{|f^{-1}(\{i\})|}y^{\#\{uv\in E(G): f(u)=f(v)\}}$$
where the sum is now over all colourings, not necessarily proper.

\end{example}

For the following two examples, compare \cite[Ex. B3]{dlHJ95}.

\begin{example}\label{ex:K2...2} 
 The graph $\overline{kK_2}$ is the complete multipartite graph $K_{2,2,,\ldots, 2}$, with $k$ occurrences of $2$. 
With ${\rm hom}(G,kK_2)=k^{c(G)}P(G;2)$, the sequence $(kK_2)$ is strongly polynomial in $k$. It follows by Proposition~\ref{complement} that $(\overline{kK_2})$ is strongly polynomial in $k$ too.
(This follows alternatively from Proposition~\ref{lexicographic}, since  $\overline{kK_2}=K_k[\overline{K_2}]$ and $(K_k)$ and $(\overline{K_2})$ are strongly polynomial sequences.)
\end{example}

\begin{example}\label{ex:cubes}
The sequence of hypercubes $(Q_k)$ ($Q_k$ is the $k$-fold Cartesian product of $K_2$ with itself) is not polynomial in the single parameter $k\in\mathbb{N}$, no matter how one re-indexes it by any other single parameter. This is because ${\rm hom}(K_1, Q_k)=2^k$ and ${\rm hom}(K_2, Q_k)=k2^{k}$. If it were the case that ${\rm hom}(K_1,Q_k)=p(K_1; f(k))$ and ${\rm hom}(K_2, Q_k)=p(K_2;f(k))$ for polynomials $p(K_1;x)$ and $p(K_2;x)$ and some injective function $f:\mathbb{N}\rightarrow\mathbb{N}$ then we have $\frac{p(K_2;f(k))}{p(K_1;f(k))}=k$ while at the same time $\frac{p(K_2;f(k))}{p(K_1;f(k))}\sim cf(k)^d$ for some constant $c$ and non-zero integer $d$; this implies $f(k)\sim (k/c)^{1/d}$ but then $p(K_1;f(k))$ would be polynomially bounded in $k$, contrary to  $p(K_1;f(k))=2^k$.
\end{example}

Counting homomorphisms to a Cartesian product of graphs (such as $Q_k\cong K_2\Box Q_{k-1}$) remains an intricate problem; see for example~\cite{HT97}. Counting (and structural properties of) homomorphisms {\em from} hypercubes have been intensively studied~\cite{K01,G03}.

If instead we index the hypercubes $(Q_k)$ 
 by two parameters $k_1=k$ and $k_2=2^k$, then ${\rm hom}(K_{1,\ell},Q_k)=k^\ell2^k=k_1^\ell k_2$ is a bivariate polynomial in $k_1$ and $k_2$.
Although as we have just seen the sequence $(Q_k)$ cannot by reindexing be made polynomial in a single variable, the quantity ${\rm hom}(G, Q_k)$ is in fact polynomial in two (dependent) variables:

\begin{proposition}\label{prop:hypercubes}
Let $Q_k$ be the hypercube on $2^k$ vertices. For each graph $G$ there is a bivariate polynomial $p(G;x_1,x_2)$ such that ${\rm hom}(G,Q_k)=p(G;k,2^k)$ for all $k\geq 0$.
\end{proposition}
\begin{proof}
Let $V(Q_k)=\mathbb{Z}_2^k$ be the set of binary vectors of length $k$. Vertices $u$ and $v$ are joined by an edge if and only if they differ in exactly one position, in other words $u+v$ has Hamming weight $1$.

The automorphism group ${\rm Aut}(Q_k)$ of $Q_k$ has size $2^kk!$, consisting of coordinate permutations and addition of a fixed vector $w\in\mathbb{Z}_2^k$ to all vertices.

To prove ${\rm hom}(G,Q_k)$ is polynomial in $k$ and $2^k$ it suffices to show that the number of subgraphs of $Q_k$ of a given isomorphism type $S$ with $|V(S)|\leq |V(G)|$ is, for sufficiently large $k$, a polynomial in $k$ and $2^k$. This is because
$${\rm hom}(G,Q_k)=\sum_{\stackrel{S\subseteq Q_k}{|V(S)|\leq |V(G)|}}{\rm sur}(G, S),$$
where ${\rm sur}(G,S)$ denotes the number of vertex- and edge-surjective homomorphisms from $G$ onto the subgraph $S$ of $Q_k$, in which $S$ has at most $|V(G)|$ vertices (independent of $k$). Moreover, we may assume here that $S$ is connected (since it is enough to prove ${\rm hom}(G, Q_k)$ is polynomial in $k$ and $2^k$ for connected~$G$).

Next we prove that any given copy of $S$ in $Q_k$ as a subgraph can be moved by an automorphism of $Q_k$ to a copy $S'$ of $S$ in $Q_m$, where $m=m(S)$ depends on $S$ but not $k$, and $Q_m\subseteq Q_k$ has vertex set $\mathbb{Z}_2^m\times \{0\}^{k-m}$.

To show this, first apply an automorphism of $Q_k$ that sends one vertex of $S$ to the all-zero vector, by adding $w$ to every vertex for some $w\in V(S)$. Denote this subgraph by $S+w$. Since $S$ is connected and $uv$ are adjacent in $S$ only if $u+v$ has Hamming weight $1$, each vertex of $S+w$ is supported on a subset of a fixed set of $m={\rm diam}(S)$ coordinates. Now apply an automorphism of $Q_k$ that permutes coordinates so that the $m$ coordinates non-zero for some vertex in $S+w$ are the initial $m$ coordinates; the remaining $k-m$ coordinates are zero. This produces the desired copy $S'$ of $S$ contained within the induced subcube $Q_m$.

It now follows that the number of distinct orbits of a copy of $S$ under ${\rm Aut}(Q_k)$ is independent of $k$. Further, each orbit of $S$ (equal to an orbit of $S'$ for subgraph $S'$ of the induced $Q_m$ on vertex set $\mathbb{Z}_2^m\times \{0\}^{k-m}$) has size a divisor of $\frac{2^kk!}{(k-m)!}$ (a polynomial in $k$ and $2^k$), since the stabilizer of $S'$ in $Q_m$ under the action of ${\rm Aut}(Q_k)$ contains as a subgroup all permutations of the last $k-m$ coordinates of vertices of $Q_k$.
We have therefore shown that, for connected $G$ and $k\geq {\rm diam}(G)$,
$${\rm hom}(G,Q_k) =\sum_{S}{\rm sur}(G,S)q(S;k,2^k),$$
where the summation is over all isomorphism types $S$ of graphs which are homomorphic images of $G$ and $q(S;k,2^k)$ is polynomial in $k$ and $2^k$ (for $k\geq {\rm diam}(G)$ equal to the number of copies of $S$ as a subgraph of $Q_k$). In other words, for $k\geq {\rm diam}(G)$,
 $${\rm hom}(G,Q_k) = p(G;k,2^k),$$
 for some bivariate polynomial $p(G)$.

On the other hand, since for each $1\leq i\leq m$ there is some vertex of the subgraph $S'$ whose $i$th coordinate is not $0$, only permutations of the first $m$ coordinates among themselves can belong to the stabilizer of $S'$. Also, none of the last $k-m$ coordinates can be flipped in value in an automorphism stabilizing $S'$. Hence the size of the stabilizer of $S'$ under the action of ${\rm Aut}(Q_k)$ is a divisor of $2^m(k-m)!m!$, and the size of the orbit of $S'$ is therefore a multiple of $\frac{2^{k-m}k!}{(k-m)!m!}$. 

A polynomial in $k$ that is both a multiple of $\frac{2^{k-m}k!}{(k-m)!m!}$ and a divisor of $\frac{2^kk!}{(k-m)!}$ vanishes when evaluated at $k<m$.
It follows that $q(S;k,2^k)=0$ whenever $k<{\rm diam}(S)$. Hence we also have, for $k<{\rm diam}(G)$,
\begin{align*}{\rm hom}(G,Q_k) & =\sum_{\stackrel{S:}{|V(S)|\leq 2^k, {\rm diam}(S)\leq k}} {\rm sur}(G,S)q(S;k,2^k)\\
& =\sum_{S} {\rm sur}(G,S)q(S;k,2^k)=p(G;k,2^k),\end{align*}
where the summations are over isomorphism types of graphs $S$ which are images of $G$, restricted in the first sum to those images $S$ which are sufficiently small to appear in $Q_k$ with $k<{\rm diam}(G)$, but unrestricted in the second sum.

We conclude that $p(G;k,2^k)$ is the number of homomorphisms from $G$ to $Q_k$ for all values $k\geq 0$.
%
\end{proof}

Proposition~\ref{prop:hypercubes} does not tell us whether the sequence of hypercubes $Q_k$ indexed by the two variables $k_1=k$ and $k_2=2^k$ is a strongly polynomial subsequence in our sense: is there a strongly polynomial sequence $(H_{k_1,k_2})$ indexed by $(k_1,k_2)\in\mathbb{N}\times\mathbb{N}$ such that $H_{k_1,k_2}=Q_k$ when $(k_1,k_2)=(k,2^k)$?

\begin{remark} The automorphism group of the hypercube $Q_k$ is isomorphic to that of the complete multipartite graph $\overline{kK_2}$, both equal to the wreath product of $\mathrm{Sym}(k)$ with $\mathrm{Sym}(2)$. Sabidussi~\cite{S59} showed in general that $G_1\times G_2$ has automorphism group isomorphic to that of the disjoint union $G_1\cup G_2$. Thus the $k$-fold Cartesian product $H^{\times k}$ of $H$ with itself has automorphism group isomorphic to that of $kH$, and the latter coincides with the automorphism group of $\overline{kH}$. All three automorphism groups are isomorphic to $\mathrm{Sym}(k)\wr \mathrm{Aut}(H)$, of order $|{\rm Aut}(H)|^kk!$. We have seen that the sequence $(\overline{kH})$ is polynomial in $k$, whereas $(Q_k)$ is only polynomial in $k$ and $2^k$. The difference lies in the nature of the actions of the automorphism group on the vertex set of the graph: for the complete multipartite graphs only the subgroup ${\rm Sym}(k)$ of ${\rm Aut}(\overline{kK_2})$ is required in order to move a subgraph $S$ of $\overline{kK_2}$ to an isomorphic copy in $\overline{mK_2}$, where $m=m(S)$ is independent of $k$.
\end{remark}

\section{Coloured rooted trees and branching}\label{sec:treesbranching}


We turn now to the main contribution of this paper, which is to introduce a new method to generate strongly polynomial sequences of graphs $(H_{\mathbf{k}})$ determining a multivariate graph polynomial $p(G;\mathbf{k})$.
We have seen that the chromatic polynomial, Tutte polynomial and Averbouch--Godlin--Makowsky polynomial can be obtained from strongly polynomial sequence of graphs built from the constant sequence $(K_1)$ alone, by applying operations of blow-ups, loop additions and joins. The Tittmann--Averbouch--Makowsky polynomial cannot be obtained using only the operations described hitherto, but can be obtained using our new construction.

\subsection{Coloured rooted trees}\label{sec:coloured_trees}

Let $T$ be a rooted tree with root $r$. For $s\in V(T)-r$, let $P(s)$ denote the unique path from root $r$ to vertex $s$. A vertex $t\neq s$ is an {\em ancestor} of $s$ if $t\in P(s)$, and a {\em descendant} of $s$ if $s\in P(t)$.  The {\em parent} of $s$ is the vertex adjacent to $s$ in $P(s)$ and is denoted by $p(s)$. The relation $s\in P(t)$ defines a partial order on $V(T)$ with minimum element $r$ (which is an ancestor of all the other vertices) and maximal elements the leaves of $T$.
The {\em level} of $s\in V(T)$ is the number of its ancestors, i.e., $|P(s)|-1$. The root of $T$ is the unique vertex at level $0$. The {\em height} of a rooted tree $T$ is defined as the maximum level, i.e., ${\rm height}(T)=\max\{|P(s)|-1:s\in V(T)\}$.
The set $$B(s)=\{t\in V(T): s\in P(t)\}$$
corresponds to the subtree of $T$ rooted at $s$. (Here the {\em subtree rooted at $s$} is a maximal subtree of $T$ in the sense that it contains all descendants of its root $s$.)

The {\em closure} of $T$ is the graph ${\rm clos}(T)$ on vertex set $V(T)$ where $st$ is an edge if $s\in P(t)$ or $t\in P(s)$, and $s\neq t$. Alternatively, we can consider the rooted tree $T$ as the Hasse diagram of the poset defined by its ancestor relation; the comparability graph of this poset is then  ${\rm clos}(T)$.

\begin{example}
Let $T=P_k$ be the path on $k$ vertices with root one of its endpoints. Then ${\rm clos}(P_k)=K_k$.
The graph parameter ${\rm hom}(G,P_k)$ is a polynomial in $k$ 
 for $k\geq 1+{\rm diam}(G)$, i.e., $(P_k)$ is a polynomial sequence. (See~\cite[Ex. B4]{dlHJ95}.) Since ${\rm clos}(P_k)=K_k$, the parameter ${\rm hom}(G,{\rm clos}(T))$ is equal to the chromatic polynomial of $G$ evaluated at $k$ (see Example~\ref{ex:hom}).
\end{example}

A simple graph $H$ has {\em tree-depth} $d$, denoted by ${\rm td}(H)=d$, if $H$ is a subgraph of ${\rm clos}(T)$ for some rooted tree $T$ of height $d-1$, and it is not a subgraph of the closure of a rooted tree of smaller height. For example the path $P_k$ has tree-depth $\lceil\log_2(k+1)\rceil$ and the complete graph $K_k$ of course has tree-depth $k$. See~\cite{NOdM06}.

We now describe a three-stage colouring of the vertices of rooted tree $T$. The first two stages encode an ornamented graph. The third stage is the subject of Section~\ref{sec:branching} below.
To each such coloured rooted tree we shall see that there corresponds a strongly polynomial sequence of graphs.

Let us start with a given simple graph $H$ together with a rooted tree $T$ of whose closure $H$ is a subgraph. 
(For given $H$ there are many choices for $T$, such as a depth-first search tree for $H$, or a tree $T$ of minimal height ${\rm td}(H)-1$ whose closure contains $H$ as a subgraph.)
We assume that $V(H)=V(T)$, i.e., $H$ is a spanning subgraph of ${\rm clos}(T)$.

A subgraph $H$ of ${\rm clos}(T)$, where $T$ has height $d\geq {\rm td}(H)-1$, will be encoded by assigning to the vertices of $T$ subsets of $\{0,1,2,\ldots, d-1\}$. These sets are interpreted as colours. Thus there are $2^d$ colours. The colour of $s$ indicates which of the vertices on the chain $P(s)$ the vertex $s$ is adjacent to in $H$. Specifically, a non-root vertex $s\in V(T)$
 is assigned the set $A_s\subseteq\{0,1,\ldots, |P(s)|\!-\!2\}$ when in the subgraph $H$ the vertex $s$ is joined to its ancestors (other vertices in the chain $P(s)$) precisely at levels $i\in A_s$. (The root, the only vertex at level $0$, is always assigned the empty set.) In other words, $$A_s=\{|P(s)|\!-\!1\!-\! d(s,t): t\in P(s), st\in E(H)\},$$ where $d(s,t)$ is the distance between $s$ and $t$ in $T$. 
For example, $s$ receives the empty set if it is joined to none of its ancestors, and the set $\{0,1\ldots, |P(s)|\!-\!2\}$ when it is joined to all its ancestors. See Figure~\ref{fig:branching} below for a small example.

We can now encode in our colouring of $T$ any ornamentation of $H$.
Suppose then that each vertex $s\in V(H)$ is ornamented with a graph $F_s$. Since $H$ is a spanning subgraph of ${\rm clos}(T),$ this is the same as saying that each vertex $s\in V(T)$ is ornamented with a graph $F_s$, in addition to the subset $A_s$ assigned to vertex $s$ encoding the adjacencies of $s$ in the subgraph $H$.
Thus a coloured rooted tree $T$, in which vertex $s$ has colour $(A_s,F_s)$, corresponds to a spanning subgraph $H$ of ${\rm clos}(T)$ ornamented by graphs $\{F_s:s\in V(H)\}$.

\subsection{Branching}\label{sec:branching}

In this section we define an operation on coloured rooted trees, which we shall call ``branching'', and which translates to an operation on ornamented simple graphs once the tree colours have been interpreted appropriately. 

In a coloured rooted tree $T$, the colour of a vertex $s\in V(T)$ for us at the moment consists of two components, described in Section~\ref{sec:coloured_trees}: the subset $A_s$ of nonnegative integers that encodes the adjacencies of $s$ (to vertices in $P(s)$) in spanning subgraph  $H$  of ${\rm clos}(T)$ , and an ornament $F_s$ of $H$ placed on vertex $s$ (graph to be used in a composition).

In order to define branching it will be helpful to augment these colours by a third component to form an auxiliary coloured rooted tree.
Take then $T$ to be our coloured rooted tree, vertex $s\in V(T)$ having colour $(A_s, F_s)$ for $A_s\subseteq \{0,1,\ldots, |P(s)|\!-\!2\}$ and graph $F_s$, and let $\mathbf{k}=(k_s:s\in V(T))$ be a tuple of positive integers. We define $T^{(\mathbf{k})}$ to be the coloured rooted tree $T$ with vertex $s$ now having augmented colour $(A_s,F_s,k_s)$, for each $s\in V(T)$.

If $k_s=1$ for each $s\in V(T)$ then we have effectively an ornamented graph $H$: vertex $s\in V(T)$ of colour $(A_s, F_s,1)$ corresponds to vertex $s\in V(H)$ of spanning subgraph $H$ with ornament $F_s$.

When $k_s>1$ for some $s\in V(T)$ we use the branching operation of Definition~\ref{def:branching} below, to replace $T^{(\mathbf{k})}$ by another coloured rooted tree. This tree has the property that vertex $s$ of colour $(A_s, F_s, k_s)$ is replaced by $k_s$ copies of $s$ of colour $(A_s,F_s,1)$. By repeatedly branching on vertices in this way eventually we reach a coloured rooted tree in which all colours take the form $(A_s, F_s,1)$. By erasing the final colour component (now constantly $1$) in this final coloured tree, we have a coloured tree having colours of the form $(A_s,F_s)$ and 
which we shall denote by $T^{\mathbf{k}}$ (see Definition~\ref{def:augment} below).

Let us formalize this procedure more precisely. We start in Definition~\ref{def:branching} below with branching at a single given vertex $s$ of the coloured rooted tree~$T^{(\mathbf{k})}$. See also Figure~\ref{fig:branch}. We then establish that the order of vertices at which we branch does not matter: we always reach the same coloured rooted tree at the end (when all branching multiplicities are equal to $1$). 

\begin{definition}\label{def:branching}
Let $T$ be a coloured rooted tree and  $\mathbf{k}=(k_s:s\in V(T))$ a tuple of positive integers indexed by the vertices of $T$. The {\em branching of $T^{(\mathbf{k})}$ at a vertex $s\in V(T^{(\mathbf{k})})$} having colour $(A_s,F_s,k_s)$ is the coloured rooted tree that
\begin{itemize}
\item[(i)] coincides with $T^{(\mathbf{k})}$ on vertices $\{t:s\not \in P(t)\}=V(T)\setminus B(s)$, and
\item[(ii)] is such that $p(s)$ now has $k_s$ isomorphic copies of $B(s)$ pendant from it instead of one, where isomorphisms here are colour-preserving, with the exception of the $k_s$ copies of $s$, which are now each coloured $(A_s,F_s,1)$ instead of the original colour $(A_s,F_s,k_s)$.
\end{itemize}
The positive integer $k_s$ is the \emph{branching multiplicity} of vertex $s$.
\end{definition}

Having performed branching at a given vertex, we obtain another coloured rooted tree. We now proceed to branch at any given vertex of this new tree, and iterate until it becomes the case that all branching multiplicities, recorded in the third component of vertex colours, are equal to $1$. (Of course, $1$-fold branching does not effect any change.)

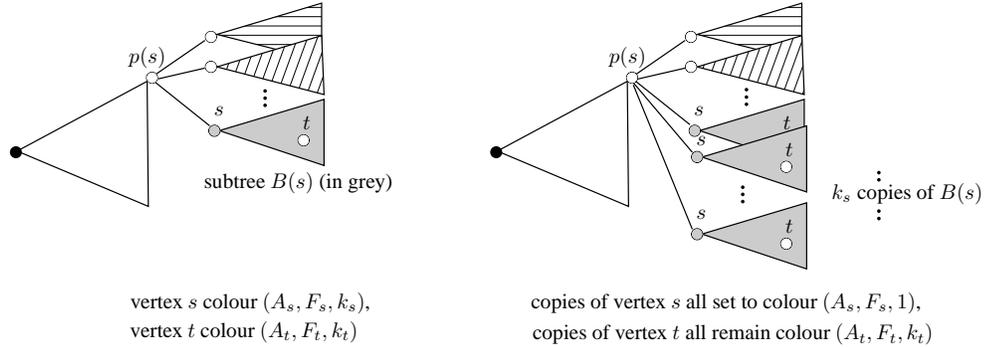
\begin{figure}[htp]

\begin{center}
\scalebox{0.8} 
{
\begin{pspicture}(0,-2.8889062)(16.621876,2.8589063)
\definecolor{color2259c}{rgb}{0.8,0.8,0.8}
\psline[linewidth=0.02cm](10.36,1.4889063)(11.34,-0.9110938)
\psline[linewidth=0.02cm](10.34,1.5489062)(11.34,0.36890626)
\psline[linewidth=0.02cm](2.4,1.6489062)(3.24,2.2289062)
\psline[linewidth=0.02cm](2.46,1.5889063)(3.24,1.7689062)
\psdots[dotsize=0.2](0.1,0.36890626)
\pspolygon[linewidth=0.02](0.12,0.38890624)(2.28,1.5489062)(2.3,-0.5310938)
\psdots[dotsize=0.2,fillstyle=solid,dotstyle=o](2.36,1.6089063)
\psdots[dotsize=0.2,fillstyle=solid,dotstyle=o](3.34,2.2889063)
\psline[linewidth=0.02cm](2.42,1.5289062)(3.32,0.80890626)
\pspolygon[linewidth=0.02,fillstyle=hlines,hatchwidth=0.0139999995,hatchangle=0.0](3.44,2.3289063)(5.18,2.8489063)(5.18,1.8889062)
\pspolygon[linewidth=0.02,fillstyle=hlines*,hatchwidth=0.0139999995,hatchangle=70.0](3.44,1.7889062)(5.18,2.3089063)(5.22,1.3289063)
\pspolygon[linewidth=0.02,fillstyle=hlines*,hatchwidth=0.0139999995,hatchangle=70.0,hatchcolor=color2259c,fillcolor=color2259c](3.48,0.7289063)(5.22,1.2489063)(5.22,0.14890625)
\psdots[dotsize=0.2,fillstyle=solid,fillcolor=color2259c,dotstyle=o](3.4,0.7289063)
\usefont{T1}{ptm}{m}{n}
\rput(3.994375,-2.1410937){vertex $s$ colour $(A_s, F_s, k_s)$,}
\psdots[dotsize=0.2,fillstyle=solid,dotstyle=o](3.34,1.7889062)
\psdots[dotsize=0.2,fillstyle=solid,dotstyle=o](4.88,0.56890625)
\usefont{T1}{ptm}{m}{n}
\rput(4.931406,0.85890627){$t$}
\usefont{T1}{ptm}{m}{n}
\rput(3.4714062,1.0589062){$s$}
\usefont{T1}{ptm}{m}{n}
\rput(2.2814062,1.9389062){$p(s)$}
\usefont{T1}{ptm}{m}{n}
\rput(3.889375,-2.6210938){vertex $t$ colour $(A_t, F_t, k_t)$}
\psline[linewidth=0.02cm](10.38,1.6489062)(11.22,2.2289062)
\psline[linewidth=0.02cm](10.44,1.5889063)(11.22,1.7689062)
\psdots[dotsize=0.2](8.08,0.36890626)
\pspolygon[linewidth=0.02](8.1,0.38890624)(10.26,1.5489062)(10.28,-0.5310938)
\psdots[dotsize=0.2,fillstyle=solid,dotstyle=o](10.34,1.6089063)
\psdots[dotsize=0.2,fillstyle=solid,dotstyle=o](11.32,2.2889063)
\psline[linewidth=0.02cm](10.4,1.5489062)(11.3,0.82890624)
\pspolygon[linewidth=0.02,fillstyle=hlines,hatchwidth=0.0139999995,hatchangle=0.0](11.42,2.3289063)(13.16,2.8489063)(13.16,1.8889062)
\pspolygon[linewidth=0.02,fillstyle=hlines*,hatchwidth=0.0139999995,hatchangle=70.0](11.42,1.7889062)(13.16,2.3089063)(13.2,1.3289063)
\pspolygon[linewidth=0.02,fillstyle=hlines*,hatchwidth=0.0139999995,hatchangle=70.0,hatchcolor=color2259c,fillcolor=color2259c](11.46,0.7289063)(13.2,1.2489063)(13.2,0.14890625)
\psdots[dotsize=0.2,fillstyle=solid,fillcolor=color2259c,dotstyle=o](11.38,0.7289063)
\psdots[dotsize=0.2,fillstyle=solid,dotstyle=o](11.32,1.7889062)
\psdots[dotsize=0.2,fillstyle=solid,dotstyle=o](12.86,0.56890625)
\usefont{T1}{ptm}{m}{n}
\rput(12.9114065,0.85890627){$t$}
\usefont{T1}{ptm}{m}{n}
\rput(11.4514065,1.0589062){$s$}
\usefont{T1}{ptm}{m}{n}
\rput(10.261406,1.9389062){$p(s)$}
\pspolygon[linewidth=0.02,fillstyle=hlines*,hatchwidth=0.0139999995,hatchangle=70.0,hatchcolor=color2259c,fillcolor=color2259c](11.5,0.28890625)(13.24,0.80890626)(13.24,-0.29109374)
\psdots[dotsize=0.2,fillstyle=solid,fillcolor=color2259c,dotstyle=o](11.42,0.28890625)
\usefont{T1}{ptm}{m}{n}
\rput(14.9114065,-0.32109374){$k_s$ copies of $B(s)$}
\psdots[dotsize=0.2,fillstyle=solid,dotstyle=o](12.9,0.12890625)
\usefont{T1}{ptm}{m}{n}
\rput(12.9514065,0.41890624){$t$}
\usefont{T1}{ptm}{m}{n}
\rput(11.491406,0.55890626){$s$}
\pspolygon[linewidth=0.02,fillstyle=hlines*,hatchwidth=0.0139999995,hatchangle=70.0,hatchcolor=color2259c,fillcolor=color2259c](11.5,-0.99109375)(13.24,-0.47109374)(13.24,-1.5710938)
\psdots[dotsize=0.2,fillstyle=solid,fillcolor=color2259c,dotstyle=o](11.42,-0.99109375)
\psdots[dotsize=0.2,fillstyle=solid,dotstyle=o](12.9,-1.1510937)
\usefont{T1}{ptm}{m}{n}
\rput(12.9514065,-0.86109376){$t$}
\usefont{T1}{ptm}{m}{n}
\rput(11.491406,-0.6610938){$s$}
\usefont{T1}{ptm}{m}{n}
\rput(4.789375,-0.14109375){subtree $B(s)$ (in grey)}
\usefont{T1}{ptm}{m}{n}
\rput(11.899844,-2.1410937){copies of vertex $s$ all set to colour $(A_s, F_s, 1)$,}
\psdots[dotsize=0.06](12.18,-0.21109375)
\psdots[dotsize=0.06](12.18,-0.33109376)
\psdots[dotsize=0.06](12.18,-0.45109376)
\psdots[dotsize=0.06](12.24,1.4089062)
\psdots[dotsize=0.06](12.24,1.2889062)
\psdots[dotsize=0.06](12.24,1.1689062)
\psdots[dotsize=0.06](4.24,1.3889062)
\psdots[dotsize=0.06](4.24,1.2689062)
\psdots[dotsize=0.06](4.24,1.1489062)
\usefont{T1}{ptm}{m}{n}
\rput(12.014844,-2.6610937){copies of vertex $t$ all remain colour $(A_t, F_t, k_t)$}
\psdots[dotsize=0.06](14.44,0.02890625)
\psdots[dotsize=0.06](14.44,-0.09109375)
\psdots[dotsize=0.06](14.44,-0.61109376)
\psdots[dotsize=0.06](14.44,-0.73109376)
\end{pspicture}
}
\caption{\small To branch with multiplicity $k_s$ at vertex $s$,  replace $s$ by $k_s$ copies, from each of which descends an isomorphic copy of its descendant subtree $B(s)$ (colours preserved except for the copies of $s$, which have their branching multiplicity now set from $k_s$ to $1$).}\label{fig:branch}
\end{center}
\end{figure}

If $s,t$ do not belong to a common subtree of $T^{(\mathbf{k})}$ then clearly it makes no difference to the resulting coloured rooted tree if we first branch at $s$ and then at $t$ or branch in the reverse order.
When $s\in P(t)$ or $t\in P(s)$ we require the following:
\begin{lem}\label{lem:branching_order}
Let $T$ be a coloured rooted tree, $\mathbf{k}=(k_s:s\in V(T))$ and $s,t\in V(T)$ distinct vertices of $T^{(\mathbf{k})}$ with colours $(A_s,F_s,k_s)$ and $(A_t,F_t,k_t)$ respectively, and $s\in P(t)$.
Then, the coloured rooted tree obtained from $T^{(\mathbf{k})}$ by first branching at $s$ and then branching at each of the $k_s$ copies of $t$ is isomorphic to the coloured rooted tree obtained by performing the branchings in the reverse order.
\end{lem}
\begin{proof}
Suppose we first branch at $s$, producing $k_s$ isomorphic copies of subtree $B(s)$ pendant from $p(s)$, with the $k_s$ copies of $s$ recoloured $(A_s,F_s,1)$. In particular, this produces $k_s$ copies of $B(t)$, each rooted by a copy of vertex $t$ with colour $(A_t,F_t,k_t)$, and each pendant from its own copy of $p(t)$. Carrying out branching at each copy of $t$ then produces $k_t$ isomorphic copies of $B(t)$ pendant from each corresponding copy of $p(t)$, with the $k_sk_t$ copies of vertex $t$ now with colour $(A_t,F_t,1)$.

Performing the branching in the reverse order first produces $k_t$ copies of $B(t)$ each rooted by a copy of $t$, all of which are pendant from $p(t)$ and coloured $(A_t,F_t,1)$, and then branching on $s$ produces $k_s$ copies of the augmented subtree $B(s)$ (extra copies of $B(t)$ pendant from $p(t)$), with the $k_s$ copies of $s$ now coloured $(A_s,F_s,1)$.

Either way, there are $k_sk_t$ copies of $B(t)$, with $k_t$ copies of $B(t)$ pendant from $p(t)$ in each of $k_s$ copies of $B(s)$ pendant from $p(s)$.  The remainder of the tree, induced on $V(T)\setminus B(s)$, remains unchanged.
\end{proof}

With Lemma~\ref{lem:branching_order} in hand, we can now unambiguously make the following definition:

\begin{definition}\label{def:augment}
Let $T$ be a coloured rooted tree and $\mathbf{k}=(k_s:s\in V(T))$ a tuple of positive integers indexed by the vertices of $T$. The {\em $\mathbf{k}$-branching} of $T$ is the coloured rooted tree $T^{\mathbf{k}}$ obtained from the augmented coloured rooted tree $T^{(\mathbf{k})}$ by successively branching at all vertices whose branching multiplicity exceeds $1$ (taking vertices in an arbitrary order).
\end{definition}

\begin{figure}[htp]
\begin{center}
\scalebox{0.8} 
{
\begin{pspicture}(0,-3.4784374)(12.23,3.4484375)
\psbezier[linewidth=0.02](10.346533,0.10490524)(10.56995,0.12204971)(10.706663,-0.31035173)(10.8,-0.8215625)(10.893337,-1.3327733)(10.98,-1.8615625)(10.78,-1.9415625)
\psdots[dotsize=0.24,fillstyle=solid,dotstyle=o](3.0,3.2984376)
\psdots[dotsize=0.24,fillstyle=solid,dotstyle=o](3.0,2.2984376)
\psdots[dotsize=0.24](3.0,1.2984375)
\psline[linewidth=0.02cm](2.98,3.1984375)(3.0,2.4184375)
\psline[linewidth=0.02cm](2.98,2.1784375)(3.0,1.4184375)
\psdots[dotsize=0.24](10.78,1.2984375)
\psline[linewidth=0.02cm](11.7,3.2184374)(11.3,2.3584375)
\psline[linewidth=0.02cm](10.36,2.1784375)(10.74,1.4184375)
\psline[linewidth=0.02cm](11.2,2.1784375)(10.8,1.3184375)
\psdots[dotsize=0.24,fillstyle=solid,dotstyle=o](11.78,3.3184376)
\psline[linewidth=0.02cm](11.1,3.2584374)(11.24,2.3784375)
\psdots[dotsize=0.24,fillstyle=solid,dotstyle=o](11.08,3.3184376)
\psdots[dotsize=0.24,fillstyle=solid,dotstyle=o](11.26,2.2784376)
\psline[linewidth=0.02cm](10.5,3.2784376)(10.3,2.3784375)
\psline[linewidth=0.02cm](9.76,3.2784376)(10.26,2.3584375)
\psdots[dotsize=0.24,fillstyle=solid,dotstyle=o](10.5,3.3184376)
\psdots[dotsize=0.24,fillstyle=solid,dotstyle=o](9.72,3.3184376)
\psdots[dotsize=0.24,fillstyle=solid,dotstyle=o](10.3,2.2784376)
\psdots[dotsize=0.24](5.02,1.2984375)
\psline[linewidth=0.02cm](5.02,2.1384375)(5.02,1.3984375)
\psline[linewidth=0.02cm](5.44,3.2184374)(5.06,2.3184376)
\psline[linewidth=0.02cm](4.5,3.2384374)(5.0,2.3184376)
\psdots[dotsize=0.24,fillstyle=solid,dotstyle=o](5.48,3.2984376)
\psdots[dotsize=0.24,fillstyle=solid,dotstyle=o](4.46,3.2984376)
\psdots[dotsize=0.24,fillstyle=solid,dotstyle=o](5.04,2.2584374)
\psdots[dotsize=0.24](7.78,1.2984375)
\psline[linewidth=0.02cm](8.18,2.2584374)(7.8,1.3584375)
\psline[linewidth=0.02cm](7.24,2.2784376)(7.74,1.3584375)
\psline[linewidth=0.02cm](6.7,3.2384374)(7.2,2.3184376)
\psdots[dotsize=0.24,fillstyle=solid,dotstyle=o](6.66,3.2984376)
\psdots[dotsize=0.24,fillstyle=solid,dotstyle=o](7.2,2.3384376)
\psline[linewidth=0.02cm](8.62,3.2384374)(8.24,2.3384376)
\psdots[dotsize=0.24,fillstyle=solid,dotstyle=o](8.66,3.3184376)
\psdots[dotsize=0.24,fillstyle=solid,dotstyle=o](8.22,2.3384376)
\psdots[dotsize=0.24,fillstyle=solid,dotstyle=o](3.02,-0.9215625)
\psdots[dotsize=0.24](3.02,-1.9215626)
\psline[linewidth=0.02cm](3.0,-0.0215625)(3.02,-0.8015625)
\psline[linewidth=0.02cm,linestyle=dotted,dotsep=0.16cm](3.0,-1.0415626)(3.02,-1.8015625)
\psbezier[linewidth=0.02](3.0016046,0.06680875)(2.52,-0.0215625)(2.315606,-0.3615625)(2.34,-1.0215625)(2.364394,-1.6815625)(2.7,-1.7615625)(3.08,-1.9701579)
\psdots[dotsize=0.24,fillstyle=solid,dotstyle=o](3.02,0.0784375)
\psdots[dotsize=0.24](5.08,-1.9215626)
\psline[linewidth=0.02cm](4.82,-0.0015625)(5.08,-0.8015625)
\psline[linewidth=0.02cm,linestyle=dotted,dotsep=0.16cm](5.08,-1.1015625)(5.1,-1.8615625)
\psbezier[linewidth=0.02](4.7521915,0.0384375)(4.42,-0.2015625)(4.32,-0.7064177)(4.4382515,-1.208674)(4.556503,-1.7109302)(4.789126,-1.7517709)(5.14,-1.9615625)
\psdots[dotsize=0.24,fillstyle=solid,dotstyle=o](4.8,0.0784375)
\psline[linewidth=0.02cm](5.54,-0.0015625)(5.16,-0.9015625)
\psbezier[linewidth=0.02](5.62,0.0384375)(5.82,-0.06260417)(5.88,-0.6615625)(5.8,-1.1336459)(5.72,-1.6057292)(5.34,-1.8415625)(5.1,-1.9015625)
\psdots[dotsize=0.24,fillstyle=solid,dotstyle=o](5.1,-0.9215625)
\psdots[dotsize=0.24,fillstyle=solid,dotstyle=o](5.58,0.0784375)
\psdots[dotsize=0.24](7.8,-1.9215626)
\psline[linewidth=0.02cm,linestyle=dotted,dotsep=0.16cm](8.2,-0.9615625)(7.82,-1.8615625)
\psline[linewidth=0.02cm,linestyle=dotted,dotsep=0.16cm](7.26,-0.9415625)(7.76,-1.8615625)
\psline[linewidth=0.02cm](6.72,0.0184375)(7.22,-0.9015625)
\psdots[dotsize=0.24,fillstyle=solid,dotstyle=o](6.68,0.0784375)
\psdots[dotsize=0.24,fillstyle=solid,dotstyle=o](7.22,-0.8815625)
\psline[linewidth=0.02cm](8.64,0.0184375)(8.26,-0.8815625)
\psdots[dotsize=0.24,fillstyle=solid,dotstyle=o](8.24,-0.8815625)
\psbezier[linewidth=0.02](6.5921917,-0.0215625)(6.26,-0.2615625)(6.44,-0.86641765)(6.72,-1.3415625)(7.0,-1.8167074)(7.38,-1.9415625)(7.78,-1.9615625)
\psbezier[linewidth=0.02](8.72,0.0184375)(8.92,-0.08260417)(9.08,-0.6615625)(8.9,-1.1536459)(8.72,-1.6457292)(8.14,-1.9015625)(7.86,-1.9415625)
\psdots[dotsize=0.24](10.88,-1.9415625)
\psline[linewidth=0.02cm](11.96,-0.0215625)(11.56,-0.8815625)
\psline[linewidth=0.02cm,linestyle=dotted,dotsep=0.16cm](10.28,-1.0415626)(10.66,-1.8015625)
\psline[linewidth=0.02cm,linestyle=dotted,dotsep=0.16cm](11.48,-1.0615625)(10.9,-1.9215626)
\psline[linewidth=0.02cm](11.36,0.0184375)(11.5,-0.8615625)
\psdots[dotsize=0.24,fillstyle=solid,dotstyle=o](11.34,0.0784375)
\psdots[dotsize=0.24,fillstyle=solid,dotstyle=o](11.5,-0.9415625)
\psline[linewidth=0.02cm](10.42,0.0584375)(10.22,-0.8415625)
\psline[linewidth=0.02cm](9.68,0.0584375)(10.18,-0.8615625)
\psdots[dotsize=0.24,fillstyle=solid,dotstyle=o](10.42,0.0984375)
\psdots[dotsize=0.24,fillstyle=solid,dotstyle=o](10.22,-0.9415625)
\psbezier[linewidth=0.02](9.58,0.0184375)(9.46,-0.3415625)(9.48,-0.8615625)(9.74,-1.3015625)(10.0,-1.7415625)(10.4,-1.8615625)(10.8,-1.9215626)
\psbezier[linewidth=0.02](12.0,0.0184375)(12.2,-0.08260417)(12.22,-0.6615625)(12.0,-1.1815625)(11.78,-1.7015625)(11.22,-1.9015625)(10.94,-1.9415625)
\psdots[dotsize=0.24,fillstyle=solid,dotstyle=o](12.04,0.0784375)
\psdots[dotsize=0.24,fillstyle=solid,dotstyle=o](9.64,0.0984375)
\psdots[dotsize=0.24,fillstyle=solid,dotstyle=o](8.68,0.0984375)
\psbezier[linewidth=0.02](11.3,-0.0215625)(11.192941,-0.3615625)(11.06,-0.8215625)(11.0,-1.1415625)(10.94,-1.4615625)(10.88,-1.8015625)(10.964705,-1.9815625)
\psdots[dotsize=0.24,fillstyle=solid,dotstyle=o](0.14,3.2384374)
\psdots[dotsize=0.24,fillstyle=solid,dotstyle=o](0.14,2.2384374)
\psdots[dotsize=0.24](0.14,1.2384375)
\psline[linewidth=0.02cm](0.12,3.1384375)(0.14,2.3584375)
\psline[linewidth=0.02cm](0.12,2.1184375)(0.14,1.3584375)
\psdots[dotsize=0.24,fillstyle=solid,dotstyle=o](0.12,0.0384375)
\psdots[dotsize=0.24,fillstyle=solid,dotstyle=o](0.12,-0.9615625)
\psdots[dotsize=0.24](0.12,-1.9615625)
\psline[linewidth=0.02cm](0.1,-0.0615625)(0.12,-0.8415625)
\psline[linewidth=0.02cm](0.1,-1.0815625)(0.12,-1.8415625)
\usefont{T1}{ptm}{m}{n}
\rput(0.6196875,3.2084374){$\{1\}$}
\usefont{T1}{ptm}{m}{n}
\rput(0.61,2.2284374){$\{0\}$}
\usefont{T1}{ptm}{m}{n}
\rput(0.59703124,1.2484375){$\emptyset$}
\usefont{T1}{ptm}{m}{n}
\rput(0.7,0.0284375){$\{0,1\}$}
\usefont{T1}{ptm}{m}{n}
\rput(0.6170313,-0.9715625){$\emptyset$}
\usefont{T1}{ptm}{m}{n}
\rput(0.59703124,-1.9515625){$\emptyset$}
\usefont{T1}{ptm}{m}{n}
\rput(3.0051563,-3.2515626){(i)}
\usefont{T1}{ptm}{m}{n}
\rput(5.1251564,-3.2115624){(ii)}
\usefont{T1}{ptm}{m}{n}
\rput(7.7851562,-3.1915624){(iii)}
\usefont{T1}{ptm}{m}{n}
\rput(10.875156,-3.1915624){(iv)}
\end{pspicture}
}

\end{center}
\caption{{\small The path $P_3$ of length $2$ embedded in two different ways in the closure of a tree (also a path of length $2$ rooted at an endpoint).
Left-most, the tree is coloured with subsets of $\{0,1\}$ indicating for each vertex to which of its predecessors it is joined in order to make the subgraph of its closure in (i). These colours are propagated in branching on vertices of the coloured roted tree, so that we obtain the subgraphs shown to the right upon (ii) branching 2-fold at the leaf only, (iii)  branching 2-fold at the central vertex only, and (iv) branching $2$-fold at both non-root vertices.} }\label{fig:branching}
\end{figure}
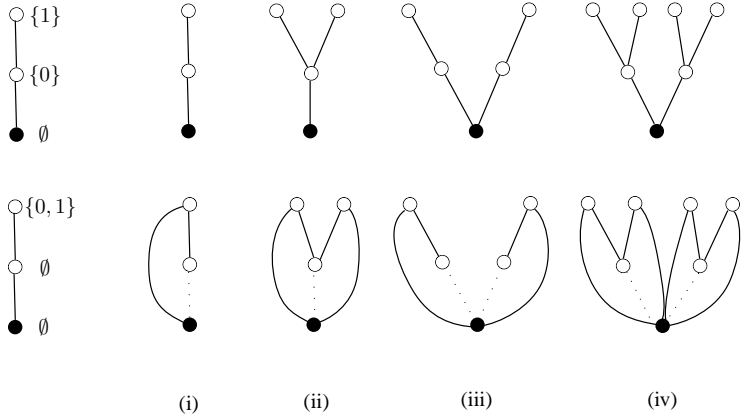

A vertex $s\in V(T)$ has $\prod_{t\in P(s)}k_t$ copies in $T^{\mathbf{k}}$, all of the same colour $(A_s,F_s)$.
We can now use the set $A_s$ given in this colour to construct a spanning subgraph of $T^{\mathbf{k}}$, which shall be denoted by $T^{\mathbf{k}}(H)$. (For each copy of vertex $s$, send an edge joining it to its predecessors at level $i$ in $T^{\mathbf{k}}$ for $i\in A_s$.)
See Figures~\ref{fig:branching}, \ref{fig:branchingC4} and~\ref{fig:branchingK13} for some examples.

The second component of colour $(A_s,F_s)$ on the copies of $s$ in $T^{\mathbf{k}}(H)$ tells us to ornament each of these copies of $s$ with the graph $F_s$. Upon composition we then obtain the graph $T^{\mathbf{k}}(H)[\{F_s:s\in V(H)\}]$, in which the $\prod_{t\in P(s)}k_t$ copies of vertex $s$ in $T^{\mathbf{k}}(H)$ all have the ornament $F_s$. We call this way of producing a sequence from an ornamented graph $H$ {\em branched composition}.

We are now ready to state the main theorem of this section:

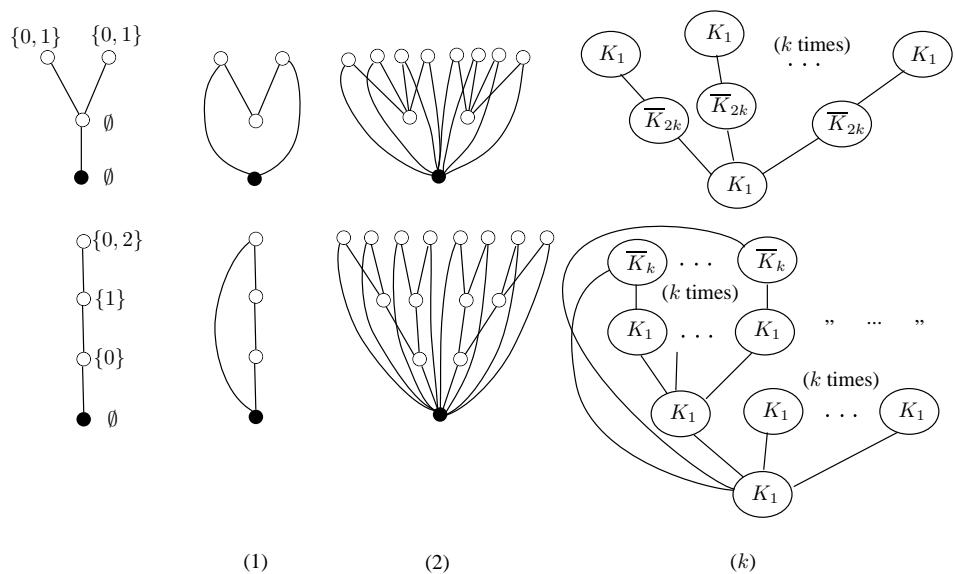
\begin{figure}
\begin{center}
\scalebox{0.8} 
{
\begin{pspicture}(0,-4.7384377)(15.550312,4.6984377)
\psline[linewidth=0.02cm](7.3303127,-1.1815625)(7.0503125,-2.0215626)
\psbezier[linewidth=0.02](6.8703127,0.8584375)(6.9103127,0.5384375)(6.8503127,-0.2415625)(6.8903127,-0.7615625)(6.9303126,-1.2815624)(6.9903126,-1.7215625)(6.9903126,-1.9815625)
\psline[linewidth=0.02cm](5.4503126,0.8384375)(6.0303125,-0.0415625)
\psline[linewidth=0.02cm](5.8903127,0.7984375)(6.1103125,-0.0815625)
\psline[linewidth=0.02cm](6.3903127,0.7984375)(6.5703125,-0.1015625)
\psline[linewidth=0.02cm](6.8303127,0.7784375)(6.6303124,-0.1215625)
\psline[linewidth=0.02cm](7.3703127,0.8584375)(7.4503126,-0.1415625)
\psline[linewidth=0.02cm](7.8103123,0.8384375)(7.4503126,-0.1415625)
\psline[linewidth=0.02cm](8.290313,0.7784375)(8.030313,-0.1015625)
\psline[linewidth=0.02cm](8.770312,0.8384375)(8.030313,-0.1215625)
\psbezier[linewidth=0.02](5.3703127,0.8384375)(5.2503123,0.4784375)(5.4503126,-0.2815625)(5.6903124,-0.7015625)(5.9303126,-1.1215625)(6.5903125,-1.9215626)(6.9703126,-1.9615625)
\psbezier[linewidth=0.02](5.8103123,0.8384375)(5.6903124,0.4784375)(5.7103124,-0.2815625)(5.9303126,-0.6815625)(6.1503124,-1.0815625)(6.6703124,-1.9615625)(7.0303125,-1.9615625)
\psbezier[linewidth=0.02](6.3303127,0.8584375)(6.2103124,0.4984375)(6.2303123,-0.5615625)(6.4303126,-1.0215625)(6.6303124,-1.4815625)(6.7703123,-1.9815625)(7.1103125,-2.0215626)
\psbezier[linewidth=0.02](7.3303127,0.8984375)(7.1503124,0.5984375)(7.0703125,-0.2015625)(7.0503125,-0.6615625)(7.0303125,-1.1215625)(7.0503125,-1.7415625)(7.0303125,-2.0215626)
\psbezier[linewidth=0.02](7.8303127,0.8384375)(7.8103123,0.4384375)(7.7425075,-0.44909674)(7.6103125,-0.9615625)(7.4781175,-1.4740282)(7.2303123,-1.8015625)(7.1103125,-2.0215626)
\psbezier[linewidth=0.02](8.330313,0.8184375)(8.31064,0.4103294)(8.203947,-0.5151954)(7.9703126,-1.0015625)(7.736678,-1.4879296)(7.3903127,-1.9615625)(7.0703125,-2.0415626)
\psbezier[linewidth=0.02](8.850312,0.8184375)(8.830641,0.4103294)(8.483947,-0.45519543)(8.230312,-0.9415625)(7.976678,-1.4279296)(7.4103127,-2.0215626)(7.0103126,-2.0615625)
\psline[linewidth=0.02cm](6.6303124,-0.2015625)(6.6903124,-1.0815625)
\psline[linewidth=0.02cm](7.9703126,-0.2215625)(7.4103127,-1.0815625)
\psline[linewidth=0.02cm](7.4503126,-0.2215625)(7.3303127,-1.0815625)
\psline[linewidth=0.02cm](6.1303124,-0.2215625)(6.6503124,-1.1215625)
\psline[linewidth=0.02cm](7.6503124,3.8784375)(7.4903126,2.8984375)
\psbezier[linewidth=0.02](5.4824576,3.8184376)(5.2031765,3.5304618)(5.5503125,2.9784374)(5.8303127,2.6384375)(6.1103125,2.2984376)(6.69875,1.9589301)(6.9103127,1.9584374)
\psline[linewidth=0.02cm](1.0903125,0.7584375)(1.1103125,-0.0215625)
\psdots[dotsize=0.24](7.0103126,1.9384375)
\psline[linewidth=0.02cm](7.9303126,3.8584375)(7.5303125,2.9984374)
\psline[linewidth=0.02cm](7.3303127,3.8984375)(7.4703126,3.0184374)
\psdots[dotsize=0.24,fillstyle=solid,dotstyle=o](7.3103123,3.9584374)
\psline[linewidth=0.02cm](6.7703123,3.8584375)(6.5303125,3.0184374)
\psline[linewidth=0.02cm](5.9903126,3.9184375)(6.4903126,2.9984374)
\psdots[dotsize=0.24,fillstyle=solid,dotstyle=o](5.9903126,3.9384375)
\psdots[dotsize=0.24](1.0703125,1.9184375)
\psline[linewidth=0.02cm](1.0703125,2.7584374)(1.0703125,2.0184374)
\psline[linewidth=0.02cm](1.4903125,3.8384376)(1.1103125,2.9384375)
\psline[linewidth=0.02cm](0.5503125,3.8584375)(1.0503125,2.9384375)
\psdots[dotsize=0.24,fillstyle=solid,dotstyle=o](1.5303125,3.9184375)
\psdots[dotsize=0.24,fillstyle=solid,dotstyle=o](0.5103125,3.9184375)
\psdots[dotsize=0.24,fillstyle=solid,dotstyle=o](1.0903125,2.8784375)
\psdots[dotsize=0.24,fillstyle=solid,dotstyle=o](1.1103125,-0.1015625)
\psdots[dotsize=0.24,fillstyle=solid,dotstyle=o](1.1103125,-1.1015625)
\psdots[dotsize=0.24](1.1103125,-2.1015625)
\psline[linewidth=0.02cm](1.0903125,-0.2015625)(1.1103125,-0.9815625)
\psline[linewidth=0.02cm](1.0903125,-1.2215625)(1.1103125,-1.9815625)
\usefont{T1}{ptm}{m}{n}
\rput(1.65,4.2684374){$\{0,1\}$}
\usefont{T1}{ptm}{m}{n}
\rput(1.5273438,2.8284376){$\emptyset$}
\usefont{T1}{ptm}{m}{n}
\rput(1.5273438,1.9284375){$\emptyset$}
\usefont{T1}{ptm}{m}{n}
\rput(1.55,-0.1115625){$\{1\}$}
\usefont{T1}{ptm}{m}{n}
\rput(1.53,-1.0915625){$\{0\}$}
\usefont{T1}{ptm}{m}{n}
\rput(1.5873437,-2.0915625){$\emptyset$}
\usefont{T1}{ptm}{m}{n}
\rput(3.9654686,-4.4915624){(1)}
\usefont{T1}{ptm}{m}{n}
\rput(6.985469,-4.5115623){(2)}
\usefont{T1}{ptm}{m}{n}
\rput(12.065469,-4.4915624){($k$)}
\usefont{T1}{ptm}{m}{n}
\rput(0.31,4.2484374){$\{0,1\}$}
\psdots[dotsize=0.24,fillstyle=solid,dotstyle=o](1.1103125,0.8584375)
\usefont{T1}{ptm}{m}{n}
\rput(1.67,0.8684375){$\{0,2\}$}
\psdots[dotsize=0.24](3.9503126,1.8984375)
\psline[linewidth=0.02cm](4.3703127,3.8184376)(3.9903126,2.9184375)
\psline[linewidth=0.02cm](3.4303124,3.8384376)(3.9303124,2.9184375)
\psdots[dotsize=0.24,fillstyle=solid,dotstyle=o](3.9703126,2.8584375)
\psbezier[linewidth=0.02](3.4424577,3.8584375)(3.1631765,3.5704618)(3.0703125,3.3233202)(3.1337287,2.763364)(3.197145,2.2034078)(3.6587498,2.0389302)(3.9303124,1.9584374)
\psdots[dotsize=0.24,fillstyle=solid,dotstyle=o](3.3903124,3.8984375)
\psbezier[linewidth=0.02](4.4903126,3.8784375)(4.6903124,3.7773957)(4.7503123,3.1784375)(4.6703124,2.7063541)(4.5903125,2.2342708)(4.2103124,1.9984375)(3.9703126,1.9384375)
\psdots[dotsize=0.24,fillstyle=solid,dotstyle=o](4.4103127,3.8984375)
\psline[linewidth=0.02cm](3.9503126,0.7984375)(3.9703126,0.0184375)
\psdots[dotsize=0.24,fillstyle=solid,dotstyle=o](3.9703126,-0.0615625)
\psdots[dotsize=0.24,fillstyle=solid,dotstyle=o](3.9703126,-1.0615625)
\psdots[dotsize=0.24](3.9703126,-2.0615625)
\psline[linewidth=0.02cm](3.9503126,-0.1615625)(3.9703126,-0.9415625)
\psline[linewidth=0.02cm](3.9503126,-1.1815625)(3.9703126,-1.9415625)
\psbezier[linewidth=0.02](3.882504,0.8584375)(3.5503125,0.6184375)(3.2903125,-0.08641765)(3.2903125,-0.5815625)(3.2903125,-1.0767074)(3.4903126,-1.7215625)(3.9103124,-1.9615625)
\psdots[dotsize=0.24,fillstyle=solid,dotstyle=o](3.9703126,0.8984375)
\psline[linewidth=0.02cm](5.5303125,3.8584375)(6.4703126,2.9384375)
\psdots[dotsize=0.24,fillstyle=solid,dotstyle=o](5.5103126,3.8784375)
\psline[linewidth=0.02cm](6.4103127,3.7984376)(6.5303125,2.8784375)
\psdots[dotsize=0.24,fillstyle=solid,dotstyle=o](6.5303125,2.9184375)
\psline[linewidth=0.02cm](8.390312,3.8384376)(7.5303125,2.9384375)
\psdots[dotsize=0.24,fillstyle=solid,dotstyle=o](7.4903126,2.9184375)
\psbezier[linewidth=0.02](5.9224577,3.8184376)(5.6431766,3.5304618)(5.9903126,2.9784374)(6.2703123,2.6384375)(6.5503125,2.2984376)(6.75875,1.9789301)(7.0103126,1.9784375)
\psbezier[linewidth=0.02](6.4503126,3.8384376)(6.58926,3.615677)(6.7415624,3.1011457)(6.8503127,2.7984376)(6.9590626,2.4957292)(6.909219,1.9764382)(7.0903125,1.9584374)
\psbezier[linewidth=0.02](6.8503127,3.8584375)(6.8903127,3.5784376)(7.0103126,2.9984374)(7.0303125,2.4184375)(7.0503125,1.8384376)(6.81875,1.9189302)(7.0303125,1.9184375)
\psdots[dotsize=0.24,fillstyle=solid,dotstyle=o](6.3903127,3.9384375)
\psdots[dotsize=0.24,fillstyle=solid,dotstyle=o](6.8303127,3.9384375)
\psbezier[linewidth=0.02](7.2503123,3.8784375)(7.1703124,3.5784376)(7.0703125,2.9784374)(7.0503125,2.5184374)(7.0303125,2.0584376)(6.9503126,2.0984375)(7.0503125,1.8784375)
\psbezier[linewidth=0.02](7.5903125,3.8984375)(7.5103126,3.5984375)(7.2703123,3.0584376)(7.1903124,2.6184375)(7.1103125,2.1784375)(7.0503125,2.0984375)(7.1103125,1.9384375)
\psdots[dotsize=0.24,fillstyle=solid,dotstyle=o](7.6703124,3.9584374)
\psbezier[linewidth=0.02](8.010312,3.8784375)(7.9421644,3.5631313)(7.79772,2.9555805)(7.6103125,2.5984375)(7.422905,2.2412946)(7.0503125,2.0666008)(7.0703125,1.8784375)
\psdots[dotsize=0.24,fillstyle=solid,dotstyle=o](8.010312,3.9584374)
\psbezier[linewidth=0.02](8.430312,3.8584375)(8.361162,3.5431314)(8.394593,3.0755804)(8.02443,2.5784376)(7.654267,2.0812945)(7.0503125,2.0466008)(7.0503125,1.8784375)
\psdots[dotsize=0.24,fillstyle=solid,dotstyle=o](8.410313,3.9184375)
\psdots[dotsize=0.24,fillstyle=solid,dotstyle=o](6.6303124,-0.1215625)
\psdots[dotsize=0.24,fillstyle=solid,dotstyle=o](6.6903124,-1.1015625)
\psdots[dotsize=0.24](7.0303125,-2.0215626)
\psdots[dotsize=0.24,fillstyle=solid,dotstyle=o](6.8703127,0.9184375)
\psdots[dotsize=0.24,fillstyle=solid,dotstyle=o](7.3703127,-1.1015625)
\psdots[dotsize=0.24,fillstyle=solid,dotstyle=o](6.0903125,-0.1215625)
\psdots[dotsize=0.24,fillstyle=solid,dotstyle=o](8.010312,-0.1215625)
\psdots[dotsize=0.24,fillstyle=solid,dotstyle=o](7.4703126,-0.1215625)
\psdots[dotsize=0.24,fillstyle=solid,dotstyle=o](5.4303126,0.9184375)
\psdots[dotsize=0.24,fillstyle=solid,dotstyle=o](5.8903127,0.9184375)
\psdots[dotsize=0.24,fillstyle=solid,dotstyle=o](6.3903127,0.9184375)
\psdots[dotsize=0.24,fillstyle=solid,dotstyle=o](8.810312,0.9184375)
\psdots[dotsize=0.24,fillstyle=solid,dotstyle=o](8.330313,0.9184375)
\psline[linewidth=0.02cm](6.7303123,-1.2215625)(7.0103126,-1.9815625)
\psdots[dotsize=0.24,fillstyle=solid,dotstyle=o](7.3703127,0.9184375)
\psdots[dotsize=0.24,fillstyle=solid,dotstyle=o](7.8303127,0.9184375)
\usefont{T1}{ptm}{m}{n}
\rput(12.015312,1.7884375){$K_1$}
\usefont{T1}{ptm}{m}{n}
\rput(10.748438,2.8884375){$\overline{K}_{2k}$}
\usefont{T1}{ptm}{m}{n}
\rput(11.848437,3.1284375){$\overline{K}_{2k}$}
\usefont{T1}{ptm}{m}{n}
\rput(13.788438,2.8284376){$\overline{K}_{2k}$}
\usefont{T1}{ptm}{m}{n}
\rput(9.895312,3.9684374){$K_1$}
\usefont{T1}{ptm}{m}{n}
\rput(11.675312,4.3084373){$K_1$}
\usefont{T1}{ptm}{m}{n}
\rput(15.075313,3.9484375){$K_1$}
\psline[linewidth=0.02cm](13.330313,2.5784376)(12.390312,1.9784375)
\psellipse[linewidth=0.02,dimen=outer](11.970312,1.7884375)(0.5,0.39)
\psellipse[linewidth=0.02,dimen=outer](10.670313,2.8684375)(0.5,0.39)
\psellipse[linewidth=0.02,dimen=outer](11.790313,3.1084375)(0.5,0.39)
\psellipse[linewidth=0.02,dimen=outer](13.710313,2.8084376)(0.5,0.39)
\psellipse[linewidth=0.02,dimen=outer](9.870313,3.9684374)(0.5,0.39)
\psellipse[linewidth=0.02,dimen=outer](11.590313,4.3084373)(0.5,0.39)
\psellipse[linewidth=0.02,dimen=outer](15.050312,3.9684374)(0.5,0.39)
\psline[linewidth=0.02cm](14.750313,3.6984375)(13.970312,3.1384375)
\psline[linewidth=0.02cm](10.030313,3.6384375)(10.430312,3.1984375)
\psline[linewidth=0.02cm](10.990313,2.5784376)(11.550312,1.9584374)
\psline[linewidth=0.02cm](11.630313,3.9584374)(11.710313,3.4784374)
\psline[linewidth=0.02cm](11.810312,2.6984375)(11.910313,2.1984375)
\psdots[dotsize=0.048](12.850312,3.7984376)
\psdots[dotsize=0.048](13.090313,3.7984376)
\psdots[dotsize=0.048](13.310312,3.7984376)
\usefont{T1}{ptm}{m}{n}
\rput(12.435312,-3.3515625){$K_1$}
\psellipse[linewidth=0.02,dimen=outer](12.390312,-3.3515625)(0.5,0.39)
\psline[linewidth=0.02cm](12.470312,-2.3215625)(12.410313,-2.9415624)
\usefont{T1}{ptm}{m}{n}
\rput(11.075313,-2.0115626){$K_1$}
\psellipse[linewidth=0.02,dimen=outer](11.030313,-2.0115626)(0.5,0.39)
\usefont{T1}{ptm}{m}{n}
\rput(12.615313,-1.9715625){$K_1$}
\psellipse[linewidth=0.02,dimen=outer](12.5703125,-1.9715625)(0.5,0.39)
\usefont{T1}{ptm}{m}{n}
\rput(14.875313,-1.9715625){$K_1$}
\psellipse[linewidth=0.02,dimen=outer](14.830313,-1.9715625)(0.5,0.39)
\psbezier[linewidth=0.02](9.850312,0.3784375)(9.450313,0.1784375)(9.050312,-0.2015625)(9.290313,-1.5015625)(9.530313,-2.8015625)(11.5703125,-3.2815626)(11.910313,-3.2815626)
\psline[linewidth=0.02cm](10.970312,-0.9815625)(10.950313,-1.6215625)
\usefont{T1}{ptm}{m}{n}
\rput(10.378437,0.5084375){$\overline{K}_k$}
\psellipse[linewidth=0.02,dimen=outer](10.310312,0.5084375)(0.5,0.39)
\usefont{T1}{ptm}{m}{n}
\rput(10.355312,-0.6315625){$K_1$}
\psellipse[linewidth=0.02,dimen=outer](10.310312,-0.6515625)(0.5,0.39)
\psdots[dotsize=0.048](13.470312,-2.0415626)
\psdots[dotsize=0.048](13.710313,-2.0415626)
\psdots[dotsize=0.048](13.930312,-2.0415626)
\psline[linewidth=0.02cm](11.250313,-2.3615625)(12.0703125,-3.0615625)
\psline[linewidth=0.02cm](14.630313,-2.3215625)(12.910313,-3.2215624)
\usefont{T1}{ptm}{m}{n}
\rput(13.215468,4.1084375){($k$ times)}
\usefont{T1}{ptm}{m}{n}
\rput(13.695469,-1.4515625){($k$ times)}
\psline[linewidth=0.02cm](10.370313,-1.0415626)(10.810312,-1.6815625)
\psline[linewidth=0.02cm](12.210313,-1.0015625)(11.450313,-1.7615625)
\usefont{T1}{ptm}{m}{n}
\rput(12.475312,-0.6315625){$K_1$}
\psellipse[linewidth=0.02,dimen=outer](12.430312,-0.6515625)(0.5,0.39)
\usefont{T1}{ptm}{m}{n}
\rput(11.355469,-0.0115625){($k$ times)}
\psdots[dotsize=0.048](11.0703125,-0.7215625)
\psdots[dotsize=0.048](11.310312,-0.7215625)
\psdots[dotsize=0.048](11.530313,-0.7215625)
\psdots[dotsize=0.048](11.050312,0.4784375)
\psdots[dotsize=0.048](11.290313,0.4784375)
\psdots[dotsize=0.048](11.510312,0.4784375)
\usefont{T1}{ptm}{m}{n}
\rput(12.538438,0.5484375){$\overline{K}_k$}
\psellipse[linewidth=0.02,dimen=outer](12.470312,0.5484375)(0.5,0.39)
\psline[linewidth=0.02cm](10.290313,0.1384375)(10.290313,-0.2815625)
\psline[linewidth=0.02cm](12.470312,0.1384375)(12.470312,-0.2615625)
\rput(15,-.5){,,}
\rput(13.5,-.5){,,}
\rput(14.25,-.5){...}
\psbezier[linewidth=0.02](12.090313,0.7584375)(11.970312,1.0984375)(9.170313,1.5584375)(9.0703125,0.1184375)(8.970312,-1.3215625)(11.610312,-3.2615626)(11.930312,-3.2615626)
\end{pspicture}
}

\end{center}
\caption{{\small Two embeddings of $C_4$ in a rooted tree, and the result of branching $2$-fold and $k$-fold at each non-root vertex. In the rightmost graphs the lines represent graph joins of the ornament graphs on the vertices. The top-right graph can be obtained as a composition of an ornamented graph (the star $K_{1,k}$ with each edge replaced by a path of length two, ornamented with $\overline{K}_{2k}$ on some vertices). For the bottom right graph just one of the $k$ branches is depicted, and that elliptically too; this graph cannot be obtained as a composition of an ornamented graph.}}\label{fig:branchingC4}
\end{figure}


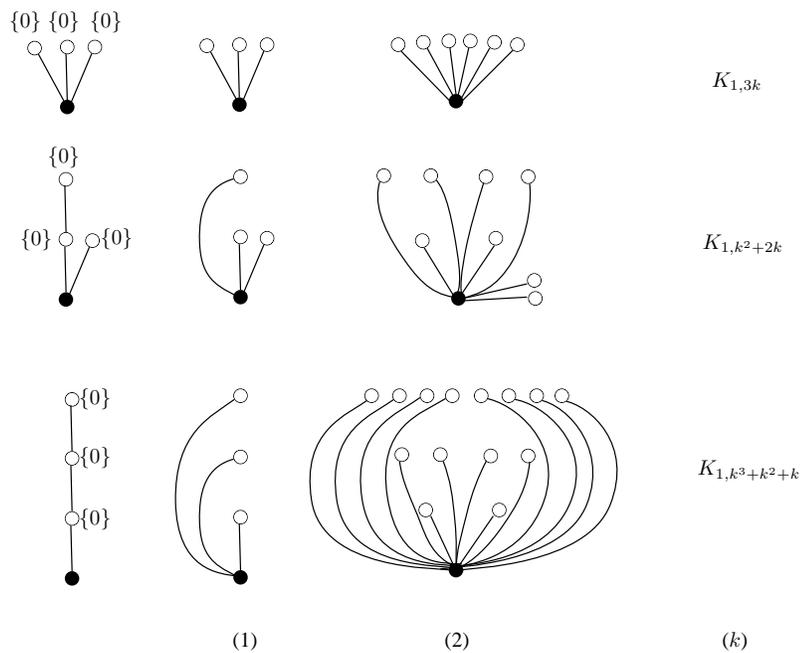
\begin{figure}
\begin{center}
%
\scalebox{0.8} 
{
\begin{pspicture}(0,-5.3534374)(14.692187,5.3534374)
\psbezier[linewidth=0.02](6.8368454,2.6299052)(7.060263,2.6470497)(7.1969757,2.2146482)(7.2903123,1.7034374)(7.3836493,1.1922268)(7.4703126,0.6634375)(7.2703123,0.5834375)
\psdots[dotsize=0.24,fillstyle=solid,dotstyle=o](3.7503126,1.6034375)
\psdots[dotsize=0.24](3.7503126,0.6034375)
\psline[linewidth=0.02cm](3.7303126,1.4834375)(3.7503126,0.7234375)
\psbezier[linewidth=0.02](3.7319171,2.5918088)(3.2503126,2.5034375)(3.0459187,2.1634376)(3.0703125,1.5034375)(3.0947063,0.8434375)(3.4303124,0.7634375)(3.8103125,0.5548421)
\psdots[dotsize=0.24,fillstyle=solid,dotstyle=o](3.7503126,2.6034374)
\psdots[dotsize=0.24](7.3703127,0.5834375)
\psline[linewidth=0.02cm](6.7903123,1.4434375)(7.3303127,0.5634375)
\psline[linewidth=0.02cm](7.9703126,1.4634376)(7.3903127,0.6034375)
\psdots[dotsize=0.24,fillstyle=solid,dotstyle=o](7.8303127,2.6034374)
\psdots[dotsize=0.24,fillstyle=solid,dotstyle=o](7.9903126,1.5834374)
\psdots[dotsize=0.24,fillstyle=solid,dotstyle=o](6.9103127,2.6234374)
\psdots[dotsize=0.24,fillstyle=solid,dotstyle=o](6.7703123,1.5434375)
\psbezier[linewidth=0.02](6.0703125,2.5434375)(5.9503126,2.1834376)(6.1103125,1.7634375)(6.3903127,1.3434376)(6.6703124,0.9234375)(6.8903127,0.6634375)(7.2903123,0.6034375)
\psbezier[linewidth=0.02](8.550312,2.5234375)(8.590313,2.2634375)(8.551369,1.6859375)(8.276022,1.2134376)(8.000674,0.7409375)(7.70566,0.6534375)(7.4303126,0.5834375)
\psdots[dotsize=0.24,fillstyle=solid,dotstyle=o](8.530313,2.6034374)
\psdots[dotsize=0.24,fillstyle=solid,dotstyle=o](6.1303124,2.6234374)
\psbezier[linewidth=0.02](7.7903123,2.5034375)(7.683254,2.1634376)(7.5503125,1.7034374)(7.4903126,1.3834375)(7.4303126,1.0634375)(7.3703127,0.7234375)(7.4550185,0.5434375)
\psdots[dotsize=0.24,fillstyle=solid,dotstyle=o](0.8703125,4.7634373)
\psdots[dotsize=0.24](0.8703125,3.7634375)
\psline[linewidth=0.02cm](0.8503125,4.6434374)(0.8703125,3.8834374)
\psdots[dotsize=0.24,fillstyle=solid,dotstyle=o](0.8503125,2.5634375)
\psdots[dotsize=0.24,fillstyle=solid,dotstyle=o](0.8503125,1.5634375)
\psdots[dotsize=0.24](0.8503125,0.5634375)
\psline[linewidth=0.02cm](0.8303125,2.4634376)(0.8503125,1.6834375)
\psline[linewidth=0.02cm](0.8303125,1.4434375)(0.8503125,0.6834375)
\usefont{T1}{ptm}{m}{n}
\rput(1.51,5.1734376){$\{0\}$}
\usefont{T1}{ptm}{m}{n}
\rput(0.81,2.9334376){$\{0\}$}
\usefont{T1}{ptm}{m}{n}
\rput(3.8254688,-5.1265626){(1)}
\usefont{T1}{ptm}{m}{n}
\rput(7.3454685,-5.1265626){(2)}
\usefont{T1}{ptm}{m}{n}
\rput(11.965468,-5.1265626){($k$)}
\psline[linewidth=0.02cm](1.2903125,4.6834373)(0.9103125,3.7834375)
\psdots[dotsize=0.24,fillstyle=solid,dotstyle=o](1.3303125,4.7634373)
\psline[linewidth=0.02cm](0.3703125,4.6834373)(0.8703125,3.7634375)
\psdots[dotsize=0.24,fillstyle=solid,dotstyle=o](0.3303125,4.7434373)
\psline[linewidth=0.02cm](1.2503124,1.4634376)(0.8703125,0.5634375)
\psdots[dotsize=0.24,fillstyle=solid,dotstyle=o](1.2903125,1.5434375)
\usefont{T1}{ptm}{m}{n}
\rput(0.17,5.1734376){$\{0\}$}
\usefont{T1}{ptm}{m}{n}
\rput(0.81,5.1734376){$\{0\}$}
\usefont{T1}{ptm}{m}{n}
\rput(1.69,1.5734375){$\{0\}$}
\usefont{T1}{ptm}{m}{n}
\rput(0.35,1.5534375){$\{0\}$}
\psdots[dotsize=0.24,fillstyle=solid,dotstyle=o](0.9503125,-3.0765624)
\psdots[dotsize=0.24](0.9503125,-4.0765624)
\psline[linewidth=0.02cm](0.9303125,-2.1765625)(0.9503125,-2.9565625)
\psline[linewidth=0.02cm](0.9303125,-3.1965625)(0.9503125,-3.9565625)
\usefont{T1}{ptm}{m}{n}
\rput(1.33,-2.0665624){$\{0\}$}
\usefont{T1}{ptm}{m}{n}
\rput(1.33,-3.0865624){$\{0\}$}
\psdots[dotsize=0.24,fillstyle=solid,dotstyle=o](0.9503125,-1.0965625)
\psline[linewidth=0.02cm](0.9303125,-1.1965625)(0.9503125,-1.9765625)
\usefont{T1}{ptm}{m}{n}
\rput(1.33,-1.0865625){$\{0\}$}
\psdots[dotsize=0.24,fillstyle=solid,dotstyle=o](0.9503125,-2.0765624)
\psline[linewidth=0.02cm](4.1503124,1.5034375)(3.7703125,0.6034375)
\psdots[dotsize=0.24,fillstyle=solid,dotstyle=o](4.1903124,1.5834374)
\psdots[dotsize=0.24,fillstyle=solid,dotstyle=o](3.7303126,4.8034377)
\psdots[dotsize=0.24](3.7303126,3.8034375)
\psline[linewidth=0.02cm](3.7103126,4.6834373)(3.7303126,3.9234376)
\psline[linewidth=0.02cm](4.1503124,4.7234373)(3.7703125,3.8234375)
\psdots[dotsize=0.24,fillstyle=solid,dotstyle=o](4.1903124,4.8034377)
\psline[linewidth=0.02cm](3.2303126,4.7234373)(3.7303126,3.8034375)
\psdots[dotsize=0.24,fillstyle=solid,dotstyle=o](3.1903124,4.7834377)
\psdots[dotsize=0.24,fillstyle=solid,dotstyle=o](3.7503126,-3.0565624)
\psdots[dotsize=0.24](3.7503126,-4.0565624)
\psline[linewidth=0.02cm](3.7303126,-3.1765625)(3.7503126,-3.9365625)
\psbezier[linewidth=0.02](3.7319171,-2.0681913)(3.2503126,-2.1565626)(3.0459187,-2.4965625)(3.0703125,-3.1565626)(3.0947063,-3.8165624)(3.4303124,-3.8965626)(3.8103125,-4.105158)
\psdots[dotsize=0.24,fillstyle=solid,dotstyle=o](3.7503126,-2.0565624)
\psbezier[linewidth=0.02](3.6503124,-1.0965625)(3.2703125,-1.3765625)(2.6503124,-1.6365625)(2.6703124,-2.7965624)(2.6903124,-3.9565625)(3.4103124,-3.9965625)(3.6903124,-4.0365624)
\psdots[dotsize=0.24,fillstyle=solid,dotstyle=o](3.7503126,-1.0365624)
\psdots[dotsize=0.24,fillstyle=solid,dotstyle=o](7.2103124,4.8634377)
\psdots[dotsize=0.24](7.3303127,3.8634374)
\psline[linewidth=0.02cm](7.2303123,4.7434373)(7.3303127,3.9834375)
\psline[linewidth=0.02cm](7.8903127,4.7434373)(7.3703127,3.8234375)
\psline[linewidth=0.02cm](6.8303127,4.7834377)(7.3303127,3.8634374)
\psdots[dotsize=0.24,fillstyle=solid,dotstyle=o](6.7903123,4.8434377)
\psline[linewidth=0.02cm](6.3703127,4.7434373)(7.2503123,3.8434374)
\psdots[dotsize=0.24,fillstyle=solid,dotstyle=o](6.3703127,4.8034377)
\psdots[dotsize=0.24,fillstyle=solid,dotstyle=o](7.9703126,4.8434377)
\psline[linewidth=0.02cm](8.290313,4.7234373)(7.3903127,3.8234375)
\psdots[dotsize=0.24,fillstyle=solid,dotstyle=o](8.350312,4.8034377)
\psline[linewidth=0.02cm](7.5503125,4.7834377)(7.3503127,3.8834374)
\psdots[dotsize=0.24,fillstyle=solid,dotstyle=o](7.5703125,4.8634377)
\psline[linewidth=0.02cm](8.510312,0.8234375)(7.3703127,0.5634375)
\psdots[dotsize=0.24,fillstyle=solid,dotstyle=o](8.630313,0.8834375)
\psline[linewidth=0.02cm](8.510312,0.6034375)(7.4503126,0.5434375)
\psdots[dotsize=0.24,fillstyle=solid,dotstyle=o](8.650312,0.5834375)
\psdots[dotsize=0.24](7.3303127,-3.9365625)
\psline[linewidth=0.02cm](6.9103127,-3.0365624)(7.2903123,-3.8565626)
\psbezier[linewidth=0.02](6.391917,-2.1081913)(6.4303126,-2.4765625)(6.5903125,-2.7565625)(6.7503123,-3.1765625)(6.9103127,-3.5965624)(7.0303125,-3.7565625)(7.4503126,-3.9565625)
\psbezier[linewidth=0.02](5.9103127,-1.0765625)(5.5303125,-1.3565625)(4.7503123,-1.5965625)(4.9303126,-2.7765625)(5.1103125,-3.9565625)(7.0103126,-3.8965626)(7.2503123,-3.9365625)
\psline[linewidth=0.02cm](7.9503126,-3.0365624)(7.3303127,-3.9365625)
\psdots[dotsize=0.24,fillstyle=solid,dotstyle=o](6.8303127,-2.9365625)
\psdots[dotsize=0.24,fillstyle=solid,dotstyle=o](7.9103127,-2.0365624)
\psdots[dotsize=0.24,fillstyle=solid,dotstyle=o](8.530313,-2.0365624)
\psdots[dotsize=0.24,fillstyle=solid,dotstyle=o](8.050312,-2.9365625)
\psdots[dotsize=0.24,fillstyle=solid,dotstyle=o](6.4303126,-2.0165625)
\psbezier[linewidth=0.02](7.111917,-2.1081913)(7.2503123,-2.4765625)(7.2703123,-2.6965625)(7.3103123,-3.0965624)(7.3503127,-3.4965625)(7.3303127,-3.6565626)(7.3303127,-3.9765625)
\psdots[dotsize=0.24,fillstyle=solid,dotstyle=o](7.0703125,-2.0165625)
\psbezier[linewidth=0.02](7.3303127,-3.9765625)(7.3360205,-3.5634134)(7.4303126,-3.3365624)(7.5103126,-2.9965625)(7.5903125,-2.6565626)(7.7103124,-2.3565626)(7.8379846,-2.1161916)
\psbezier[linewidth=0.02](7.367453,-3.9034812)(7.6103125,-3.7965624)(7.8748093,-3.5405412)(8.110312,-3.1965625)(8.345816,-2.852584)(8.510312,-2.5765624)(8.556061,-2.1360288)
\psbezier[linewidth=0.02](6.3303127,-1.0765625)(5.9503126,-1.3565625)(5.1703124,-1.5965625)(5.3503127,-2.7765625)(5.5303125,-3.9565625)(7.0903125,-3.8365624)(7.4103127,-3.9165626)
\psdots[dotsize=0.24,fillstyle=solid,dotstyle=o](6.3903127,-1.0365624)
\psdots[dotsize=0.24,fillstyle=solid,dotstyle=o](5.9303126,-1.0365624)
\psbezier[linewidth=0.02](6.7503123,-1.0565625)(6.3703127,-1.3365625)(5.5903125,-1.5765625)(5.7703123,-2.7565625)(5.9503126,-3.9365625)(7.1703124,-3.8365624)(7.4503126,-3.8765626)
\psdots[dotsize=0.24,fillstyle=solid,dotstyle=o](6.8503127,-1.0365624)
\psbezier[linewidth=0.02](7.1703124,-1.0765625)(6.7503123,-1.3365625)(6.0103126,-1.5365624)(6.1903124,-2.7165625)(6.3703127,-3.8965626)(7.1303124,-3.8165624)(7.4303126,-3.8365624)
\psdots[dotsize=0.24,fillstyle=solid,dotstyle=o](7.2703123,-1.0365624)
\psbezier[linewidth=0.02](9.110312,-1.1165625)(9.5703125,-1.1965625)(10.290313,-1.8565625)(9.890312,-2.8565626)(9.490313,-3.8565626)(7.5856595,-3.9065626)(7.3503127,-3.9365625)
\psdots[dotsize=0.24,fillstyle=solid,dotstyle=o](9.090313,-1.0365624)
\psbezier[linewidth=0.02](8.605006,-1.0965625)(9.065414,-1.0765625)(9.890312,-1.8365625)(9.5066395,-2.8365624)(9.122966,-3.8365624)(7.296054,-3.8865626)(7.0703125,-3.9165626)
\psdots[dotsize=0.24,fillstyle=solid,dotstyle=o](8.670313,-1.0365624)
\psbezier[linewidth=0.02](8.290313,-1.0962808)(8.750313,-1.1565624)(9.630313,-1.8258582)(9.230312,-2.8117738)(8.830313,-3.7976892)(7.4856596,-3.846985)(7.3103123,-3.8765626)
\psdots[dotsize=0.24,fillstyle=solid,dotstyle=o](8.210313,-1.0365624)
\psbezier[linewidth=0.02](7.8103123,-1.0765625)(8.290313,-1.1565624)(9.130313,-1.8165625)(8.790313,-2.7917738)(8.450313,-3.766985)(7.4656596,-3.866985)(7.2303123,-3.8965626)
\psdots[dotsize=0.24,fillstyle=solid,dotstyle=o](7.7503123,-1.0365624)
\usefont{T1}{ptm}{m}{n}
\rput(12,4.1734376){$K_{1,3k}$}
\usefont{T1}{ptm}{m}{n}
\rput(12.1,1.4734375){$K_{1,k^2+2k}$}
\usefont{T1}{ptm}{m}{n}
\rput(12.2,-2.2865624){$K_{1,k^3+k^2+k}$}
\end{pspicture}
}

\end{center}
\caption{Starting with the star $K_{1,3}$ in one of its three embeddings as a spanning subgraph of the closure of a rooted tree (encoded in the left-most coloured rooted trees), branching $k$-fold at every non-root vertex yields one of the three polynomial sequences $(K_{1,3k})$, $(K_{1,k^2+k})$, $(K_{k^3+k^2+k})$. }\label{fig:branchingK13}
\end{figure}

\begin{figure}
\begin{center}
\scalebox{0.9} 
{
\begin{pspicture}(0,-3.35)(14.758437,3.35)
\definecolor{color8144b}{rgb}{0.8,0.8,0.8}
\psellipse[linewidth=0.02,dimen=outer,fillstyle=solid,fillcolor=black](1.2,-2.44)(0.2,0.2)
\psellipse[linewidth=0.02,dimen=outer,fillstyle=solid,fillcolor=black](1.2,-0.95)(0.2,0.2)
\psellipse[linewidth=0.02,dimen=outer,fillstyle=solid,fillcolor=black](1.2,0.46)(0.2,0.2)
\psellipse[linewidth=0.02,dimen=outer,fillstyle=solid,fillcolor=black](1.2,1.82)(0.2,0.2)
\usefont{T1}{ptm}{m}{n}
\rput(1.3,2.25){$(\emptyset, K_k)$}
\usefont{T1}{ptm}{m}{n}
\rput(1.3,0.9){$(\emptyset, K_k^\ell)$}
\psline[linewidth=0.02cm](1.4,-0.95)(4,-0.95)
\usefont{T1}{ptm}{m}{n}
\rput(1.5,-0.5){$(\emptyset,K_{k-j}^1)$}
\usefont{T1}{ptm}{m}{n}
\rput(4.3,-0.5){$(\{0\},K_{j}^\ell)$}
\psellipse[linewidth=0.02,dimen=outer](4.2,-0.95)(0.2,0.2)
\psline[linewidth=0.02cm](1.4,-2.44)(4,-2.44)
\usefont{T1}{ptm}{m}{n}
\rput(1.3,-1.95){$(\emptyset,K_1^1)$}
\usefont{T1}{ptm}{m}{n}
\rput(4.3,-1.95){$(\{0\},K_j^1)$}
\psellipse[linewidth=0.02,dimen=outer](4.2,-2.44)(0.2,0.2)
\usefont{T1}{ptm}{m}{n}
\rput(8.5,0.41){Potts partition function}
\usefont{T1}{ptm}{m}{n}
\rput(8.5,1.89){Chromatic polynomial}
\usefont{T1}{ptm}{m}{n}
\rput(9.9,-0.95){Averbouch--Godlin--Makowsky polynomial}
\usefont{T1}{ptm}{m}{n}
\rput(10,-2.37){Tittmann--Averbouch--Makowsky polynomial}
\usefont{T1}{ptm}{m}{n}
\rput(9.615156,-1.37){($\ell=0$ is Dohmen-Ponitz--Tittmann)}
\psline[linewidth=0.02cm](6.46,3.34)(6.46,-3.34)
\psline[linewidth=0.02cm](0.0,2.56)(13.76,2.56)
\usefont{T1}{ptm}{m}{n}
\rput(2.7084374,3.07){Coloured rooted tree}
\usefont{T1}{ptm}{m}{n}
\rput(10.8,3.05){Polynomial sequence upon branching and composition}
\usefont{T1}{ptm}{m}{n}
\rput(3.8,-2.7){$k$}
\usefont{T1}{ptm}{m}{n}
\rput(3.8,-1.2){$1$}
\end{pspicture}
}
\end{center}
\caption{Examples of coloured rooted trees (root black vertex) and the polynomial sequences resulting from them. Colour pairs $(A_s,F_s)$ consist of subgraph encoding $A_s$ and ornament $F_s$. These coloured rooted trees are augmented by branching numbers (in the last pair only, indicated separately at non-root vertices) and then branching performed as described in Definition~\ref{def:augment}, followed by composition of the resulting ornamented graph.   
} \label{fig:ornamented_trees}
\end{figure}

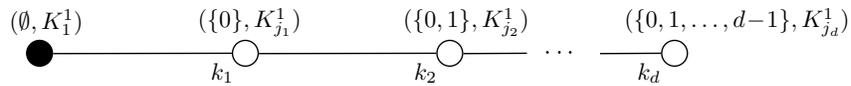
\begin{figure}
\begin{center}
\scalebox{0.9} 
{
\begin{pspicture}(1,-0.5)(11,0.545)
\definecolor{color8693b}{rgb}{0.8,0.8,0.8}
\psdots[dotsize=0.04](8.0,-0.045)
\psdots[dotsize=0.04](8.18,-0.045)
\psdots[dotsize=0.04](8.34,-0.045)
\psellipse[linewidth=0.02,dimen=outer,fillstyle=solid,fillcolor=black](0.5,-0.06)(0.2,0.2)
\psline[linewidth=0.02cm](0.7,-0.065)(3.34,-0.06)
\usefont{T1}{ptm}{m}{n}
\rput(0.6,0.4){$(\emptyset,K_1^1)$}
\usefont{T1}{ptm}{m}{n}
\rput(3.6,0.4){$(\{0\},K_{j_1}^1)$}
\psellipse[linewidth=0.02,dimen=outer](3.54,-0.06)(0.2,0.2)
\usefont{T1}{ptm}{m}{n}
\rput(3.2,-0.35){$k_1$}
\psline[linewidth=0.02cm](3.74,-0.065)(6.36,-0.06)
\usefont{T1}{ptm}{m}{n}
\rput(6.8,0.4){$(\{0,1\},K_{j_2}^1)$}
\psellipse[linewidth=0.02,dimen=outer](6.56,-0.06)(0.2,0.2)
\usefont{T1}{ptm}{m}{n}
\rput(6.2,-0.35){$k_2$}
\psline[linewidth=0.02cm](6.76,-0.065)(7.7,-0.06)
\psline[linewidth=0.02cm](8.78,-0.065)(9.68,-0.06)
\usefont{T1}{ptm}{m}{n}
\rput(10.8,0.4){$(\{0,1,\ldots, d\!-\!1\},K_{j_d}^1)$}
\psellipse[linewidth=0.02,dimen=outer](9.88,-0.06)(0.2,0.2)
\usefont{T1}{ptm}{m}{n}
\rput(9.5,-0.35){$k_d$}
\end{pspicture}
}
\caption{Coloured rooted path of Example~\ref{ex:path}. Colour pairs $(A_s,F_s)$ consist of subgraph encoding $A_s$ and ornament $F_s$ (here a complete graph with a single loop on every vertex). Colours are augmented by branching numbers $k_1,\ldots, k_d$ (root has branching number $1$, not shown). This can be regarded as one way to generalize the Tittmann--Averbouch--Makowsky polynomial (the case $d=1$). }\label{fig:path}
\end{center}
\end{figure}


\begin{thm}\label{thm:poly_branching}
Let $T$ be a rooted tree,  $H$ a spanning subgraph of ${\rm clos}(T)$ and $\mathbf{k}=(k_s:s\in V(T))\in \mathbb{N}^{|V(T)|}$. Further let $\{F_{s;\mathbf{j}_s}:s\in V(T)\}$ be ornaments on the vertices of $T$ such that for each $s\in V(T)$ the sequence $(F_{s;\mathbf{j}_s})$ is strongly polynomial in $\mathbf{j}_s$. 

Then the sequence $(\,T^{\mathbf{k}}(H)[\{F_{s;\mathbf{j}_s}:s\in V(T)\}]\,)$ of branched compositions is strongly polynomial in $(\mathbf{j},\mathbf{k})$, where $\mathbf{j}=(\mathbf{j}_s:s\in V(T))$. 
\end{thm}
\begin{proof}
We first prove the unornamented version of the theorem, i.e., that the sequence $(T^{\mathbf{k}}(H))$ is strongly polynomial
in $\mathbf{k}$, and after this, we indicate how to extend it to the ornamented case.

It will be convenient for the proof to assume that the coloured rooted tree $T$ encoding the graph $H$ (by subset $A_s$ on $s\in V(T)$ encoding adjacencies of $H$ to predecessors of $s$ in $T$) is asymmetric, i.e., there are no two isomorphic subtrees $B(s_1)$ and $B(s_2)$ with roots $s_1\neq s_2$ having $p(s_1)=p(s_2)$.
For if it were the case that for each $s\in V(T)$ the coloured rooted tree $T$ has $l_s$ isomorphic copies of $B(s)$ pendant from $p(s)$, then there is an asymmetric coloured rooted tree $T_0$ which is an induced subgraph of $T$ and such that $T=T_0^{\mathbf{l}}$, where $\mathbf{l}=(l_s:s\in V(T_0))$ has $l_s$ equal to the number of isomorphic copies of $B(s)$ in $T$ pendant from $p(s)$. (The coloured rooted tree $T_0$ here is in fact the {\em branching core} of $T$, to be defined in Section~\ref{sec:bushiness} below.) Once we have the result for asymmetric $T_0$, it then follows that the sequence  $(T^{\mathbf{k}}(H))$ is strongly polynomial in variables $(k_sl_s:s\in V(T))$ since $(T_0^{\mathbf{k}}(H))$ is strongly polynomial in $(k_s:s\in V(T))$. As $\mathbf{l}$ is constant, a polynomial in $(k_sl_s:s\in V(T))$ is a polynomial in $(k_s:s\in V(T))$.

Start then with asymmetric coloured rooted tree $T$ (colours encoding $H$ as a subgraph of ${\rm clos}(T)$) and the fact that
$${\rm hom}(G, T^{\mathbf{k}}(H))=\sum_{\stackrel{S\subseteq T^{\mathbf{k}}(H)}{|V(S)|\leq |V(G)|}}{\rm sur}(G, S),$$
where ${\rm sur}(G, S)$ denotes the number of vertex- and edge-surjective homomorphisms from $G$ to the subgraph $S$ of $T^{\mathbf{k}}(H)$.
It is enough to prove the theorem for connected $G$, so we may assume that the induced subgraph $S$ is connected. 
By assumption $S$ has at most $|V(G)|$ vertices and has maximum degree $\Delta(S)\leq \Delta(G)$, which is independent of $\mathbf{k}$.

For each subtree $B(s)$ rooted at $s$ in the coloured rooted tree $T$, there are $k_s$ isomorphic copies of $B(s)$ in $T^{\mathbf{k}}(H)$ (exactly $k_s$ under the assumption that $T$ is asymmetric); consequently the automorphism group of $T^{\mathbf{k}}(H)$ contains $k_s!$ elements arising from permutations of the $k_s$ copies of $B(s)$ in $T^{\mathbf{k}}$. Let $\Sigma$ be the wreath product of these permutation groups for each $s$; then $\Sigma\leq{\rm Aut}(T^{\mathbf{k}}(H))$ and $|\Sigma|=\mathbf{k}!=\prod_{s\in V(T)}k_s!$.
It now suffices to observe that any given copy of $S$ in $T^{\mathbf{k}}(H)$ as a subgraph can be moved by an automorphism in $\Sigma$ to an isomorphic copy $S'$ in the induced subgraph $T^{\mathbf{m}}(H)$, where $\mathbf{m}=\mathbf{m}(S)=(m_s:s\in V(T))$ 
 depends on $S$ but not on $\mathbf{k}$.  This is because, as a connected graph, $S$ has vertices in at most $\Delta(S)$ copies of the subtree $B(s)$ of $T$ contained in $T^{\mathbf{k}}$, 
so we can take $m_s\leq \Delta(S)$. Hence the number of orbits of a copy of $S$ under $\Sigma$ acting on $T^{\mathbf{k}}(H)$ is independent of $\mathbf{k}$ when $k_s\geq\Delta(S)$ for each $s\in V(T)$

 For $k_s\geq m_s$, the stabilizer of $S'$ under the action of $\Sigma$ contains all permutations of $k_s-m_s$ copies of subtree $B(s)$ not containing any vertices of $S'$, and so has size a multiple of $(\mathbf{k}-\mathbf{m})!=\prod_s(k_s-m_s)!$.
Thus each orbit under $\Sigma$ of a subgraph isomorphic to $S$ has size a divisor of $\mathbf{k}!/(\mathbf{k}-\mathbf{m})!$, which is polynomial in $\mathbf{k}$.
On the other hand, an automorphism in $\Sigma$ that stabilizes $S'$ cannot move one of the $m_s$ branches containing vertices of $S$ to one of the $k_s-m_s$ branches that do not contain a vertex of $S'$, but only these $m_s$ branches among themselves; hence the stabilizer of $S'$ has size a divisor of $(\mathbf{k}-\mathbf{m})!\mathbf{m}!$.

We have just seen that the size of the orbit of $S$ under the action of $\Sigma$ is a multiple of $\frac{\mathbf{k}!}{(\mathbf{k}-\mathbf{m})!\mathbf{m}!}$ and a divisor of $\frac{\mathbf{k}!}{(\mathbf{k}-\mathbf{m})!}$. Let $q(S;\mathbf{k})$ denote the number of isomorphic copies of $S$ occurring as a subgraph of $T^{\mathbf{k}}(H)$. Then, when $k_s\geq m_s$ for each $s\in V(T)$, $q(S;\mathbf{k})$ is a sum of polynomials in $\mathbf{k}$, one for each orbit under $\Sigma$ of subgraphs isomorphic to $S$, and each of which is a multiple of $\frac{\mathbf{k}!}{(\mathbf{k}-\mathbf{m})!\mathbf{m}!}$ and divisor of $\frac{\mathbf{k}!}{(\mathbf{k}-\mathbf{m})!}$.
In particular, $q(S;\mathbf{k})=0$ when $k_s<m_s$ for some $s\in V(T)$. But this means it is also true that $q(S;\mathbf{k})$ is the number of copies of $S$ in $T^{\mathbf{k}}(H)$ when $k_s<m_s$ for some $s\in V(T)$.
Here $m_s$ is the number of copies of subtree $B(s)$ in coloured rooted tree $T^{\mathbf{k}}$ encoding the graph $T^{\mathbf{k}}(H)$ that contain some vertex of $S$. 
If $k_s<m_s$ for some $s\in V(T)$ then $S$ does not occur as a subgraph of $T^{\mathbf{k}}(H)$ (here we use the fact that the coloured rooted tree $T$ is asymmetric). 
Hence, for all $\mathbf{k}\in \mathbb{N}^{V(T)}$, we have
$${\rm hom}(G,T^{\mathbf{k}}(H))=\sum_S{\rm sur}(G,S)q(S;\mathbf{k}),$$
where the sum is over all isomorphism types $S$ of homomorphic images of $G$, ${\rm sur}(G,S)$ is the number of vertex- and edge-surjective homorphisms from $G$ onto $S$ (independent of $\mathbf{k}$), and where $q(G;\mathbf{k})$ is equal to the number of isomorphic copies of $S$ as a subgraph of $T^{\mathbf{k}}(H)$, which we have seen is a polynomial in $\mathbf{k}$ (of degree at most $\Delta(G)$ in each $k_s$, $s\in V(T)$).

This completes the proof that $(T^{\mathbf{k}}(H))$ is a strongly polynomial sequence.

For the ornamented case, to show that ${\rm hom}(G,T^{\mathbf{k}}(H)[\{F_{s;\mathbf{j}_s}:s\in V(T)\}]\,)$ is a polynomial in $(\mathbf{j},\mathbf{k})$ the argument follows the same lines, only in order to count the number of surjective homomorphisms from $G$ to a subgraph $S$ of $T^{\mathbf{k}}(H)[\{F_{s;\mathbf{j}_s}:s\in V(T)\}]$ we first
count the number of subgraphs $S'$ of $T^{\mathbf{k}}(H)$ on vertex set $\{s\in V(T^{\mathbf{k}}):V(S)\cap V(F_{s;\mathbf{k}})\neq\emptyset\}$ (where we recall that ornament $F_{s;\mathbf{j}_s}$ on vertex $s\in V(T)$ is propagated to all copies of $s$ in $T^{\mathbf{k}}(H)$) and second count the number of isomorphic copies of $S\cap F_{s;\mathbf{j}_s}$ in $F_{s;\mathbf{j}_s}$.

 Since the branching operation produces isomorphic copies of the ornamented subtrees $B(s)$ rooted at $s\in V(T)$, the action of $\Sigma$ as an automorphism of $T^{\mathbf{k}}(H)$ also gives an automorphism of $T^{\mathbf{k}}(H)[\{F_{s;\mathbf{j}_s}:s\in V(T)\}]$, and the same counting argument yields the conclusion that there are polynomial in $\mathbf{k}$ of them. By hypothesis, the number of homomorphic copies of $S\cap F_{s;\mathbf{j}_s}$ in $F_{s;\mathbf{j}_s}$ is a polynomial in $\mathbf{j}_s$; by inclusion-exclusion the same is true of the number of isomorphic copies.
Combined together, this yields the desired conclusion that, for given $G$, ${\rm sur}(G,T^{\mathbf{k}}(H)[\{F_{s;\mathbf{j}_s}:s\in V(T)\}]\,)$ is a polynomial in $\mathbf{k}$ and $\mathbf{j}=(\mathbf{j}_s:s\in V(T))$, 
 and hence the same is true of the number of homomorphisms from $G$ to $T^{\mathbf{k}}(H)[\{F_{s;\mathbf{j}_s}:s\in V(T)\}]$.
\end{proof}




\begin{remark} The variables $\mathbf{j_s}$ and $\mathbf{k}$ in Theorem~\ref{thm:poly_branching} need not be independent. Indeed, we may have $\mathbf{j}_s=\mathbf{k}$ for each $s\in V(T)$. In such cases strictly speaking we have a strongly polynomial {\em subsequence} of graphs in the variables $(\mathbf{j},\mathbf{k})$ rather than a strongly polynomial sequence. However, if there is a vector $\mathbf{l}$ of independent variables ranging over $\mathbb{N}^h$ such that each component variable of $(\mathbf{j}_s:s\in V(T))$ and $\mathbf{k}$ can be expressed as a polynomial in the component variables of $\mathbf{l}$, then by taking $\mathbf{l}$ rather than $(\mathbf{j},\mathbf{k})$ in the conclusion of Theorem~\ref{thm:poly_branching} we do have a strongly polynomial sequence in $\mathbf{l}$.
So for example, in the case $\mathbf{j}_s=\mathbf{k}$ for each $s\in V(T)$ we have $\mathbf{l}=\mathbf{k}$, i.e., the sequence $(T^{\mathbf{k}}[\{F_{s;\mathbf{k}}:s\in V(T)\}])$ is strongly polynomial in $\mathbf{k}$ when the same is true of sequences $(F_{s;\mathbf{k}})$ for each $s\in V(T)$.
\end{remark}

The special case $H={\rm clos}(T)$ of Theorem~\ref{thm:poly_branching} 
allows of a different, slightly more direct proof, which has the virtue of also yielding a ``state sum expansion'' for the polynomial ${\rm hom}(G,T^{\mathbf{k}}(H))$ in this case (see Example~\ref{ex:path} that follows): 

\begin{proposition}\label{thm:strong_poly_branching}
Let $T$ be a rooted tree and $\mathbf{k}=(k_s:s\in V(T))\in \mathbb{N}^{|V(T)|}$. Further let $\{F_{s;\mathbf{j}_s}:s\in V(T)\}$ be ornaments on the vertices of $T$ such that for each $s\in V(T)$ the sequence $(F_{s;\mathbf{j}_s})$ is strongly polynomial in $\mathbf{j}_s$.

Then the sequence $(\,{\rm clos}(T^{\mathbf{k}})[\{F_{s;\mathbf{j}_s}:s\in V(T)\}]\,)$ of branched compositions  is strongly polynomial in $(\mathbf{j},\mathbf{k})$ where $\mathbf{j}=(\mathbf{j}_s:s\in V(T))$.
\end{proposition}
\begin{proof}
Although this result is a corollary of  Theorem~\ref{thm:poly_branching}, we give a different proof here by exploiting the recursive structure of closures of rooted trees.

We proceed by induction on the height of $T$. For height $0$ ($T$ just a root vertex $r$, ${\rm clos}(T^{k_r})$ consisting of $k_r$ isolated vertices) we have
${\rm hom}(G,{\rm clos}(T^{k_r})[\{F_r\}])=k_r^{|V(G)|}{\rm hom}(G,F_{r;\mathbf{j}_r})$,
which is strongly polynomial in $(\mathbf{j}_r,k_r)$ by hypothesis on $(F_{r;\mathbf{j}_s})$.

Now take $T$ to be of height $d>0$ and assume the truth of the theorem for trees of height at most $d-1$. Suppose that $T$ has $b$ subtrees $T_1, T_2,\ldots, T_b$ pendant from its root. These subtrees are rooted trees of height at most $d-1$. 

Let ${\rm clos}(T^{\mathbf{k}})\left[\,\{F_{r;\mathbf{j}_r}=K_0\}\cup\{F_{s;\mathbf{j}_s}:s\in V(T)\!-\!r\}\,\right]$ denote the subgraph of\newline ${\rm clos}(T^{\mathbf{k}})\left[\,\{F_{s;\mathbf{j}_s}:s\in V(T)\}\,\right]$ obtained by deleting the root of $T^{\mathbf{k}}$ and its ornament (setting $F_{r;\mathbf{j}_r}=K_0$, the empty graph) before making the composition. This subgraph is the disjoint union of graphs ${\rm clos}(T_i^{\mathbf{k}_i})[\{F_{s;\mathbf{j}_s}:s\in V(T_i)\}]$, $1\leq i\leq b$, where $\mathbf{k}_i=(k_s:s\in V(T_i))$.

For a graph $G=(V,E)$, 
since in ${\rm clos}(T^{\mathbf{k}})$ the root $r$ is connected to every other vertex we then have the decomposition\medskip

{\small{$ {\rm hom}(G,{\rm clos}(T^{\mathbf{k}})\left[\,\{F_{s;\mathbf{j}_s}:s\in V(T)\}\,\right]=$\\
$$\sum_{U\subseteq V}\:k_r^{|V\setminus U|}{\rm hom}(G[V\setminus U],F_{r;\mathbf{j}_r}){\rm hom}\!\left(G[U],{\rm clos}(T^{\mathbf{k}})[\{F_{r;\mathbf{j}_r}\!=\!K_0\}\cup\{F_{s;\mathbf{j}_s}:s\in V(T)\!-\!r\}]\right).$$}}

 By inductive hypothesis the sequence ${\rm clos}(T_i^{\mathbf{k}_i})[\{F_{s;\mathbf{j}_s}:s\in V(T_i)\}])$ is strongly polynomial in $(\mathbf{j}_s:s\in V(T_i))$ and $\mathbf{k}_i$, for each $1\leq i\leq b$. Moreover, ${\rm hom}(G[U],\,\cdot\,)$ is multiplicative over the connected components of $G[U]$, and a connected graph when sent by homomorphism to the disjoint union of  ${\rm clos}(T_i^{\mathbf{k}_i})[\{F_{s;\mathbf{j}_s}:s\in V(T_i)\}]$, $1\leq i\leq b$, must be sent to just one of the graphs ${\rm clos}(T_i^{\mathbf{k}_i})[\{F_{s;\mathbf{j}_s}:s\in V(T_i)\}]$. Thus it follows that
${\rm hom}(\,G[U],{\rm clos}(T^{\mathbf{k}})[\{F_{r;\mathbf{j}_r}\!=\!K_0\}\cup\{F_{s;\mathbf{j}_s}:s\in V(T)\!-\!r\}]\,)$ is polynomial in $\mathbf{k}_1,\ldots, \mathbf{k}_b$. The inductive step now goes through.
\end{proof}




\begin{example}\label{ex:path} Consider $T=P_{d+1}$, the path on $d+1$ vertices rooted at one of its endpoints $r$, for which ${\rm clos}(T)=K_{d+1}$ (in the rooted coloured tree $T$ encoded by subset $\{0,1,\ldots, \ell-1\}$ at vertex $\ell$ away from the root). Let each non-root vertex $s$ be ornamented by $F_s=K_{j_s}^1$, and the root $r$ ornamented by $F_r=K_1^1$. Finally, let each non-root vertex $s\in V(T)$ have branching number $k_s$. (See Figure~\ref{fig:path}.) Here 
$${\rm hom}(G,{\rm clos}(T^{\mathbf{k}})[\{F_s\}]\,)=\sum_{V_0\sqcup V_1\sqcup\cdots\sqcup V_d=V}\;\prod_{1\le\ell\le d}j_\ell^{|V_\ell|}\prod_{1\le\ell\le d}k_{\ell}^{c(G[V-\cup_{0\le i<\ell}V_i])}.$$
The special case when $j_{\ell}=j$ and $k_{\ell}=k$ for $1\le \ell\le d$ gives
\begin{align*}
{\rm hom}(G,{\rm clos}(T^{\mathbf{k}})[K_j^1;K_k^1,\ldots, K_k^1])  & =\sum_{V_0\sqcup V_1\sqcup\cdots\sqcup V_d=V}j^{|V|-|V_0|}k^{\sum_{1\le \ell\le d} c(G[V-\cup_{0\le i< \ell}V_i])}\\ \noalign{\bigskip}
& = \sum_{\emptyset\subseteq W_1\subseteq W_2\subseteq\cdots\subseteq W_d\subseteq V}j^{|W_d|}k^{\sum_{1\le\ell\le d}c(G[W_\ell])}\end{align*}
(In the second summation, $W_{\ell}=V\setminus (V_0\cup  V_1\cup\cdots\cup V_{d-\ell})$.)


\end{example}

\subsection{Branching cores}\label{sec:bushiness}

We have seen in Theorem~\ref{thm:poly_branching} 
 how branching, in conjunction with composition, can  generate a strongly polynomial sequence of graphs from a fixed ornamented graph.  We now turn to the question of when it is possible to index a given countable family of graphs by tuples $\mathbf{k}\in\mathbb{N}^h$ so that the resulting sequence is the union of a finite number of strongly polynomial subsequences in $\mathbf{k}$. In Theorem~\ref{thm:bounded_bushiness} we provide a sufficient condition for this to be true for a family of simple graphs, and a similar condition results for a family of ornamented graphs (Corollary~\ref{cor:ornamented}).

First we need to consider the inverse operation to branching in order to define a new graph parameter that measures how symmetric a simple graph $H$ is with respect to involutive automorphisms with a fixed point. (These automorphisms are those that arise from exchanging isomorphic branches in a rooted tree $T$ that contains $H$ as a subgraph of its closure.)  

Consider then a simple graph $H$ together with a coloured rooted tree $T$ of whose closure $H$ is a spanning subgraph.
Proceeding by the method described in Section 4.2, we construct the coloured rooted tree $T_0=T^{(\mathbf{k})}$ with $\mathbf{k}$ set equal to the all-one tuple. Thus, every vertex $s\in V(T_0)=V(T)$ has assigned a colour with three components: the set $A_s$ (encoding $H$), the ornament $F_s$, and the branching multiplicity $k_s=1$.

Suppose now that there exist vertices $s, s_1, \ldots ,s_{\ell} \in V(T_0)$ with the same predecessor (i.e., $p(s)=p(s_1)=\cdots =p(s_{\ell})$) and such that the subtrees $B(s), B(s_1), \ldots, B(s_{\ell})$ of $T_0$ are isomorphic (including colours -- at this initial stage these colours all have third component corresponding to branching multiplicity equal to 1). Let $T_1$ be the coloured rooted tree obtained from $T_0$ by deleting the subtrees $B(s_1), \ldots, B(s_{\ell})$ and by recolouring vertex $s$ from $(A_s, F_s,1)$ to $(A_s, F_s, \ell+1)$.

By repeating this process of eliminating isomorphic subtrees whose roots have the same predecessor, at an arbitrary stage, we get a coloured rooted tree $T_m$ in which we have to eliminate the isomorphic subtrees with roots, say $t, t_1,\ldots, t_{p}$, having the same predecessor. These subtrees are colour isomorphic, except possibly differing in the branching multiplicities $k_t, k_{t_1}, \ldots, k_{t_{p}}$ of the roots. Then, the next coloured rooted tree $T_{m+1}$ is obtained by deleting the subtrees $B(t_1), \ldots, B(t_{p})$ and by recolouring vertex $t$ from $(A_t, F_t,1)$ to $(A_t, F_t, k)$ where $k=k_t+k_{t_1}+ \ldots +k_{t_{p}}$.

The process concludes
with a coloured rooted tree, say $T'$, in which there are no isomorphic subtrees remaining whose roots share the same predecessor. Moreover,
by similar reasoning to the proof of Lemma~\ref{lem:branching_order}, one can prove that $T'$ is uniquely determined by $T$. Let us call this tree $T'$ the {\em branching core of $T$}.
Note that, reversing the above reduction, by branching $T'$ we obtain $T$ again. Figure 7 illustrates some examples.



\begin{definition}
The {\em branching core size}, ${\rm bc}(T)$, of a coloured rooted tree $T$ is defined by ${\rm bc}(T)=|V(T')|$, where $T'$ is the branching core of $T$, as defined above.
\end{definition}
\begin{remark}
The vertices of the branching core $T'$ of $T$ are in one-to-one correspondence with the orbits of $V(T)$ under the action of  of ${\rm Aut(T)}$: a vertex $s\in V(T')$ corresponds to an orbit of $\prod_{t\in P(s)}k_t$ vertices on the same level in $T$ as that of $s$ in $T'$. 
 Thus, the branching core size of a coloured rooted tree $T$ may alternatively be defined as the number of orbits under the action of ${\rm Aut(T)}$ on $V(T)$.
\end{remark}
Clearly, for a coloured rooted tree $T$ we have ${\rm height}(T)\leq{\rm bc}(T)\leq |V(T)|$, with ${\rm height}(T)={\rm bc}(T)$ when the branching core of $T$ is a path rooted at an endpoint. We have ${\rm bc}(T)=|V(T)|$ when the branching core of $T$ is $T$ itself, which happens when $T$ has no isomorphic subtrees whose roots share the same predecessor ($T$ has no non-trivial automorphisms). 

For a given simple graph $H$, the branching core size of a coloured rooted tree $T$ encoding $H$ as a subgraph of its closure can vary considerably. 
For example, consider the star $K_{1,k}$. On the one hand, as a subgraph of $T={\rm clos}(T)=K_{1,k}$ rooted at its central vertex (subgraph colour component $A_s=\{0\}$ on leaf $s$), the branching core of this coloured rooted tree is $K_2$ (there is branching multiplicity $k$ on the non-root vertex), so that ${\rm bc}(T)=2$. On the other hand, the star $K_{1,k}$ is also a subgraph of the closure of $P_{k+1}$ rooted at an endpoint; in other words, $K_{1,k}$ can be encoded by coloured tree $T=P_{k+1}$ rooted at an endpoint, each non-root vertex $s$ having subgraph colour component $A_s=\{0\}$; here the branching core of $T$ is $T$ itself, so that ${\rm bc}(T)=k+1$.

It will be convenient to denote the graph $H$ ornamented by $F_s$ for each $s\in V(H)$ by $H\langle\{F_s:s\in V(H)\}\rangle$, or more briefly by $H\langle\{F_s\}\rangle$. After composition this becomes the graph $H[\{F_s\}]$.

\begin{definition}
The {\em minimum branching core size}, ${\rm \ul{bc}}(H\langle\{F_s\}\rangle)$ of an ornamented graph $H\langle\{F_s\}\rangle$ is the minimum value of ${\rm bc}(T)$ over all coloured rooted trees $T$ on vertex set $V(H)$, in which vertex $s\in V(T)$ has colour $(A_s,F_s)$, where $A_s=\{|P(s)|\!-\!1\!-\!d(s,t):t\in P(s), st\in E(H)\}$.

The minimum branching core size ${\rm \ul{bc}}(H)$ of a simple graph $H$ is defined as ${\rm \ul{bc}}(H\langle\{K_1\}\rangle)$, where $H$ is ornamented with the same graph $K_1$ on each of its vertices.
\end{definition}

\begin{figure}
\begin{center}
\scalebox{0.75} 
{
\begin{pspicture}(0.8,-6.65625)(17.201876,6.65625)
\psline[linewidth=0.02cm](7.02,5.1521873)(7.4,3.8721876)
\psline[linewidth=0.02cm](14.44,0.7321875)(14.44,-0.5278125)
\psline[linewidth=0.02cm](14.44,-0.6078125)(14.46,-1.8078125)
\psdots[dotsize=0.2](14.46,-1.7678125)
\psdots[dotsize=0.2,fillstyle=solid,dotstyle=o](14.44,-0.5078125)
\psdots[dotsize=0.2,fillstyle=solid,dotstyle=o](14.42,0.8121875)
\usefont{T1}{ptm}{m}{n}
\rput(14.611406,0.5621875){$\mathbf{2}$}
\usefont{T1}{ptm}{m}{n}
\rput(14.591406,-0.8178125){$\mathbf{2}$}
\psline[linewidth=0.02cm](8.16,5.1321874)(7.48,3.8721876)
\psline[linewidth=0.02cm](7.42,3.7921875)(6.68,2.6121874)
\psline[linewidth=0.02cm](5.78,3.9321876)(6.68,2.5921874)
\psline[linewidth=0.02cm](6.24,5.1521873)(5.82,3.8921876)
\psline[linewidth=0.02cm](5.0,5.1721873)(5.74,3.9721875)
\psdots[dotsize=0.2](6.68,2.6121874)
\psdots[dotsize=0.2,fillstyle=solid,dotstyle=o](4.98,5.1921873)
\psdots[dotsize=0.2,fillstyle=solid,dotstyle=o](6.26,5.1921873)
\psdots[dotsize=0.2,fillstyle=solid,dotstyle=o](5.8,3.8921876)
\psdots[dotsize=0.2,fillstyle=solid,dotstyle=o](7.44,3.8721876)
\psdots[dotsize=0.2,fillstyle=solid,dotstyle=o](8.18,5.2121873)
\usefont{T1}{ptm}{m}{n}
\rput(13.131406,4.9421873){$\mathbf{2}$}
\psdots[dotsize=0.2,fillstyle=solid,dotstyle=o](7.0,5.2121873)
\psline[linewidth=0.02cm](15.94,5.1121874)(15.26,3.8521874)
\psline[linewidth=0.02cm](15.22,3.7721875)(14.46,2.5921874)
\psline[linewidth=0.02cm](13.56,3.9121876)(14.46,2.5721874)
\psline[linewidth=0.02cm](12.78,5.1521873)(13.52,3.9521875)
\psdots[dotsize=0.2](14.46,2.5921874)
\psdots[dotsize=0.2,fillstyle=solid,dotstyle=o](12.76,5.1721873)
\psdots[dotsize=0.2,fillstyle=solid,dotstyle=o](13.58,3.8721876)
\psdots[dotsize=0.2,fillstyle=solid,dotstyle=o](15.26,3.8721876)
\psdots[dotsize=0.2,fillstyle=solid,dotstyle=o](15.96,5.1921873)
\usefont{T1}{ptm}{m}{n}
\rput(16.111406,4.9421873){$\mathbf{2}$}
\usefont{T1}{ptm}{m}{n}
\rput(10.011406,0.5621875){$\mathbf{2}$}
\psline[linewidth=0.02cm](11.82,0.8521875)(11.38,-0.4278125)
\psline[linewidth=0.02cm](11.34,-0.5678125)(10.88,-1.7678125)
\psline[linewidth=0.02cm](10.28,-0.5678125)(10.88,-1.7878125)
\psline[linewidth=0.02cm](9.68,0.8121875)(10.24,-0.4878125)
\psdots[dotsize=0.2](10.88,-1.7678125)
\psdots[dotsize=0.2,fillstyle=solid,dotstyle=o](9.64,0.8721875)
\psdots[dotsize=0.2,fillstyle=solid,dotstyle=o](10.24,-0.4678125)
\psdots[dotsize=0.2,fillstyle=solid,dotstyle=o](11.36,-0.4878125)
\psdots[dotsize=0.2,fillstyle=solid,dotstyle=o](11.84,0.8721875)
\usefont{T1}{ptm}{m}{n}
\rput(11.991406,0.5621875){$\mathbf{2}$}
\psline[linewidth=0.02cm](9.04,-3.2078125)(7.62,-4.5278125)
\psline[linewidth=0.02cm](7.18,-3.1878126)(7.56,-4.4678125)
\psline[linewidth=0.02cm](8.06,-3.2478125)(7.64,-4.4678125)
\psline[linewidth=0.02cm](7.6,-4.5678124)(6.84,-5.7278123)
\psline[linewidth=0.02cm](5.94,-4.4078126)(6.84,-5.7478123)
\psline[linewidth=0.02cm](6.4,-3.1878126)(5.98,-4.4478126)
\psline[linewidth=0.02cm](5.16,-3.1678126)(5.9,-4.3678126)
\psdots[dotsize=0.2](6.84,-5.7278123)
\psdots[dotsize=0.2,fillstyle=solid,dotstyle=o](5.14,-3.1478126)
\psdots[dotsize=0.2,fillstyle=solid,dotstyle=o](6.42,-3.1478126)
\psdots[dotsize=0.2,fillstyle=solid,dotstyle=o](5.96,-4.4478126)
\psdots[dotsize=0.2,fillstyle=solid,dotstyle=o](7.6,-4.5278125)
\psdots[dotsize=0.2,fillstyle=solid,dotstyle=o](8.08,-3.1678126)
\psdots[dotsize=0.2,fillstyle=solid,dotstyle=o](7.16,-3.1278124)
\psdots[dotsize=0.2,fillstyle=solid,dotstyle=o](9.1,-3.1678126)
\usefont{T1}{ptm}{m}{n}
\rput(13.271406,-3.5178125){$\mathbf{2}$}
\psline[linewidth=0.02cm](15.96,-3.2678125)(15.28,-4.5278125)
\psline[linewidth=0.02cm](15.26,-4.5878124)(14.48,-5.7878127)
\psline[linewidth=0.02cm](13.58,-4.4678125)(14.48,-5.8078127)
\psline[linewidth=0.02cm](12.8,-3.2278125)(13.54,-4.4278126)
\psdots[dotsize=0.2](14.48,-5.7878127)
\psdots[dotsize=0.2,fillstyle=solid,dotstyle=o](12.78,-3.1878126)
\psdots[dotsize=0.2,fillstyle=solid,dotstyle=o](13.6,-4.5078125)
\psdots[dotsize=0.2,fillstyle=solid,dotstyle=o](15.3,-4.5078125)
\psdots[dotsize=0.2,fillstyle=solid,dotstyle=o](15.98,-3.1878126)
\usefont{T1}{ptm}{m}{n}
\rput(16.151405,-3.4978125){$\mathbf{3}$}
\psline[linewidth=0.02cm](7.06,0.8121875)(7.44,-0.4678125)
\psline[linewidth=0.02cm](8.2,0.7921875)(7.52,-0.4678125)
\psline[linewidth=0.02cm](7.46,-0.5878125)(6.72,-1.7278125)
\psline[linewidth=0.02cm](5.82,-0.4078125)(6.72,-1.7478125)
\psline[linewidth=0.02cm](6.28,0.8121875)(5.86,-0.4478125)
\psline[linewidth=0.02cm](5.04,0.8321875)(5.78,-0.3678125)
\psdots[dotsize=0.2](6.72,-1.7278125)
\psdots[dotsize=0.2,fillstyle=solid,dotstyle=o](5.02,0.9321875)
\psdots[dotsize=0.2,fillstyle=solid,dotstyle=o](6.3,0.8521875)
\psdots[dotsize=0.2,fillstyle=solid,dotstyle=o](5.84,-0.4478125)
\psdots[dotsize=0.2,fillstyle=solid,dotstyle=o](7.48,-0.4878125)
\psdots[dotsize=0.2,fillstyle=solid,dotstyle=o](8.22,0.8721875)
\psdots[dotsize=0.2,fillstyle=solid,dotstyle=o](7.04,0.8721875)
\usefont{T1}{ptm}{m}{n}
\rput(1.808125,6.4521875){\small simple graph $H$}
\usefont{T1}{ptm}{m}{n}
\rput(14.4025,2.1071875){\footnotesize bc$(T)=5$}
\usefont{T1}{ptm}{m}{n}
\rput(14.376875,6.4721875){\small branching core of $T$}
\usefont{T1}{ptm}{m}{n}
\rput(6.540625,6.4721875){\small coloured rooted tree $T$}
\usefont{T1}{ptm}{m}{n}
\rput(10.860156,-2.2128124){\footnotesize (intermediate coloured tree)}
\psbezier[linewidth=0.02](0.28,5.1121874)(0.04,4.6721873)(0.45355412,3.922176)(0.76,3.4721875)(1.0664458,3.0221992)(1.44,2.6921875)(2.04,2.6321876)
\psline[linewidth=0.02cm](2.38,5.1121874)(2.76,3.8321874)
\psline[linewidth=0.02cm](3.52,5.0921874)(2.84,3.8321874)
\psline[linewidth=0.02cm](2.78,3.7521875)(2.04,2.5721874)
\psline[linewidth=0.02cm](1.14,3.8921876)(2.04,2.5521874)
\psline[linewidth=0.02cm](1.6,5.1121874)(1.18,3.8521874)
\psline[linewidth=0.02cm](0.36,5.1321874)(1.1,3.9321876)
\psdots[dotsize=0.2](2.04,2.5721874)
\psdots[dotsize=0.2,fillstyle=solid,dotstyle=o](0.34,5.1521873)
\psdots[dotsize=0.2,fillstyle=solid,dotstyle=o](1.16,3.8521874)
\psdots[dotsize=0.2,fillstyle=solid,dotstyle=o](2.8,3.8321874)
\psdots[dotsize=0.2,fillstyle=solid,dotstyle=o](3.54,5.1721873)
\psdots[dotsize=0.2,fillstyle=solid,dotstyle=o](2.36,5.1721873)
\psbezier[linewidth=0.02](1.68,5.1121874)(1.8231908,4.7796674)(2.0,4.4921875)(2.04,3.8721876)(2.08,3.2521875)(2.14,2.9121876)(2.046316,2.6321876)
\psdots[dotsize=0.2,fillstyle=solid,dotstyle=o](1.62,5.1521873)
\usefont{T1}{ptm}{m}{n}
\rput(4.811406,5.5021877){$\{0,1\}$}
\usefont{T1}{ptm}{m}{n}
\rput(6.1714063,5.5021877){$\{0,1\}$}
\usefont{T1}{ptm}{m}{n}
\rput(7.1714063,5.5021877){$\{1\}$}
\usefont{T1}{ptm}{m}{n}
\rput(8.431406,5.5021877){$\{1\}$}
\usefont{T1}{ptm}{m}{n}
\rput(7.871406,3.8421874){$\{0\}$}
\usefont{T1}{ptm}{m}{n}
\rput(5.391406,3.8221874){$\{0\}$}
\usefont{T1}{ptm}{m}{n}
\rput(14.3825,-6.4528127){\footnotesize bc$(T)=5$}
\usefont{T1}{ptm}{m}{n}
\rput(14.4225,-2.3728125){\footnotesize bc$(T)=3$}
\usefont{T1}{ptm}{m}{n}
\rput(12.491406,5.4821873){$\{0,1\}$}
\usefont{T1}{ptm}{m}{n}
\rput(16.211407,5.5021877){$\{1\}$}
\usefont{T1}{ptm}{m}{n}
\rput(13.111406,3.7621875){$\{0\}$}
\usefont{T1}{ptm}{m}{n}
\rput(15.691406,3.8221874){$\{0\}$}
\psline[linewidth=0.02cm](2.36,0.7121875)(2.74,-0.5678125)
\psline[linewidth=0.02cm](3.5,0.6921875)(2.82,-0.5678125)
\psline[linewidth=0.02cm](2.76,-0.6878125)(2.02,-1.8278126)
\psline[linewidth=0.02cm](1.12,-0.5078125)(2.02,-1.8478125)
\psline[linewidth=0.02cm](1.58,0.7121875)(1.16,-0.5478125)
\psline[linewidth=0.02cm](0.34,0.7321875)(1.08,-0.4678125)
\psdots[dotsize=0.2](2.02,-1.8278126)
\psdots[dotsize=0.2,fillstyle=solid,dotstyle=o](0.32,0.7521875)
\psdots[dotsize=0.2,fillstyle=solid,dotstyle=o](1.6,0.7521875)
\psdots[dotsize=0.2,fillstyle=solid,dotstyle=o](1.14,-0.5478125)
\psdots[dotsize=0.2,fillstyle=solid,dotstyle=o](2.78,-0.5878125)
\psdots[dotsize=0.2,fillstyle=solid,dotstyle=o](3.52,0.7721875)
\psdots[dotsize=0.2,fillstyle=solid,dotstyle=o](2.34,0.7721875)
\usefont{T1}{ptm}{m}{n}
\rput(5.411406,-0.5578125){$\{0\}$}
\usefont{T1}{ptm}{m}{n}
\rput(7.931406,-0.5378125){$\{0\}$}
\usefont{T1}{ptm}{m}{n}
\rput(4.6914062,1.1421875){$\{1\}$}
\usefont{T1}{ptm}{m}{n}
\rput(6.0714064,1.1621875){$\{1\}$}
\usefont{T1}{ptm}{m}{n}
\rput(7.1514063,1.1621875){$\{1\}$}
\usefont{T1}{ptm}{m}{n}
\rput(8.231406,1.1621875){$\{1\}$}
\usefont{T1}{ptm}{m}{n}
\rput(9.431406,1.1821876){$\{1\}$}
\usefont{T1}{ptm}{m}{n}
\rput(11.991406,1.2021875){$\{1\}$}
\usefont{T1}{ptm}{m}{n}
\rput(9.751407,-0.4578125){$\{0\}$}
\usefont{T1}{ptm}{m}{n}
\rput(11.731406,-0.4578125){$\{0\}$}
\usefont{T1}{ptm}{m}{n}
\rput(14.051406,-0.2778125){$\{0\}$}
\usefont{T1}{ptm}{m}{n}
\rput(14.171406,1.1221875){$\{1\}$}
\psline[linewidth=0.02cm](4.26,-3.2278125)(2.84,-4.5478125)
\psline[linewidth=0.02cm](2.4,-3.2078125)(2.78,-4.4878125)
\psline[linewidth=0.02cm](3.28,-3.2678125)(2.86,-4.4878125)
\psline[linewidth=0.02cm](2.82,-4.5878124)(2.06,-5.7478123)
\psline[linewidth=0.02cm](1.16,-4.4278126)(2.06,-5.7678127)
\psline[linewidth=0.02cm](1.62,-3.2078125)(1.2,-4.4678125)
\psline[linewidth=0.02cm](0.38,-3.1878126)(1.12,-4.3878126)
\psdots[dotsize=0.2](2.06,-5.7478123)
\psdots[dotsize=0.2,fillstyle=solid,dotstyle=o](0.36,-3.1678126)
\psdots[dotsize=0.2,fillstyle=solid,dotstyle=o](1.64,-3.1678126)
\psdots[dotsize=0.2,fillstyle=solid,dotstyle=o](1.18,-4.4678125)
\psdots[dotsize=0.2,fillstyle=solid,dotstyle=o](2.82,-4.5478125)
\psdots[dotsize=0.2,fillstyle=solid,dotstyle=o](3.3,-3.1878126)
\psdots[dotsize=0.2,fillstyle=solid,dotstyle=o](2.38,-3.1478126)
\psdots[dotsize=0.2,fillstyle=solid,dotstyle=o](4.26,-3.1878126)
\usefont{T1}{ptm}{m}{n}
\rput(5.0514064,-2.8578124){$\{1\}$}
\usefont{T1}{ptm}{m}{n}
\rput(6.291406,-2.8578124){$\{1\}$}
\usefont{T1}{ptm}{m}{n}
\rput(7.0714064,-2.8378124){$\{1\}$}
\usefont{T1}{ptm}{m}{n}
\rput(8.011406,-2.8578124){$\{1\}$}
\usefont{T1}{ptm}{m}{n}
\rput(9.111406,-2.8578124){$\{1\}$}
\usefont{T1}{ptm}{m}{n}
\rput(5.5514064,-4.5778127){$\{0\}$}
\usefont{T1}{ptm}{m}{n}
\rput(8.031406,-4.5978127){$\{0\}$}
\usefont{T1}{ptm}{m}{n}
\rput(13.251407,-4.7378125){$\{0\}$}
\usefont{T1}{ptm}{m}{n}
\rput(15.571406,-4.7578125){$\{0\}$}
\usefont{T1}{ptm}{m}{n}
\rput(12.431406,-2.9778125){$\{1\}$}
\usefont{T1}{ptm}{m}{n}
\rput(16.231407,-2.9178126){$\{1\}$}
\psline[linewidth=0.02cm](0.0,6.1521873)(17.16,6.1721873)
\end{pspicture}
}

\caption{On the left, a simple graph $H$ given as a subgraph of the closure of coloured rooted tree $T$. 
On the right, reduction of $T$ to its branching core $T'$. Branching multiplicities are only indicated when they exceed $1$ and are shown in boldface.
 }\label{fig:branching_core}
\end{center}
\end{figure}
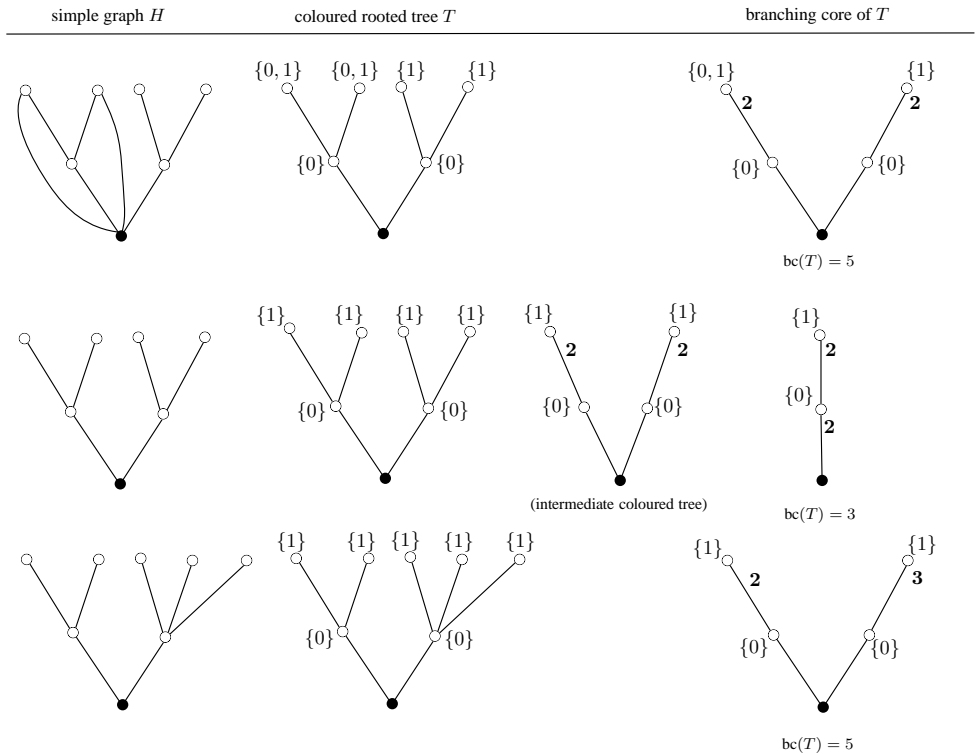



\begin{example}
The graph $K_1$ with ornament $K_k$ on its single vertex has minimum branching core size $1$, i.e., ${\rm \ul{bc}}(K_1\langle K_k\rangle)=1$. However, after composition, $K_1\langle K_k\rangle$ becomes the graph $K_1[K_k]\cong K_k$ and as a simple graph $K_k$ has minimum branching core size $k$,  which is to say ${\rm \ul{bc}}(K_k)=k$.
\end{example}

Encoding a simple graph $H$ as a coloured rooted tree $T$ with vertex $s\in V(T)$ having colour $A_s=\{|P(s)|\!-\!1\!-\!d(s,t):t\in P(s), st\in E(H)\}$, we have ${\rm td}(H)\leq {\rm \ul{bc}}(H)\leq |V(H)|$. There is equality ${\rm td}(H)={\rm \ul{bc}}(H)$ when $H$ can be encoded as a coloured rooted tree $T$ whose branching core is a path rooted at an endpoint. For example, ${\rm td}(K_k)={\rm \ul{bc}}(K_k)=k$. There is equality ${\rm \ul{bc}}(H)=|V(H)|$ when $H$ has no automorphisms arising from exchanging isomorphic branches in a coloured tree representation of $H$ (equivalently, $H$ has no involutive automorphisms with a fixed point). For example, ${\rm \ul{bc}}(P_{2\ell})=2\ell$, whereas ${\rm \ul{bc}}(P_{2\ell-1})=\ell$.
Another example: the star $K_{1,k}$ is a subgraph of the closure of the tree $T=K_{1,k}$ rooted at its central vertex and having all its leaves at height $1$. This gives a coloured rooted tree of branching core size $2$, which clearly is mimimum, so ${\rm \ul{bc}}(K_{1,k})=2$. On the other hand, by embedding $K_{1,k}$ as a subgraph of the closure of $P_{k+1}$ rooted at an endpoint we have a coloured rooted tree representation of $K_{1,k}$ with branching core size $k+1$.

Generally, given $H$, there are many ways to embed it as a spanning subgraph of the closure of a rooted tree: choose a vertex of $H$ to be the root, and then a spanning tree of $H$. The example of $\{K_{1,k}:k\in\mathbb{N}\}$ shows that a family of graphs can have bounded minimum branching core size but have an encoding by coloured rooted trees that have unbounded branching core size.

\begin{remark}
(i) Any finite tree $T$ (not rooted or coloured) has either a vertex or an edge which is fixed by all automorphisms, according as it is central or bicentral. In the latter case there are no involutive automorphisms with a fixed point. In the former case, taking the centre of $T$ as root, all automorphisms of the tree are also automorphisms of the rooted tree. Thus, any given tree $T$ always has a coloured rooted version of itself whose automorphism group is exactly the set of automorphisms of $T$ as a simple unrooted graph that have a fixed point.

(ii) For a simple graph $H$ that is not a tree there is not necessarily a coloured rooted tree $T$ encoding $H$ as a subgraph of ${\rm clos}(T)$ with the property that {\em all} involutive automorphisms of $H$ with a fixed point can be realized as an automorphism of $T$. It is not even true that for each involutive automorphism of $H$ with a fixed point there will be {\em some} coloured rooted tree $T$ for which this automorphism of $H$ corresponds to an automorphism of $T$.
For example, the triangle $C_3$ has three involutive automorphisms having a fixed point, but a coloured rooted tree $T$ encoding $C_3$ as a subgraph of ${\rm clos}(T)$ has no automorphisms (since each vertex of $C_3$ must appear at a different level of $T$). 
%

So while the minimum branching core size depends on the size of the subgroup of automorphisms of $H$ that are involutive with fixed points, it may not reflect its true size due to the loss of symmetry in encoding it by a coloured rooted tree.
\end{remark}

We say that a family of simple graphs $\mathcal{H}$ has bounded minimum branching core size if $\{{\rm \ul{bc}}(H):H\in\mathcal{H}\}$ is bounded. 
The family $\mathcal{H}$ cannot have bounded minimum branching core size if is has unbounded tree-depth, because maximal subtrees rooted at different levels cannot be isomorphic.
So if $\mathcal{H}$ has bounded minimum branching core size then there are encodings of $H\in\mathcal{H}$ by coloured rooted trees $T$ of bounded height and such that the branching core of $T$ has a bounded number of non-isomorphic subtrees at each level.

We first state the main result of this section for simple graphs, and extend it to ornamented graphs in Corollary~\ref{cor:ornamented}.

\begin{thm}\label{thm:bounded_bushiness}
Let $\mathcal{H}$ be a family of simple graphs of bounded minimum branching core size. 

Then $\mathcal{H}$ can be partitioned into a finite number of strongly polynomial branching subsequences.
More specifically, there is a partition of $\mathcal{H}$ into subsets $\mathcal{H}_\ell$, $1\leq \ell\leq m$, and indexing $\mathcal{H}_\ell=(H_{\mathbf{k}_\ell}:\mathbf{k}_\ell\in I_\ell)$ for some $I_\ell\subseteq\mathbb{N}^{h_\ell}$ ($h_\ell\geq 1$) so that, for each $\ell$,
\begin{itemize}
\item[(i)]  there is a rooted tree $T_\ell$ with $|V(T_\ell)|=h_\ell$ and fixed graph $G_\ell\subseteq{\rm clos}(T_\ell)$ such that for every $\mathbf{k}_\ell\in I_\ell$ we have $H_{\mathbf{k}_\ell}=T_\ell^{\mathbf{k}_\ell}(G_\ell)$, and
\item[(ii)] $(H_{\mathbf{k}_\ell})$ is a strongly polynomial subsequence indexed by ${\mathbf{k}_\ell}\in I_\ell\subseteq\mathbb{N}^{h_\ell}$. 
\end{itemize}
\end{thm}

\begin{proof} It suffices to prove (i), since (ii) then follows by Theorem~\ref{thm:poly_branching} (the sequence ($T_\ell^{\mathbf{k}_\ell}(G_\ell)$) indexed by ${\mathbf{k}_\ell}\in I_\ell\subseteq\mathbb{N}^{h_\ell}$ is a subsequence of the strongly polynomial sequence ($T_\ell^{\mathbf{k}}(G_\ell):\mathbf{k}\in \mathbb{N}^{h_\ell}$).

By the assumption that $\mathcal{H}$ is of bounded minimum branching core size, there is a bound $B$ such that for each $H\in\mathcal{H}$ we can encode $H$ by an asymmetric coloured rooted tree $T$ with $|V(T)|={\rm \ul{bc}}(H)\leq B$, where vertex $s\in V(T)$ is coloured by a set $A_s\subseteq\{0,1,\ldots, {\rm height}(T)\!-\!1\}$ encoding adjacencies of a subgraph $G$ of ${\rm clos}(T)$, such that $H\cong T^{\mathbf{k}}(G)$ for some $\mathbf{k}\in\mathbb{N}^{|V(T)|}$. Since height$(T)<|V(T)|\leq B$ there is a finite number of possible colours for $A_s$, in particular implying that there is a finite number of possibilities for the coloured rooted tree $T$ when expressing any $H\in\mathcal{H}$ in this form.
Let $\{T_\ell:1\leq\ell\leq m\}$ be this finite set of coloured rooted trees (where colours encode
 adjacencies of a subgraph $G_\ell$ of ${\rm clos}(T_\ell)$), i.e., such that each $H\in\mathcal{H}$ is of the form $T_\ell^{\mathbf{k}_\ell}(G_\ell)$ for some $\mathbf{k}_\ell\in\mathbb{N}^{|V(T_\ell)|}$.

By taking $I_\ell=\{\mathbf{k}_\ell\in\mathbb{N}^{|V(T_\ell)|}:H= T_\ell^{\mathbf{k}}(G_\ell)\;\mbox{\rm for some}\; H\in\mathcal{H}\}$ we obtain statement (i) of the theorem.
\end{proof}

\begin{remark} 
Sometimes a family of graphs $\mathcal{H}$ already presented as a sequence $(H_{\mathbf{k}})$ for $\mathbf{k}\in \mathbb{N}^h$ may require re-indexing by a larger number of variables in order to realize it as a polynomial sequence. For example, in Figure~\ref{fig:polybranching} we illustrate an example of a sequence of graphs $(H_k)$ of bounded minimum branching core size
 indexed by a single parameter $k$, and for which ${\rm hom}(K_1,H_k)=2^{k+1}+1$ and ${\rm hom}(K_2,H_k)=2^{k+1}+k$.
This excludes the possibility that there might be a re-indexing function $f:\mathbb{N}\rightarrow\mathbb{N}$ such that for every graph $G$ we have ${\rm hom}(G,H_k)=p(G;f(k))$ for a polynomial $p(G)$.
However, if we re-index the sequence in terms of two parameters $k_1=k$ and $k_2=2^k-k$ then there is by Theorem~\ref{thm:poly_branching} a bivariate polynomial $p(G)$ such that ${\rm hom}(G,H_{k_1,k_2})=p(G;k_1,k_2)$.
(Compare the case of the sequence of hypercubes $(Q_k)$, Example~\ref{ex:cubes} in Section~\ref{examples}.) \end{remark}

An easy consequence of Theorem~\ref{thm:bounded_bushiness} is a criterion for a set of ornamented graphs to be decomposable into a finite number of strongly polynomial subsequences. 

\begin{corollary}\label{cor:ornamented}
Let $\mathcal{F}$ be a finite set of strongly polynomial graph sequences, each indexed by $\mathbf{j}\in\mathbb{N}^h$ (for some $h$), and $\mathcal{H}$ a family of ornamented graphs, such that each ornament on a vertex $s$ is a graph sequence $(F_{s;\mathbf{j}})$ belonging to~$\mathcal{F}$. 

Suppose that $\mathcal{H}$ has bounded minimum branching core size, i.e.,\newline $\{{\rm \ul{bc}}(H\langle\{(F_{s;\mathbf{j}})\}\rangle):H\langle\{F_{s;\mathbf{j}}\}\rangle\in\mathcal{H}\}$ is bounded.

Then the family of compositions $\{H[\{F_{s;\mathbf{j}}\}]:H\langle\{(F_{s;\mathbf{j}})\}\rangle\in\mathcal{H},\, \mathbf{j}\in\mathbb{N}^h\}$ can be partitioned into a finite number of strongly polynomial subsequences, $(H_{\mathbf{k}_\ell,\mathbf{j}}\,:\,\mathbf{k}_\ell\in I_\ell,\mathbf{j}\in\mathbb{N}^{h})$, for some $I_\ell\subseteq \mathbb{N}^{h_\ell}$ ($h_\ell\geq 1$).
Moreover,
\begin{itemize}
\item[(i)]  there is a rooted tree $T_\ell$ with $|V(T_\ell)|=h_\ell$, such that for every $\mathbf{k}_\ell\in I_\ell$ we have $H_{\mathbf{k}_\ell,\mathbf{j}}=T_\ell^{\mathbf{k}_\ell}(G_\ell)[\{F_{s;\mathbf{j}}\}]$ for some fixed graph $G_\ell\subseteq{\rm clos}(T_\ell)$, and
\item[(ii)] for each fixed $\mathbf{j}$, $(H_{\mathbf{k}_\ell,\mathbf{j}})$ is a strongly polynomial subsequence indexed by ${\mathbf{k}_\ell}\in I_\ell\subseteq\mathbb{N}^{h_\ell}$.
\end{itemize}
\end{corollary}

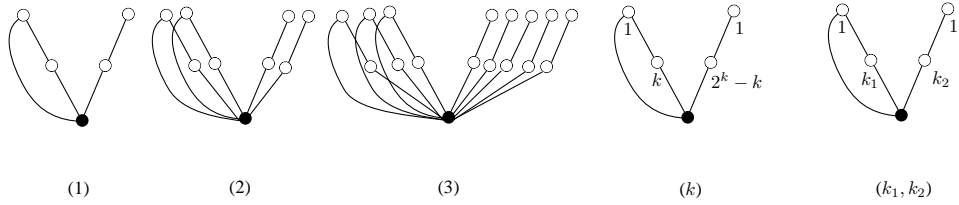
\begin{figure}
\begin{center}
\scalebox{0.7} 
{
\begin{pspicture}(0,-1.9184375)(18.401875,1.8984375)
\psdots[dotsize=0.24](1.54,-0.3615625)
\psline[linewidth=0.02cm](1.94,0.5984375)(1.56,-0.3015625)
\psline[linewidth=0.02cm](1.0,0.6184375)(1.5,-0.3015625)
\psline[linewidth=0.02cm](0.46,1.5784374)(0.96,0.6584375)
\psdots[dotsize=0.24,fillstyle=solid,dotstyle=o](0.42,1.6384375)
\psdots[dotsize=0.24,fillstyle=solid,dotstyle=o](0.96,0.6784375)
\psline[linewidth=0.02cm](2.38,1.5784374)(2.0,0.6784375)
\psdots[dotsize=0.24,fillstyle=solid,dotstyle=o](2.42,1.6584375)
\psdots[dotsize=0.24,fillstyle=solid,dotstyle=o](1.98,0.6784375)
\psbezier[linewidth=0.02](0.3321916,1.5584375)(0.0,1.3184375)(0.18,0.71358234)(0.46,0.2384375)(0.74,-0.23670734)(1.12,-0.3615625)(1.52,-0.3815625)
\psdots[dotsize=0.24](4.64,-0.3215625)
\psline[linewidth=0.02cm](5.04,0.6384375)(4.66,-0.2615625)
\psline[linewidth=0.02cm](4.1,0.6584375)(4.6,-0.2615625)
\psline[linewidth=0.02cm](3.56,1.6184375)(4.06,0.6984375)
\psdots[dotsize=0.24,fillstyle=solid,dotstyle=o](3.52,1.6784375)
\psdots[dotsize=0.24,fillstyle=solid,dotstyle=o](4.06,0.7184375)
\psline[linewidth=0.02cm](5.48,1.6184375)(5.1,0.7184375)
\psdots[dotsize=0.24,fillstyle=solid,dotstyle=o](5.46,1.6184375)
\psdots[dotsize=0.24,fillstyle=solid,dotstyle=o](5.08,0.7184375)
\psbezier[linewidth=0.02](3.4321916,1.5984375)(3.1,1.3584375)(3.28,0.75358236)(3.56,0.2784375)(3.84,-0.19670735)(4.22,-0.3215625)(4.62,-0.3415625)
\psline[linewidth=0.02cm](3.18,1.5784374)(3.68,0.6584375)
\psdots[dotsize=0.24,fillstyle=solid,dotstyle=o](3.14,1.6384375)
\psbezier[linewidth=0.02](3.0521915,1.5584375)(2.72,1.3184375)(2.9,0.71358234)(3.18,0.2384375)(3.46,-0.23670734)(4.2,-0.3215625)(4.64,-0.3815625)
\psline[linewidth=0.02cm](3.74,0.6184375)(4.52,-0.3615625)
\psdots[dotsize=0.24,fillstyle=solid,dotstyle=o](3.68,0.6784375)
\psline[linewidth=0.02cm](5.36,0.5584375)(4.6,-0.4215625)
\psline[linewidth=0.02cm](5.8,1.5384375)(5.42,0.6384375)
\psdots[dotsize=0.24,fillstyle=solid,dotstyle=o](5.84,1.6184375)
\psdots[dotsize=0.24,fillstyle=solid,dotstyle=o](5.4,0.6384375)
\psdots[dotsize=0.24](8.5,-0.3015625)
\psline[linewidth=0.02cm](8.96,0.6584375)(8.52,-0.2415625)
\psline[linewidth=0.02cm](7.96,0.6784375)(8.46,-0.2415625)
\psline[linewidth=0.02cm](7.42,1.6384375)(7.92,0.7184375)
\psdots[dotsize=0.24,fillstyle=solid,dotstyle=o](7.38,1.6984375)
\psdots[dotsize=0.24,fillstyle=solid,dotstyle=o](7.92,0.7384375)
\psline[linewidth=0.02cm](9.34,1.6384375)(8.96,0.7384375)
\psdots[dotsize=0.24,fillstyle=solid,dotstyle=o](9.32,1.6384375)
\psdots[dotsize=0.24,fillstyle=solid,dotstyle=o](8.98,0.7384375)
\psbezier[linewidth=0.02](7.2921915,1.6184375)(6.96,1.3784375)(7.14,0.77358234)(7.42,0.2984375)(7.7,-0.17670734)(8.08,-0.3015625)(8.48,-0.3215625)
\psline[linewidth=0.02cm](7.04,1.5984375)(7.54,0.6784375)
\psdots[dotsize=0.24,fillstyle=solid,dotstyle=o](7.0,1.6584375)
\psbezier[linewidth=0.02](6.9121914,1.5784374)(6.58,1.3384376)(6.76,0.7335824)(7.04,0.2584375)(7.32,-0.21670735)(8.06,-0.3015625)(8.5,-0.3615625)
\psline[linewidth=0.02cm](7.6,0.6384375)(8.38,-0.3415625)
\psdots[dotsize=0.24,fillstyle=solid,dotstyle=o](7.54,0.6984375)
\psline[linewidth=0.02cm](9.22,0.5784375)(8.46,-0.4015625)
\psline[linewidth=0.02cm](9.66,1.5584375)(9.28,0.6584375)
\psdots[dotsize=0.24,fillstyle=solid,dotstyle=o](9.7,1.6384375)
\psdots[dotsize=0.24,fillstyle=solid,dotstyle=o](9.28,0.6784375)
\psline[linewidth=0.02cm](6.52,1.5584375)(7.02,0.6384375)
\psdots[dotsize=0.24,fillstyle=solid,dotstyle=o](6.48,1.6184375)
\psbezier[linewidth=0.02](6.3921914,1.5384375)(6.06,1.2984375)(6.24,0.69358236)(6.52,0.2184375)(6.8,-0.25670734)(8.06,-0.2815625)(8.56,-0.3815625)
\psline[linewidth=0.02cm](7.08,0.5984375)(8.5,-0.4015625)
\psdots[dotsize=0.24,fillstyle=solid,dotstyle=o](7.02,0.6584375)
\psline[linewidth=0.02cm](10.1,1.6384375)(9.66,0.7184375)
\psdots[dotsize=0.24,fillstyle=solid,dotstyle=o](10.08,1.6184375)
\psline[linewidth=0.02cm](9.98,0.5784375)(8.58,-0.3615625)
\psline[linewidth=0.02cm](10.42,1.5584375)(10.04,0.6584375)
\psdots[dotsize=0.24,fillstyle=solid,dotstyle=o](10.46,1.6384375)
\psdots[dotsize=0.24,fillstyle=solid,dotstyle=o](10.02,0.6584375)
\psdots[dotsize=0.24,fillstyle=solid,dotstyle=o](9.6,0.6784375)
\psline[linewidth=0.02cm](9.5,0.5784375)(8.58,-0.3215625)
\psline[linewidth=0.02cm](10.8,1.5584375)(10.42,0.6584375)
\psdots[dotsize=0.24,fillstyle=solid,dotstyle=o](10.84,1.6384375)
\psdots[dotsize=0.24,fillstyle=solid,dotstyle=o](10.36,0.6584375)
\psline[linewidth=0.02cm](10.28,0.5584375)(8.52,-0.3815625)
\psdots[dotsize=0.24](13.04,-0.3015625)
\psline[linewidth=0.02cm](13.44,0.6584375)(13.06,-0.2415625)
\psline[linewidth=0.02cm](12.5,0.6784375)(13.0,-0.2415625)
\psline[linewidth=0.02cm](11.96,1.6384375)(12.46,0.7184375)
\psdots[dotsize=0.24,fillstyle=solid,dotstyle=o](11.92,1.6984375)
\psdots[dotsize=0.24,fillstyle=solid,dotstyle=o](12.46,0.7384375)
\psline[linewidth=0.02cm](13.88,1.6384375)(13.5,0.7384375)
\psdots[dotsize=0.24,fillstyle=solid,dotstyle=o](13.92,1.7184376)
\psdots[dotsize=0.24,fillstyle=solid,dotstyle=o](13.48,0.7384375)
\psbezier[linewidth=0.02](11.832191,1.6184375)(11.5,1.3784375)(11.68,0.77358234)(11.96,0.2984375)(12.24,-0.17670734)(12.62,-0.3015625)(13.02,-0.3215625)
\usefont{T1}{ptm}{m}{n}
\rput(12.4514065,0.4084375){$k$}
\usefont{T1}{ptm}{m}{n}
\rput(11.9114065,1.3684375){$1$}
\usefont{T1}{ptm}{m}{n}
\rput(14.031406,1.3884375){$1$}
\usefont{T1}{ptm}{m}{n}
\rput(13.971406,0.4){$2^k-k$}
\usefont{T1}{ptm}{m}{n}
\rput(1.4751563,-1.6715626){(1)}
\usefont{T1}{ptm}{m}{n}
\rput(4.5351562,-1.6515625){(2)}
\usefont{T1}{ptm}{m}{n}
\rput(8.495156,-1.6515625){(3)}
\usefont{T1}{ptm}{m}{n}
\rput(13.095157,-1.6915625){($k$)}
\psdots[dotsize=0.24](17.1,-0.2615625)
\psline[linewidth=0.02cm](17.5,0.6984375)(17.12,-0.2015625)
\psline[linewidth=0.02cm](16.56,0.7184375)(17.06,-0.2015625)
\psline[linewidth=0.02cm](16.02,1.6784375)(16.52,0.7584375)
\psdots[dotsize=0.24,fillstyle=solid,dotstyle=o](15.98,1.7384375)
\psdots[dotsize=0.24,fillstyle=solid,dotstyle=o](16.52,0.7784375)
\psline[linewidth=0.02cm](17.94,1.6784375)(17.56,0.7784375)
\psdots[dotsize=0.24,fillstyle=solid,dotstyle=o](17.98,1.7584375)
\psdots[dotsize=0.24,fillstyle=solid,dotstyle=o](17.54,0.7784375)
\psbezier[linewidth=0.02](15.892192,1.6584375)(15.56,1.4184375)(15.74,0.81358236)(16.02,0.3384375)(16.3,-0.13670735)(16.68,-0.2615625)(17.08,-0.2815625)
\usefont{T1}{ptm}{m}{n}
\rput(16.5,0.4){$k_1$}
\usefont{T1}{ptm}{m}{n}
\rput(15.971406,1.4084375){$1$}
\usefont{T1}{ptm}{m}{n}
\rput(18.091406,1.4284375){$1$}
\usefont{T1}{ptm}{m}{n}
\rput(17.851406,0.4284375){$k_2$}
\usefont{T1}{ptm}{m}{n}
\rput(17.145157,-1.6715626){($k_1, k_2$)}
\end{pspicture}
}
\caption{A sequence $(H_k)$ of graphs of branching core size $5$ that is not polynomial in the parameter $k$. (Number beside each non-root vertex indicate multiplicity of branchings.) However, when re-indexed by the pair of branching parameters $(k_1, k_2)$, the sequence $(H_{k,2^k-k})$ is polynomial in $k$ and $2^k-k$ and a subsequence of the strongly polynomial sequence $(H_{k_1,k_2})$ (illustrated as the right-most graph).}\label{fig:polybranching}
\end{center}
\end{figure}


Let us return to the unornamented graphs of Theorem~\ref{thm:bounded_bushiness} and the question of whether some partial converse is possible.

There are families of graphs of unbounded minimum branching core size, some of which can still be partitioned into a finite number of polynomial subsequences such as that in Figure~\ref{fig:unbounded_bc}(c) and Figure~\ref{fig:brush and k-ary tree}(c), and others such as  Figure~\ref{fig:unbounded_bc}(a) and  Figure~\ref{fig:brush and k-ary tree}(a) and~\ref{fig:brush and k-ary tree}(b) which cannot be so partitioned.

The family illustrated in Figure~\ref{fig:unbounded_bc}(a), in which $k_1,\ldots, k_\ell$ are independent variables and $\ell$ unbounded, is of bounded tree-depth (each graph has tree-depth $3$) but cannot be partitioned into a finite number of polynomial subsequences; the branching core size of these graphs is equal to $2\ell+1$ and hence unbounded. On the other hand, for each graph $G$ there is a multivariate polynomial in infinitely many variables, $p(G;x_1,x_2,\ldots)$ such that the number of homomorphisms from $G$ to the graph in Figure~\ref{fig:unbounded_bc}(a) is an evaluation $p(G;k_1,\ldots, k_\ell,0,0,\ldots)$ of this polynomial. 
The same is true of the analogously defined family in Figure~\ref{fig:brush and k-ary tree}(a) (and here the tree-depth is equal to $\ell+1$, equal to the branching core size).

The closure of the tree in Figure~\ref{fig:unbounded_bc}(c) (let us denote it by $H_{k,\ell}$)  is the lexicographic product of $K_{1,\ell}$ with graphs $K_{1,p_1(k)},K_{1,p_2(k)},\ldots, K_{1,p_\ell(k)}$ on its leaves $v_1,\ldots, v_\ell$ and $K_1$ on its central vertex $v_0$. Thus, in a similar way to the proof of Proposition~3.8, the number of homomorphisms from $G$ to this graph $H_{k,\ell}$ is seen to be given by
\begin{equation}\label{eq:starstarclos}\sum_{\stackrel{f:G\rightarrow K_{1,\ell}[\{K_1^1\}]}{\mbox{\rm \tiny homomorphism}}}{\rm hom}(G[f^{-1}(\{v_0\})],K_1)\prod_{1\leq i\leq\ell}{\rm hom}(G[f^{-1}(\{v_i\})], K_{1,p_i(k)}),\end{equation}
where $G[f^{-1}[\{v_i\}]$ is the induced subgraph of $G$ on vertices sent by homomorphism $f$ to $v_i$. There are at most $|V(G)|$ factors not equal to $1$ or $0$ in the sum~\eqref{eq:starstarclos}. Furthermore, the expression~\eqref{eq:starstarclos} is symmetric in variables $p_1(k), p_2(k), \ldots, p_\ell(k)$ since the leaves $v_1,v_2,\ldots, v_\ell$ of $K_{1,\ell}$ are twin vertices (so we have by permutation of these vertices an action on homomorphisms from $G$).
For each $1\leq i\leq \ell$, the sequence $(K_{1,p_i(k)})$ is strongly polynomial in variable $p_i(k)$, with ${\rm hom}(G,K_{1,p_i(k)})=0$ when $G$ is non-bipartite and, when $G$ is bipartite,
 $${\rm hom}(G,K_{1,p_i(k)})=\sum_{\stackrel{U\subseteq V}{\mbox{\rm \tiny $U$, $V\backslash U$ independent}}}p_i(k)^{|V\setminus U|}.$$

The sum~\eqref{eq:starstarclos} is thus a symmetric polynomial in variables $p_1(k),p_2(k),\ldots, p_\ell(k)$, of degree at most $|V(G)|$.
Since $p_i(k)=p(i,k)$ for polynomial $p$, the sum~\eqref{eq:starstarclos} is thus a sum of polynomials in $k$ and
$1^r+2^r+\cdots + \ell^r$, for $0\leq r\leq d|V(G)|$, where $d$ is the maximum degree of $p(i,k)$ in variable $i$. Finally, since
$1^r+2^r+\cdots + \ell^r$ is a polynomial in $\ell$ of degree $r+1$, this makes the sum~\eqref{eq:starstarclos} a polynomial in $k$ and $\ell$, of degree at most $(d+1)|V(G)|$ in $\ell$. 

Is a partial converse of Theorem~\ref{thm:bounded_bushiness} to actually just require that the number of non-isomorphic subtrees {\em ignoring branching multiplicities} is finite? (In other words, reduce to branching core, and then count as isomorphic those subtrees which only differ on branching multiplicities.) For an unbounded number of some given subtree type pendant from vertex $s$ we have colour-isomorphism (for colours encoding subgraph of tree closure), but different branching numbers -- if this set of branching numbers can be expressed as values $p_s(i)$ of a polynomial in $i$ (with $i$ taking values $1,2,\dots, \ell_s$, variable $\ell_s$ unbounded being the number of isomorphic copies of subtree pendant from $s$) then we can write these graphs as strongly polynomial sequence of the form  $(T^{\mathbf{k}})^{\mathbf{p}}(H)$ -- otherwise, as a subsequence of such a strongly polynomial sequence.

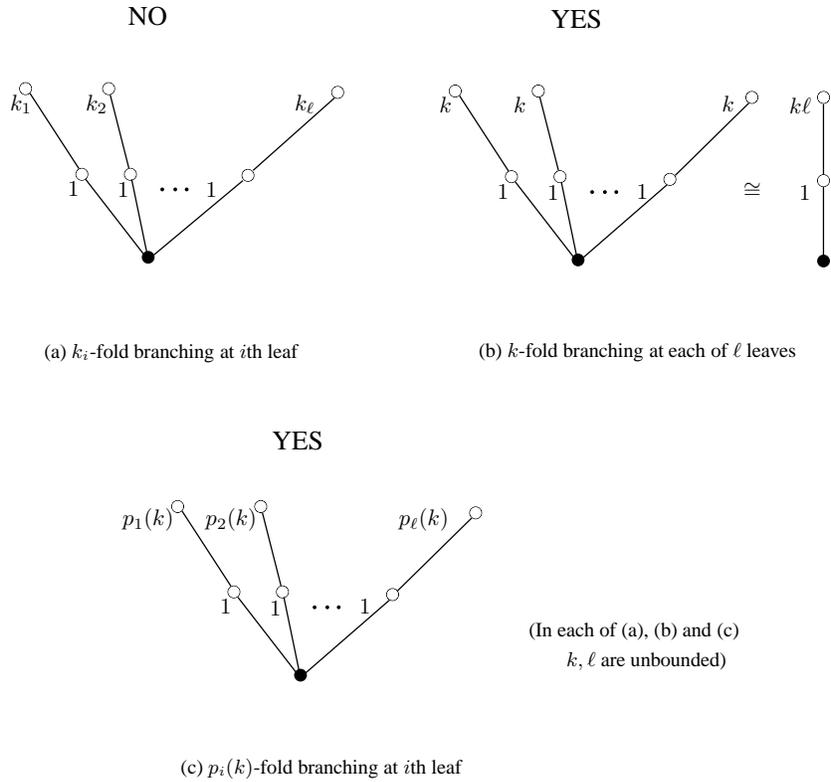
\begin{figure}
\begin{center}
\scalebox{0.85} 
{
\begin{pspicture}(0,-6.0957813)(13.24125,6.0957813)
\psline[linewidth=0.02cm](5.3209376,4.617656)(4.0409374,3.4776564)
\psline[linewidth=0.02cm](3.8809376,3.3176563)(2.4209375,2.0976562)
\psline[linewidth=0.02cm](1.4209375,3.3376563)(2.3809376,2.0976562)
\psline[linewidth=0.02cm](1.8009375,4.677656)(2.1009376,3.4976563)
\psline[linewidth=0.02cm](0.5409375,4.697656)(1.3209375,3.4976563)
\psdots[dotsize=0.2](2.4009376,2.1176562)
\psdots[dotsize=0.2,fillstyle=solid,dotstyle=o](0.4809375,4.757656)
\psdots[dotsize=0.2,fillstyle=solid,dotstyle=o](1.7809376,4.757656)
\psdots[dotsize=0.2,fillstyle=solid,dotstyle=o](5.3609376,4.697656)
\psdots[dotsize=0.2,fillstyle=solid,dotstyle=o](1.3609375,3.4176562)
\psdots[dotsize=0.2,fillstyle=solid,dotstyle=o](2.1209376,3.4176562)
\psdots[dotsize=0.2,fillstyle=solid,dotstyle=o](3.9609375,3.3976562)
\psline[linewidth=0.02cm](2.1409376,3.3376563)(2.3809376,2.1776562)
\psdots[dotsize=0.06](2.6209376,3.1776562)
\psdots[dotsize=0.06](2.8209374,3.1776562)
\psdots[dotsize=0.06](3.0009375,3.1776562)
\usefont{T1}{ptm}{m}{n}
\rput(0.41234374,4.507656){$k_1$}
\usefont{T1}{ptm}{m}{n}
\rput(1.5923438,4.507656){$k_2$}
\usefont{T1}{ptm}{m}{n}
\rput(4.8523436,4.507656){$k_\ell$}
\usefont{T1}{ptm}{m}{n}
\rput(1.2323438,3.1876562){$1$}
\usefont{T1}{ptm}{m}{n}
\rput(2.0123436,3.1676562){$1$}
\usefont{T1}{ptm}{m}{n}
\rput(3.3923438,3.1676562){$1$}
\usefont{T1}{ptm}{m}{n}
\rput(2.7578125,0.61765623){\small (a) $k_i$-fold branching at $i$th leaf}
\usefont{T1}{ptm}{m}{n}
\rput(2.394375,5.8926563){\large NO}
\psline[linewidth=0.02cm](11.800938,4.5376563)(10.640938,3.3976562)
\psline[linewidth=0.02cm](10.520938,3.2576563)(9.140938,2.0576563)
\psline[linewidth=0.02cm](8.140938,3.2976563)(9.100938,2.0576563)
\psline[linewidth=0.02cm](8.520938,4.637656)(8.820937,3.4576561)
\psline[linewidth=0.02cm](7.2609377,4.657656)(8.040937,3.4576561)
\psdots[dotsize=0.2](9.120937,2.0776563)
\psdots[dotsize=0.2,fillstyle=solid,dotstyle=o](7.2009373,4.717656)
\psdots[dotsize=0.2,fillstyle=solid,dotstyle=o](8.500937,4.717656)
\psdots[dotsize=0.2,fillstyle=solid,dotstyle=o](11.840938,4.617656)
\psdots[dotsize=0.2,fillstyle=solid,dotstyle=o](8.080937,3.3776562)
\psdots[dotsize=0.2,fillstyle=solid,dotstyle=o](8.840938,3.3776562)
\psdots[dotsize=0.2,fillstyle=solid,dotstyle=o](10.560938,3.3376563)
\psline[linewidth=0.02cm](8.860937,3.2976563)(9.100938,2.1376562)
\psdots[dotsize=0.06](9.340938,3.1376562)
\psdots[dotsize=0.06](9.540937,3.1376562)
\psdots[dotsize=0.06](9.720938,3.1376562)
\usefont{T1}{ptm}{m}{n}
\rput(7.052344,4.507656){$k$}
\usefont{T1}{ptm}{m}{n}
\rput(8.212344,4.507656){$k$}
\usefont{T1}{ptm}{m}{n}
\rput(11.472343,4.527656){$k$}
\usefont{T1}{ptm}{m}{n}
\rput(7.952344,3.1476562){$1$}
\usefont{T1}{ptm}{m}{n}
\rput(8.732344,3.1276562){$1$}
\usefont{T1}{ptm}{m}{n}
\rput(10.132343,3.1476562){$1$}
\usefont{T1}{ptm}{m}{n}
\rput(10.054063,0.6376563){\small (b) $k$-fold branching at each of $\ell$ leaves}
\usefont{T1}{ptm}{m}{n}
\rput(9.099375,5.8526564){\large YES}
\usefont{T1}{ptm}{m}{n}
\rput(11.852344,3.1676562){$\cong$}
\psline[linewidth=0.02cm](12.9609375,4.5176563)(12.9609375,3.3776562)
\psdots[dotsize=0.2,fillstyle=solid,dotstyle=o](12.9609375,4.617656)
\psline[linewidth=0.02cm](12.9609375,3.2176561)(12.9609375,2.0776563)
\psdots[dotsize=0.2,fillstyle=solid,dotstyle=o](12.9609375,3.3176563)
\psdots[dotsize=0.2](12.9609375,2.0576563)
\usefont{T1}{ptm}{m}{n}
\rput(12.602344,4.487656){$k\ell$}
\usefont{T1}{ptm}{m}{n}
\rput(12.672344,3.1276562){$1$}
\psline[linewidth=0.02cm](7.4609375,-1.9623437)(6.3009377,-3.1023438)
\psline[linewidth=0.02cm](6.1809373,-3.2423437)(4.8009377,-4.4423437)
\psline[linewidth=0.02cm](3.8009374,-3.2023437)(4.7609377,-4.4423437)
\psline[linewidth=0.02cm](4.1809373,-1.8623438)(4.4809375,-3.0423439)
\psline[linewidth=0.02cm](2.9209375,-1.8423438)(3.7009375,-3.0423439)
\psdots[dotsize=0.2](4.7809377,-4.4223437)
\psdots[dotsize=0.2,fillstyle=solid,dotstyle=o](2.8609376,-1.7823437)
\psdots[dotsize=0.2,fillstyle=solid,dotstyle=o](4.1609373,-1.7823437)
\psdots[dotsize=0.2,fillstyle=solid,dotstyle=o](7.5209374,-1.8823438)
\psdots[dotsize=0.2,fillstyle=solid,dotstyle=o](3.7409375,-3.1223438)
\psdots[dotsize=0.2,fillstyle=solid,dotstyle=o](4.5009375,-3.1223438)
\psdots[dotsize=0.2,fillstyle=solid,dotstyle=o](6.2209377,-3.1623437)
\psline[linewidth=0.02cm](4.5209374,-3.2023437)(4.7609377,-4.362344)
\psdots[dotsize=0.06](5.0009375,-3.3623438)
\psdots[dotsize=0.06](5.2009373,-3.3623438)
\psdots[dotsize=0.06](5.3809376,-3.3623438)
\usefont{T1}{ptm}{m}{n}
\rput(2.4,-1.9923438){$p_1(k)$}
\usefont{T1}{ptm}{m}{n}
\rput(3.7,-1.9923438){$p_2(k)$}
\usefont{T1}{ptm}{m}{n}
\rput(6.7123437,-1.9923438){$p_\ell(k)$}
\usefont{T1}{ptm}{m}{n}
\rput(3.6123438,-3.3523438){$1$}
\usefont{T1}{ptm}{m}{n}
\rput(4.3923435,-3.3723438){$1$}
\usefont{T1}{ptm}{m}{n}
\rput(5.7923436,-3.3523438){$1$}
\usefont{T1}{ptm}{m}{n}
\rput(5.1078124,-5.882344){\small (c) $p_i(k)$-fold branching at $i$th leaf}
\usefont{T1}{ptm}{m}{n}
\rput(4.739375,-0.7473438){\large YES}
\usefont{T1}{ptm}{m}{n}
\rput(10,-3.7323437){\small (In each of (a), (b) and (c)}
\usefont{T1}{ptm}{m}{n}
\rput(10.2,-4.2123437){\small $k, \ell$ are unbounded)}
\end{pspicture}
}
\caption{\small Which families can be partitioned into finitely many polynomial sequences? (a)~Branching $k_i$-fold for $i=1,2,\ldots$ at leaf vertices here gives a family of graphs with no partition into finitely many polynomial subsequences.
(b) Setting $k_i=k$ for each $1\leq i\leq\ell$ makes each graph in the family now have minimum branching core size $3$ and yields a strongly polynomial sequence in $k\ell$. (c) Setting $k_i=p_i(k)$ for $1\leq i\leq \ell$, where $p_i(k)$ is a polynomial in $i$ and $k$, gives a graph of minimum branching core size $2\ell+1$ (when the $p_i(k)$ are distinct, e.g. $p_i(k)=k+i$) and a strongly polynomial sequence in two variables $k$ and $\ell$.}\label{fig:unbounded_bc}
\end{center}
\end{figure}

\begin{figure}
\begin{center}
\scalebox{0.8} 
{
\begin{pspicture}(0,-4.0157814)(14.762188,4.0157814)
\psline[linewidth=0.02cm](1.6271875,3.0976562)(1.6271875,2.0976562)
\psline[linewidth=0.02cm](7.5671873,-0.38234374)(7.5671873,-1.3823438)
\psline[linewidth=0.02cm](7.5671873,-1.4823438)(7.5671873,-2.4823437)
\psdots[dotsize=0.2](7.5671873,-2.5223436)
\psdots[dotsize=0.2,fillstyle=solid,dotstyle=o](7.5671873,-1.4223437)
\psdots[dotsize=0.2,fillstyle=solid,dotstyle=o](7.5671873,-0.32234374)
\psdots[dotsize=0.06](7.5071874,0.39765626)
\psdots[dotsize=0.06](7.5071874,0.53765625)
\psdots[dotsize=0.06](7.5071874,0.6976563)
\psline[linewidth=0.02cm](7.5071874,2.9976563)(7.5071874,1.9976562)
\psdots[dotsize=0.2,fillstyle=solid,dotstyle=o](7.5071874,3.0576563)
\usefont{T1}{ptm}{m}{n}
\rput(7.358594,2.7676563){$k$}
\usefont{T1}{ptm}{m}{n}
\rput(7.398594,-0.65234375){$k$}
\usefont{T1}{ptm}{m}{n}
\rput(7.378594,-1.7523438){$k$}
\usefont{T1}{ptm}{m}{n}
\rput(7.5,-3.4023438){\small (b) $k$-fold branching at non-root vertices}
\usefont{T1}{ptm}{m}{n}
\rput(7.5598435,-3.8023438){\small (perfect $k$-ary tree of height $\ell$)}
\psline[linewidth=0.02cm](1.6671875,-0.36234376)(1.6671875,-1.3623438)
\psline[linewidth=0.02cm](1.6671875,-1.4623437)(1.6671875,-2.4623437)
\psdots[dotsize=0.2](1.6671875,-2.5023437)
\psdots[dotsize=0.2,fillstyle=solid,dotstyle=o](1.6671875,-1.4023438)
\psdots[dotsize=0.2,fillstyle=solid,dotstyle=o](1.6671875,-0.30234376)
\psdots[dotsize=0.06](1.6271875,0.33765626)
\psdots[dotsize=0.06](1.6271875,0.47765625)
\psdots[dotsize=0.06](1.6271875,0.6376563)
\psline[linewidth=0.02cm](1.6271875,2.0176563)(1.6271875,1.1576562)
\psdots[dotsize=0.2,fillstyle=solid,dotstyle=o](1.6271875,2.0776563)
\usefont{T1}{ptm}{m}{n}
\rput(2.1,1.8076563){$k_{\ell-1}$}
\usefont{T1}{ptm}{m}{n}
\rput(2.05,-0.59234375){$k_2$}
\usefont{T1}{ptm}{m}{n}
\rput(2.05,-1.6923437){$k_1$}
\psdots[dotsize=0.2,fillstyle=solid,dotstyle=o](1.6271875,3.1576562)
\usefont{T1}{ptm}{m}{n}
\rput(2,2.8676562){$k_\ell$}
\usefont{T1}{ptm}{m}{n}
\rput(13.2,-3.4023438){\small (c) $k$-fold branching at}
\psline[linewidth=0.02cm](7.5271873,1.8976562)(7.5271873,1.0376563)
\psdots[dotsize=0.2,fillstyle=solid,dotstyle=o](7.5271873,1.9576563)
\usefont{T1}{ptm}{m}{n}
\rput(7.378594,1.7076563){$k$}
\usefont{T1}{ptm}{m}{n}
\rput(13.274844,-3.7823439){\small end-vertex of path}
\usefont{T1}{ptm}{m}{n}
\rput(1.6,-3.3823438){\small (a) $k_i$-fold branching}
\psline[linewidth=0.02cm](13.307187,2.9976563)(13.307187,1.9976562)
\psline[linewidth=0.02cm](13.347187,-0.46234375)(13.347187,-1.4623437)
\psline[linewidth=0.02cm](13.347187,-1.5623437)(13.347187,-2.5623438)
\psdots[dotsize=0.2](13.347187,-2.6023438)
\psdots[dotsize=0.2,fillstyle=solid,dotstyle=o](13.347187,-1.5023438)
\psdots[dotsize=0.2,fillstyle=solid,dotstyle=o](13.347187,-0.40234375)
\psdots[dotsize=0.06](13.307187,0.23765625)
\psdots[dotsize=0.06](13.307187,0.37765625)
\psdots[dotsize=0.06](13.307187,0.53765625)
\psline[linewidth=0.02cm](13.307187,1.9176563)(13.307187,1.0576563)
\psdots[dotsize=0.2,fillstyle=solid,dotstyle=o](13.307187,1.9776562)
\usefont{T1}{ptm}{m}{n}
\rput(13.198594,1.6476562){$1$}
\usefont{T1}{ptm}{m}{n}
\rput(13.218594,-0.67234373){$1$}
\usefont{T1}{ptm}{m}{n}
\rput(13.218594,-1.8123437){$1$}
\psdots[dotsize=0.2,fillstyle=solid,dotstyle=o](13.307187,3.0576563)
\usefont{T1}{ptm}{m}{n}
\rput(13.118594,2.7276564){$k$}
\usefont{T1}{ptm}{m}{n}
\rput(1.6510937,-3.7723436){\small at $i$th vertex from root}
\usefont{T1}{ptm}{m}{n}
\rput(1.660625,3.7926562){\large NO}
\usefont{T1}{ptm}{m}{n}
\rput(7.480625,3.8126562){\large NO}
\usefont{T1}{ptm}{m}{n}
\rput(13.265625,3.8126562){\large YES}
\end{pspicture}
}
\caption{\small Which families can be partitioned into finitely many polynomial sequences? (a)~Paths of unbounded length $\ell$ with $i$th vertex from root branched $k_i$-fold: this family has no partition into a finite number of polynomial subsequences. (b)~Perfect $k$-ary trees of height $\ell$: closures of these rooted trees do not form a polynomial sequence in $k$ and $\ell$ (but for fixed $\ell$ they form a strongly polynomial sequence in $k$; also, for $k=1$, ${\rm clos}(P_{\ell})=K_\ell$ and hom$(G,K_\ell)=P(G;\ell)$, the value of the chromatic polynomial of $G$ at $\ell$). (c)~Paths of unbounded length $\ell$ with $k$ vertices of degree $1$ added adjacent to an endpoint: closures of these rooted trees, $\overline{K}_k+K_\ell$, form a strongly polynomial sequence in $k$ and $\ell$.}\label{fig:brush and k-ary tree}

\end{center}
\end{figure}
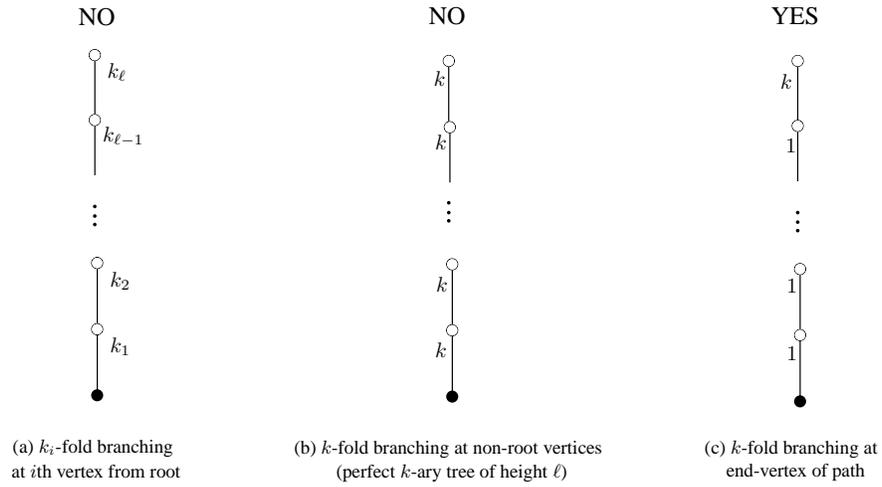

\section{Concluding remarks}\label{sec:other}

The graph polynomials $p(G;\mathbf{x})$ that we consider in this paper that are defined by $p(G;\mathbf{k})={\rm hom}(G,H_\mathbf{k})$ for a strongly polynomial sequence $(H_{\mathbf{k}})$ do not in general fall into the class of ``generalized chromatic polynomials'' of~\cite[Sect. 6]{M08} (see also~\cite[Sect. 3]{KMZ08}), due to the simple reason that the set of colourings (homomorphisms to $H_{\mathbf{k}}$) counted by $p(G;\mathbf{k})$ is not usually closed under permutation of the colour set (vertices of $H_{\mathbf{k}}$). An exceptional case is the sequence of complete graphs $(K_k)$,  which as we have seen in Example~\ref{ex:hom} give the chromatic polynomial. 

In this paper we have built strongly polynomial graph sequences by the following basic operations:
loop additions/removals and complementation (Section~\ref{sec:complementation}), blow-ups and compositions (Section~\ref{sec:ornamentation}) --- the operations thus far are used in the examples contained in Section~\ref{examples} --- and, finally and constituting the most important contribution of our paper, by branching (Section~\ref{sec:branching}).
There are however examples of graph polynomials defined by strongly polynomial graph sequences which fit into neither the general framework of our paper (construction of strongly polynomial sequences by the aforementioned basic operations) nor the ``zoology'' of~\cite{M08}. We give here two examples.

First, in Proposition~\ref{prop:hypercubes} we have seen that the sequence of hypercubes $(Q_k)$ determines a bivariate polynomial $p(G;x_1,x_2)$ such that ${\rm hom}(G, Q_k)=p(G;k,2^k)$.  A homomorphism from $G$ to $Q_k$ is the same as a colouring of the vertices of $G$ with subsets of $[k]$ having the property that colours on adjacent vertices receive colours that differ only by the addition of removal of a single element. The graph polynomial $p(G;x_1,x_2)$ thus defined by hypercube-colourings does not belong to the family of generalized chromatic polynomials of~\cite{KMZ08}, nor can the sequence $(Q_k)$ be generated by blow-ups, loop additions/removals, complementations, compositions and branchings alone. 

A second example comes from~\cite[Ex. B8]{dlHJ95}. The {\em generalized Johnson graph} $(J_{k,\ell,D})$, $1\leq\ell\leq k\in\mathbb{N}, \emptyset\subset D\subseteq\{0,1,\ldots, \ell-1\}$, is the graph whose vertices are subsets of $\{1,2,\ldots, k\}$ of size $\ell$, two vertices being adjacent if and only if their intersection has size belonging to $D$. For fixed $\ell, D$, the sequence $(J_{k,\ell,D})$ is shown in~\cite[Prop. 3]{dlHJ95} to form a polynomial sequence in $k$ (in fact it can be shown to be strongly polynomial, in a similar way to our proof for hypercubes in Proposition~\ref{prop:hypercubes} below). The graph polynomial $p(G;x)$ defined by $p(G;k)={\rm hom}(G,J_{k,\ell,D})$ does not belong to the family of generalized chromatic polynomials of~\cite{KMZ08}; neither can the sequence $(J_{k,\ell,D})$ (apart from the case $\ell=1$ when $J_{k,1,\{0\}}=K_k$)  be constructed solely by the operations of blow-ups, loop additions/removals, complementations, compositions and branchings. 


\subsection{Branching operations on other types of coloured rooted tree}\label{sec:blah}
The operation of branching coloured rooted trees applies quite generally, independent of how the colours are interpreted. The object of Section~\ref{sec:branching} was to prove that branching produces strongly polynomial sequences of graphs when the colours of vertices of $T$ encode a subgraph $H$ of ${\rm clos}(T)$ by taking appropriate subsets of $\{0,1,\ldots, {\rm height}(T)-1\}$. We then further considered augmenting this colouring by an ornament (member of a strongly polynomial sequence of graphs) to prove the same for branched compositions (Theorem~\ref{thm:poly_branching}).

There are other ways to encode a graph by coloured rooted trees and it is natural to ask whether the branching operation produces a strongly polynomial sequence in these cases too.
We give an answer for cographs (encoded by cotrees) below as this is easy to describe in limited space. 

This might be extended to clique-width expression trees~\cite{CO00}, or to rooted trees encoding $m$-partite cographs~\cite{GHNOMR12}. However, even if we show  that these ways of encoding graphs produce strongly polynomial sequences by branching, it still leaves such examples as the generalized Johnson graphs mentioned in Section~\ref{sec:blah} unaccounted for.

We therefore 
content ourselves here with just the example of cotrees to illsutrate the scope of generating strongly polynomial sequences by branching coloured rooted trees. 

\subsection{Cographs and cotrees}
Cographs are generated from $K_1$ by complementation and disjoint union (and, since $G+H=\overline{\overline{G}\cup\overline{H}}$, joins too). Complete graphs and complete bipartite graphs are basic examples of cographs.  We state here some well known properties of cographs. 


A cograph $H$ can be represented by a cotree, a rooted tree $T$ whose leaves are vertices of $H$ and whose non-leaf vertices are labelled $0$ or $1$, representing respectively disjoint union and join. A subtree rooted at a vertex $s$ labelled $0$ corresponds to the disjoint union of subgraphs defined by the children of $s$, while if $s$ is labelled $1$ then the join is taken instead.
(Two vertices of $H$ are connected by an edge if and only if the lowest common ancestor of the corresponding leaves in $T$ is labeled by $1$.)
The representation of $G$ by a cotree is unique if we require the labels on any root-leaf path of this tree to alternate between 0 and 1. 
Switching everywhere labels $0$ and $1$ corresponds to complementation.

\begin{figure}
\begin{center}
\scalebox{0.8} 
{
\begin{pspicture}(0,-2.7567186)(16.315,2.7367187)
\psline[linewidth=0.02cm](6.8621874,0.35671875)(5.9421873,-1.2032813)
\psline[linewidth=0.02cm](2.9821875,0.33671874)(2.9821875,-1.0632813)
\psdots[dotsize=0.44](3.0021875,0.51671875)
\usefont{T1}{ptm}{m}{n}
\rput(2.9735937,-1.4332813){$1$}
\usefont{T1}{ptm}{m}{n}
\rput(2.9296875,-2.5032814){\small $K_k$}
\usefont{T1}{ptm}{m}{n}
\rput(0.91359377,-0.01328125){$k$}
\pscircle[linewidth=0.02,dimen=outer](2.9821875,-1.4032812){0.32}
\psline[linewidth=0.02cm](1.2621875,0.33671874)(1.2621875,-1.0632813)
\psdots[dotsize=0.44](1.2821875,0.51671875)
\usefont{T1}{ptm}{m}{n}
\rput(1.2535938,-1.4332813){$0$}
\pscircle[linewidth=0.02,dimen=outer](1.2621875,-1.4032812){0.32}
\usefont{T1}{ptm}{m}{n}
\rput(2.7135937,0.00671875){$k$}
\usefont{T1}{ptm}{m}{n}
\rput(1.1996875,-2.5032814){\small $\overline{K_k}$}
\psline[linewidth=0.02cm](5.1221876,0.31671876)(5.7421875,-1.1632812)
\psdots[dotsize=0.44](5.0021877,2.4367187)
\usefont{T1}{ptm}{m}{n}
\rput(5.813594,-1.4932812){$1$}
\pscircle[linewidth=0.02,dimen=outer](5.8221874,-1.4832813){0.32}
\usefont{T1}{ptm}{m}{n}
\rput(4.6535935,1.9467187){$k$}
\usefont{T1}{ptm}{m}{n}
\rput(5.0135937,0.58671874){$0$}
\pscircle[linewidth=0.02,dimen=outer](5.0221877,0.61671877){0.32}
\psline[linewidth=0.02cm](5.0021877,2.3367188)(5.0021877,0.93671876)
\psdots[dotsize=0.44](7.0021877,2.4367187)
\usefont{T1}{ptm}{m}{n}
\rput(6.663594,1.9267187){$l$}
\usefont{T1}{ptm}{m}{n}
\rput(6.9935937,0.62671876){$0$}
\pscircle[linewidth=0.02,dimen=outer](7.0021877,0.65671873){0.32}
\psline[linewidth=0.02cm](6.9821873,2.3767188)(6.9821873,0.9767187)
\usefont{T1}{ptm}{m}{n}
\rput(5.8796873,-2.5432813){\small $K_{k,l}$}
\psline[linewidth=0.02cm](10.462188,0.41671875)(9.542188,-1.1432812)
\psline[linewidth=0.02cm](8.722187,0.37671876)(9.342188,-1.1032813)
\psdots[dotsize=0.44](8.602187,2.4567187)
\usefont{T1}{ptm}{m}{n}
\rput(9.413593,-1.4332813){$0$}
\pscircle[linewidth=0.02,dimen=outer](9.422188,-1.4232812){0.32}
\usefont{T1}{ptm}{m}{n}
\rput(8.253593,2.0067186){$k$}
\usefont{T1}{ptm}{m}{n}
\rput(8.613594,0.64671874){$1$}
\pscircle[linewidth=0.02,dimen=outer](8.622188,0.6767188){0.32}
\psline[linewidth=0.02cm](8.602187,2.3967187)(8.602187,0.99671876)
\psdots[dotsize=0.44](10.582188,2.4767187)
\usefont{T1}{ptm}{m}{n}
\rput(10.263594,1.9867188){$l$}
\usefont{T1}{ptm}{m}{n}
\rput(10.593594,0.68671876){$1$}
\pscircle[linewidth=0.02,dimen=outer](10.602187,0.71671873){0.32}
\psline[linewidth=0.02cm](10.582188,2.4367187)(10.582188,1.0367187)
\usefont{T1}{ptm}{m}{n}
\rput(9.439688,-2.5232813){\small $K_k\cup K_l$}
\psline[linewidth=0.02cm](13.542188,0.35671875)(13.602187,-1.1432812)
\psdots[dotsize=0.44](13.522187,2.4967186)
\usefont{T1}{ptm}{m}{n}
\rput(13.573594,-1.4732813){$1$}
\pscircle[linewidth=0.02,dimen=outer](13.582188,-1.4632813){0.32}
\usefont{T1}{ptm}{m}{n}
\rput(13.2135935,2.0467188){$k$}
\usefont{T1}{ptm}{m}{n}
\rput(13.533594,0.64671874){$0$}
\pscircle[linewidth=0.02,dimen=outer](13.542188,0.6767188){0.32}
\psline[linewidth=0.02cm](13.522187,2.3967187)(13.522187,0.99671876)
\usefont{T1}{ptm}{m}{n}
\rput(13.849688,-2.5032814){\small $K_{k,\ldots, k}=\overline{lK_k}$}
\usefont{T1}{ptm}{m}{n}
\rput(13.203594,0.10671875){$l$}
\end{pspicture}
}
\caption{Cotrees with branching. Non-leaf nodes labelled $0$ (disjoint union) or $1$ (join). Branching label indicates multiplicity of branching operation to be applied-- for example in the first cotree, $k$ copies of the leaf are produced, whose disjoint union is taken in order to form the cograph $\overline{K_k}$.}\label{fig:cograph_branching}
\end{center}
\end{figure}
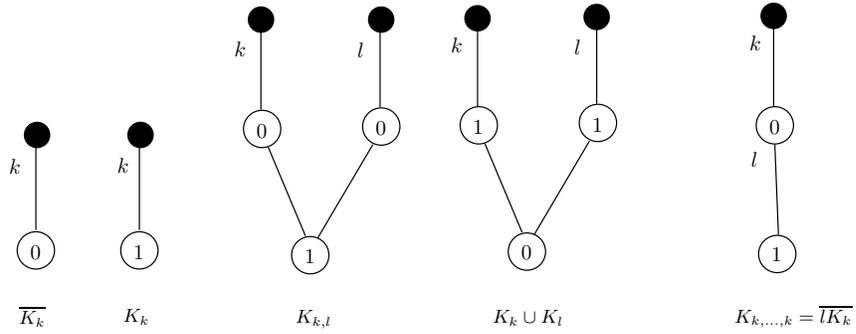

\begin{figure}
\begin{center}
\scalebox{0.8} 
{
\begin{pspicture}(0,-3.7767189)(11.555,3.7567186)
\psline[linewidth=0.02cm](3.8821876,1.3767188)(3.9421875,-0.12328125)
\psdots[dotsize=0.44](3.8621874,3.5167189)
\usefont{T1}{ptm}{m}{n}
\rput(3.9135938,-0.45328125){$1$}
\pscircle[linewidth=0.02,dimen=outer](3.9221876,-0.44328126){0.32}
\usefont{T1}{ptm}{m}{n}
\rput(3.5535936,3.0667188){$k$}
\usefont{T1}{ptm}{m}{n}
\rput(3.8735938,1.6667187){$0$}
\pscircle[linewidth=0.02,dimen=outer](3.8821876,1.6967187){0.32}
\psline[linewidth=0.02cm](3.8621874,3.4167187)(3.8621874,2.0167189)
\usefont{T1}{ptm}{m}{n}
\rput(4.0396876,-3.5632813){\small $mK_{k,\ldots, k}=\overline{K_m}[K_l[\overline{K_k}]]$}
\usefont{T1}{ptm}{m}{n}
\rput(3.5835938,1.0867188){$l$}
\psline[linewidth=0.02cm](3.9221876,-0.74328125)(3.9821875,-2.2432814)
\usefont{T1}{ptm}{m}{n}
\rput(3.9535937,-2.5732813){$0$}
\pscircle[linewidth=0.02,dimen=outer](3.9621875,-2.5632813){0.32}
\usefont{T1}{ptm}{m}{n}
\rput(3.6635938,-1.0732813){$m$}
\psline[linewidth=0.02cm](9.462188,1.3767188)(9.522187,-0.12328125)
\psdots[dotsize=0.44](9.442187,3.5167189)
\usefont{T1}{ptm}{m}{n}
\rput(9.493594,-0.45328125){$0$}
\pscircle[linewidth=0.02,dimen=outer](9.502188,-0.44328126){0.32}
\usefont{T1}{ptm}{m}{n}
\rput(9.133594,3.0667188){$k$}
\usefont{T1}{ptm}{m}{n}
\rput(9.453594,1.6667187){$1$}
\pscircle[linewidth=0.02,dimen=outer](9.462188,1.6967187){0.32}
\psline[linewidth=0.02cm](9.442187,3.4167187)(9.442187,2.0167189)
\usefont{T1}{ptm}{m}{n}
\rput(9.479688,-3.5432813){\small $K_m[\overline{K_l}[K_k]]$}
\usefont{T1}{ptm}{m}{n}
\rput(9.163593,1.0867188){$l$}
\psline[linewidth=0.02cm](9.502188,-0.74328125)(9.562187,-2.2432814)
\usefont{T1}{ptm}{m}{n}
\rput(9.533594,-2.5732813){$1$}
\pscircle[linewidth=0.02,dimen=outer](9.542188,-2.5632813){0.32}
\usefont{T1}{ptm}{m}{n}
\rput(9.243594,-1.0732813){$m$}
\end{pspicture}
}
\end{center}
\caption{Some depth $3$ cotrees with branching. Non-leaf nodes labelled $0$ (disjoint union) or $1$ (join). Branching label indicates multiplicity of branching operation to be applied.}\label{fig:depth_3}
\end{figure}
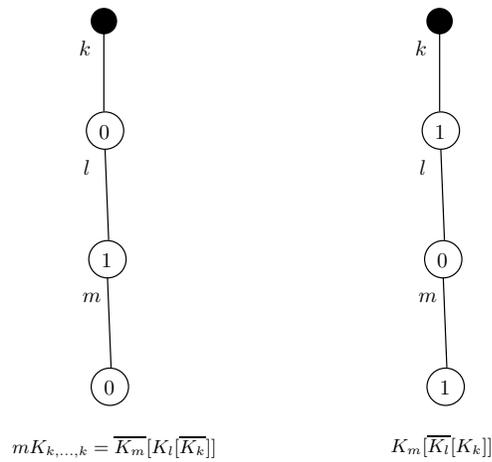

We start with a cograph $H$ given by its cotree representation, i.e., a rooted tree $T$ with labels $0$ or $1$ on non-leaf vertices and set of leaves equal to $V(H)$.
For each $\mathbf{k}\!=\!(k_s:s\!\in\!V(T))\in\mathbb{N}^{|V(T)|}$ we construct a cotree $T^{\mathbf{k}}$, in which, for each $s\in V(T)$, we create $k_s$ isomorphic copies (isomorphism including labels) of the subtree $T_s$ of $T$ rooted at $s$, all pendant from the same vertex as $T_s$. This cograph represents a larger cograph which we shall denote by $T^{\mathbf{k}}(H)$.

For example, if $H=K_1$, represented by cotree $T$ with root labelled $0$ and one leaf, then $T^{k}$ is the star $K_{1,k}$ rooted at its centre labelled $0$, and $T^k(H)=\overline{K_k}$. See Figures~\ref{fig:cograph_branching} and \ref{fig:depth_3} for further examples.   

\begin{thm}
The sequence $(T^{\mathbf{k}}(H))$ of cographs is strongly polynomial.
\end{thm}
\begin{proof} By induction on depth of $T$ (removing the root of a cotree gives a set of smaller cotrees). Without loss of generality we may assume the root is labelled $0$, representing a final disjoint union (we know that a sequence $(H_{\mathbf{k}})$ is strongly polynomial if and only if $(\overline{H_{\mathbf{k}}})$ is strongly polynomial). The property ${\rm hom}(G,H_1\cup H_2)={\rm hom}(G,H_1)+{\rm hom}(G,H_2)$ for connected $G$ (it is enough to prove polynomiality for connected $G$) allows the inductive step to go through.
\end{proof}


We also have an analogue of Theorem~\ref{thm:bounded_bushiness}. For a cograph $H$, define $\gamma(H)$ to be the minimum value of $|V(T)|$ such that $H$ is represented by the cotree $T^{\mathbf{k}}$. For example, $\gamma(K_{k})=2$, $\gamma(K_{k,l})=5$ ($k\neq l$), $\gamma(K_{k,k,\ldots, k})=3$. 

\begin{thm}
Let $\mathcal{H}$ be a family of cographs such that $\{\gamma(H):H\in \mathcal{H}\}$ is bounded.  Then $\mathcal{H}$ can be partitioned into a finite number of subsequences of strongly polynomial sequences of graphs. 
\end{thm}

\end{document}